\documentclass[letterpaper, oneside]{amsart}
\usepackage{layout}	

\usepackage[
]{geometry}

\usepackage[parfill]{parskip}   
\usepackage{amssymb}
\usepackage{amsthm}
\usepackage{amsmath}
\usepackage{cite}
\usepackage[all]{xy}
\usepackage{enumitem}
\usepackage{verbatim}
\usepackage{mathtools}
\usepackage{hyperref}
\usepackage{colonequals}
\usepackage{mathrsfs}
\usepackage{color}
\usepackage{subcaption}
\usepackage{bigfoot}
\usepackage{adjustbox}

\usepackage{tikz}
\usetikzlibrary{decorations.markings}
\usetikzlibrary{arrows.meta}
\usetikzlibrary{patterns}
\tikzset{
    triple/.style args={[#1] in [#2] in [#3]}{
        #1,preaction={preaction={draw,#3},draw,#2}
    }
}      

\usepackage{longtable}
\usepackage{multicol}
\usepackage{multirow}
\usepackage{thmtools}
\usepackage{tensor}

\theoremstyle{plain}
\newtheorem{thm}{Theorem}[section]
\newtheorem{mainthm}{Main Theorem}

\newtheorem{lem}[thm]{Lemma}
\newtheorem{prop}[thm]{Proposition}
\newtheorem{cor}[thm]{Corollary}

\theoremstyle{definition}

\newtheorem*{ack}{Acknowledgement}

\theoremstyle{remark}
\newtheorem*{rmk}{Remark}

\numberwithin{equation}{section}

\newcommand{\A}{\mathbb{A}}
\newcommand{\C}{\mathbb{C}}

\newcommand{\N}{\mathbb{N}}
\newcommand{\bP}{\mathbb{P}}
\newcommand{\Q}{\mathbb{Q}}

\newcommand{\Z}{\mathbb{Z}}

\newcommand{\B}{\mathcal{B}}

\newcommand{\cF}{\mathcal{F}}
\newcommand{\cP}{\mathcal{P}}

\newcommand{\V}{\mathcal{V}}

\newcommand{\cX}{\mathcal{X}}

\newcommand{\g}{\mathfrak{g}}

\newcommand{\ol}[1]{\overline{#1}}
\newcommand{\br}[1]{\left\langle{#1}\right\rangle}
\newcommand{\brt}[1]{\left\langle\!\left\langle{#1}\right\rangle\!\right\rangle}
\newcommand{\qlbar}{{\overline{\mathbb{Q}_\ell}}}
\newcommand{\floor}[1]{\left\lfloor#1\right\rfloor}

\newcommand{\pch}[2]{\downarrow^{\left(#1\right)}_{\left(#2\right)}}
\newcommand{\spn}[1]{\textbf{\textup{sp}}\left\{#1\right\}}
\newcommand{\wt}[1]{\widetilde{#1}}

\newcommand{\cq}[2]{{#2}^{(#1)}}
\newcommand{\h}[1]{\textup{ht}\left(#1\right)}

\newcommand{\tsp}[1]{\textup{\textbf{TSp}}\left(#1\right)}
\newcommand{\atsp}[1]{\widetilde{\textup{\textbf{TSp}}}\left(#1\right)}

\newcommand{\gue}[1]{\Upsilon_{#1}}
\newcommand{\ague}[1]{\widetilde{\Upsilon}_{#1}}
\newcommand{\cf}[1]{\mathcal{R}(#1)}

\newcommand{\sg}[1]{\varepsilon_{#1}}
\newcommand{\tsg}[1]{\epsilon_2\left({#1}\right)}
\newcommand{\diff}[1]{\Theta_{#1}}

\newcommand{\tyB}{\framebox[10pt]{$B$}\,}
\newcommand{\tyC}{\framebox[10pt]{$C$}\,}

\newcommand{\orangebox}[2]{
	\draw [line width=1mm, orange] (#1) rectangle (#2);
}
\newcommand{\bluefill}[2]{
	\fill [blue!40!white] (#1) rectangle (#2);
}
\newcommand{\yellowfill}[2]{
	\fill [yellow!40!white] (#1) rectangle (#2);
}
\newcommand{\hatch}[2]{
	\fill [pattern=dots, pattern color=blue!30!white] (#1) rectangle (#2);
}

\DeclareMathOperator{\gr}{\mathcal{Q}}
\DeclareMathOperator{\rk}{\textup{ss-rank}}
\DeclareMathOperator{\sgn}{\textup{sgn}}

\DeclareMathOperator{\Ind}{\textup{Ind}}
\DeclareMathOperator{\Res}{\textup{Res}}
\DeclareMathOperator{\tr}{\textup{tr}}

\DeclareMathOperator{\Ad}{\textup{Ad}}

\DeclareMathOperator{\ch}{\textmd{\textup{ch}}}

\DeclareMathOperator{\del}{\partial}
\DeclareMathOperator{\Lie}{\textup{Lie}}

\DeclareMathOperator{\genQ}{\mathbf{Q}}
\DeclareMathOperator{\genF}{\mathbf{F}}
\DeclareMathOperator{\sym}{\mathfrak{S}}
\DeclareMathOperator{\dummy}{\framebox{$\Q[t]\cdot Q$}}
\DeclareMathOperator{\dummyy}{\framebox{$\Q[t]\cdot Q^\prime$}}
\DeclareMathOperator{\Id}{\textup{Id}}
\DeclareMathOperator{\versch}{\mathbb{V}_2}

\title{On total Springer representations for classical types}
\author{Dongkwan Kim}
\address{Department of Mathematics\\
  Massachusetts Institute of Technology\\
  Cambridge, MA 02139-4307\\
  U.S.A.}
\email{sylvaner@math.mit.edu}
\date{\today}							

\begin{document}
\begin{abstract} We give explicit formulas on total Springer representations for classical types. We also describe the characters of restrictions of such representations to a maximal parabolic subgroup isomorphic to a symmetric group. As a result, we give closed formulas for the Euler characteristic of Springer fibers.
\end{abstract}

\setcounter{tocdepth}{1}
\maketitle

\renewcommand\contentsname{}
\tableofcontents


\section{Introduction}


This is a sequel to \cite{kim:euler}. There, we gave inductive formulas to calculate the Euler characteristic of Springer fibers for classical types. Here, we develop its method further and give explicit formulas on total Springer representations for classical types, i.e. the alternating sum of all the cohomology spaces of Springer fibers.

For type $A$, it is well-known that there is a strong connection between Kostka-Foulkes polynomials, Green polynomials, and total Springer representations. More precisely, modified Kostka-Foulkes polynomials (resp. Green polynomials) give the graded multiplicities of irreducible representations in (resp. the graded character values of) total Springer representations. As both polynomials can be computed by the theory of symmetric functions, they provide strong tools to study Springer theory combinatorially. For more information we refer readers to \cite{lus81}, \cite{mac95}, etc.

The graded characters of total Springer representations for other types are also called Green functions, and they play an important role in the representation theory of finite groups of Lie type. It was proved by Kazhdan \cite{kaz77} (under some restriction on the characteristic of the base field) that these functions are identical to the ones defined by Deligne and Lusztig \cite{delu76} using the $\ell$-adic cohomology of Deligne-Lusztig varieties. Later, Lusztig defined generalized Green functions using character sheaves \cite{luchar2}. His theory provides a way to explicitly calculate Green functions using their orthogonality property, which is now called Lusztig-Shoji algorithm \cite{sho83}, \cite{luchar5}, \cite{sho88}.

Here, we give an alternative way to describe total Springer representations for classical types. The main results of this paper, stated in Section \ref{sec:mainthm}, claim that if we forget the grading and the action of centralizers of nilpotent elements, the total Springer representations for type $B$, $C$, and $D$ can also be fully described by Kostka-Foulkes polynomials of type $A$. Also, using Green polynomials of type $A$, we give explicit formulas for the restriction of such representations to a maximal parabolic subgroup isomorphic to a symmetric group. This means, roughly speaking, that combinatorics for type $A$ also governs Springer theory for other classical types. This is quite surprising--- at least to the author.

As a corollary, we obtain closed formulas for the Euler characteristic of Springer fibers for classical types in terms of Green polynomials. To the best of the author's knowledge, there was no elementary way to calculate such numbers. The most common method is to use Lusztig-Shoji algorithm, but it depends on the theory of character sheaves and orthogonality of Green functions--- none of which are elementary. However, our methods only depend on some geometric properties of Springer fibers and the theory of symmetric functions, both of which are simpler. In addition, the previous paper \cite{kim:euler} showed that these numbers can be computed inductively, but it does not give a closed formula except in some cases. On the other hand, Kostka-Foulkes and Green polynomials can be easily obtained using combinatorics.

This paper is organized as follows: in Section \ref{sec:defnot} we recall the basic definitions and notations that are frequently used in this paper; in Section \ref{sec:param} we describe parametrizations of conjugacy classes in Weyl groups and nilpotent conjugacy classes in Lie algebras for classical types; in Section \ref{sec:parspr} we give parametrizations of irreducible representations of Weyl groups and their relation to Springer correspondence; in Section \ref{sec:mainthm} we state the main results of this paper; Sections \ref{sec:prelim}--\ref{sec:typeD} are devoted to investigating some geometric properties of Springer fibers and total Springer representations; in Section \ref{sec:symm} we recall the theory of symmetric functions and its relation to representation theory; in Section \ref{sec:greenkostka} we prove some formulas about Green and Kostka-Foulkes polynomials; and finally, in Sections \ref{sec:proofmain2}--\ref{sec:mainD} we prove the main theorems.

\begin{ack} The author wishes to thank George Lusztig for having stimulating discussions with him and checking the draft of this paper. He is grateful to Jim Humphreys for his detailed remarks which help improve the readability of this paper. Also he thanks Gus Lonergan and Toshiaki Shoji for useful comments.
\end{ack}

\section{Definitions and notations} \label{sec:defnot}
\subsection{Setup}
Let $\C$ be the field of complex numbers and $G$ be a reductive group whose derived group is simple of classical type over $\C$. (Mostly $G$ is one of $GL_n(\C), SO_n(\C), $ or $Sp_n(\C)$ in this paper.) Define $\g\colonequals \Lie G$ to be the Lie algebra of $G$. We denote the semisimple rank of $G$ by $\rk G$. Define $B\subset G$ to be a Borel subgroup of $G$ and $T\subset B$ to be a maximal torus. Also we let $W$ be the Weyl group of $G$, defined by the normalizer of $T$ in $G$ modulo $T$. To avoid ambiguity, we sometimes write $W(X)$ for the Weyl group of type $X$. Thus in particular $W(A_{n-1})$ is isomorphic to the symmetric group permuting $n$ elements, denoted $\sym_n$.

\subsection{Partitions}
For a partition $\lambda$, we write either $\lambda = (\lambda_1, \lambda_2, \cdots)$ for some $\lambda_1 \geq \lambda_2 \geq \cdots$ or $\lambda=(1^{m_1}2^{m_2} \cdots)$ to describe its parts. We write $\lambda \vdash n$ if $\sum_i \lambda_i = n$ or equivalently $\sum_i im_i = n$. In this case we also write $|\lambda| = n$. We let $l(\lambda)$ be the length of $\lambda$, i.e. the number of nonzero parts in $\lambda$ or equivalently $\sum_{i} m_i$. We denote by $\lambda'$ the conjugate of $\lambda$. We often identify partitions with Young diagrams. In this paper we adopt the English notation, thus each part in a partition corresponds to each row in the corresponding Young diagram, and the size of $i$-th row is $\lambda_i$.

We say that $\lambda$ is even (resp. odd) if every part in $\lambda$ is even (resp. odd), or equivalently $m_i = 0$ for $i$ odd (resp. even). In particular, $\lambda$ is called very even if it is even and with even multiplicity, i.e. $m_i \in 2\N$ for any $i$. We say that $\lambda$ is strict if parts of $\lambda$ are pairwise different, that is for $i\neq j$ such that $\lambda_i, \lambda_j\neq 0$ we have $\lambda_i \neq \lambda_j$.

For partitions $\lambda=(\lambda_1, \lambda_2, \cdots)$ and $\mu=(\mu_1, \mu_2, \cdots)$, we define $\lambda\cup \mu$ be the union of two partitions, i.e. the partition of $|\lambda|+|\mu|$ whose parts are $\{\lambda_1, \lambda_2, \cdots\} \cup \{\mu_1, \mu_2, \cdots\}$ as a multiset. Also for $k \in \N$ we define $k\lambda\colonequals (k\lambda_1, k\lambda_2, \cdots)$. For partitions $\lambda, \mu$, we write $\lambda \supset \mu$ if $\lambda_i \geq \mu_i$ for all $i$, i.e. if the Young diagram of $\lambda$ contains that of $\mu$. In this case, we write $\lambda/\mu$ for the corresponding skew partition. Also we write $\lambda \geq \mu$ if $\sum_{i=1}^r \lambda_i \geq \sum_{i=1}^r \mu_i$ for any $r \geq 1$.

For $\{a_1, a_2, \cdots \} \subset \{\lambda_1, \lambda_2, \cdots\}$ as a multiset, we define $\lambda\pch{a_1, a_2, \cdots}{b_1, b_2, \cdots}$ to be the partition obtained by replacing $a_1, a_2, \cdots$ with $b_1, b_2, \cdots$ and rearranging rows if necessary. For example, $(6, 6, 3)\pch{6}{2}= (6,3,2)$. Note that $\lambda^{h,i}, \lambda^{v,i}$ in \cite{kim:euler} can now be written as $\lambda\pch{i}{i-2}, \lambda\pch{i,i}{i-1,i-1},$ respectively.

Let $b(\lambda) \colonequals \sum_{i\geq 1}(i-1)\lambda_i = \sum_{j\geq 1} \binom{\lambda'_j}{2}$ be the weighted size of $\lambda$. We also let $z_\lambda\colonequals \prod_{i\geq 1} i^{m_i}m_i!$, the size of the centralizer of an element of cycle type $\lambda$ in $\sym_{|\lambda|}$. For partitions $\lambda, \mu$ such that $\mu \subset \lambda$, we define $\h{\lambda/\mu}$ to be the height of the skew-partition $\lambda/\mu$, which is given by $\#\{i \geq 1 \mid \lambda_i >\mu_i \}-1.$ For example, $\h{(1^r)}= r-1$ for $r \geq 1$. 

\subsection{2-core, 2-quotient, and sign of partitions} \label{sec:corequot} For a partition $\lambda$, the 2-core of $\lambda$, denoted $\cq{c}{\lambda}$, is obtained from $\lambda$ by removing as many 2-border strips (i.e. dominos) as possible. Note that $\cq{c}{\lambda}$ does not depend on the order of removing 2-border strips from $\lambda$. For $|\lambda|$ even (resp. odd), we say the 2-core of $\lambda$ is minimal if $\cq{c}{\lambda}= \emptyset$ (resp. $\cq{c}{\lambda}= (1)$).

In order to define the 2-quotient of $\lambda$, we first choose $m \geq l(\lambda)$ and perform the following.
\begin{enumerate}  
\item For $1 \leq i \leq m$, set $a_i \colonequals \lambda_i+m-i$. In particular, $a_1>a_2 >\cdots>a_m$. (Note that $a_m$ is zero if $m>l(\lambda)$.)
\item Separate even and odd integers from $a_1, \cdots, a_m$, say $2b_1> 2b_2 > \cdots > 2b_s$ and $2c_1+1> 2c_2+1>  \cdots >  2c_t+1$, where $s+t=m$. 
\item Define $\cq{0}{\lambda}\colonequals(b_1-s+1, \cdots, b_{s-1}-1, b_s)$ and $\cq{1}{\lambda}\colonequals(c_1-t+1, \cdots, c_{t-1}-1, c_t)$.
\end{enumerate}
We call $(\cq{0}{\lambda}, \cq{1}{\lambda})$ the 2-quotient of $\lambda$. Note that this definition depends on the choice of (the parity of) $m$; indeed, choosing $m \geq l(\lambda)$ of different parity results in swapping $\cq{0}{\lambda}$ and $\cq{1}{\lambda}$. In this article we always choose $m$ to be even. Thus for example, if $\lambda=(10)$ then we may choose $m=2$ and get $(\cq{0}{\lambda},\cq{1}{\lambda})=(\emptyset, (5))$.

We define $\sg{\lambda} \colonequals (-1)^{b(\lambda)} = (-1)^{\sum_{i\geq 1} \lambda_{2i}}$ to be the sign of $\lambda$. Also for $\mu \subset \lambda$, we let $\tsg{\lambda/\mu}$ be the 2-sign of $\lambda/\mu$, which is defined as follows. If it is impossible to fill the Young diagram of $\lambda/\mu$ with dominoes, then set $\epsilon_2(\lambda/\mu)=0.$ Otherwise, choose any domino filling of $\lambda/\mu$ and set $\epsilon_2(\lambda/\mu)$ to be 1 (resp. -1) if the number of vertical dominoes in this filling (with respect to the English notation) is even (resp. odd). Note that this does not depend on the choice of domino fillings. Also for $\mu \subset \lambda$, it is easy to show that
$$\sg{\lambda}=\sg{\mu}\tsg{\lambda/\mu}, \textup{ thus in particular, } \sg{\lambda}=\sg{\cq{c}{\lambda}}\tsg{\lambda/\cq{c}{\lambda}}.$$


\subsection{$\ell$-adic cohomology}
For a variety $X$ over $\C$, we define $Sh(X)$ to be the category of constructible $\ell$-adic sheaves on $X$. Also define $H^*(X) \colonequals \bigoplus_{i \in \Z} (-1)^i H^i(X)$ to be the alternating sum of $\ell$-adic cohomology spaces of $X$ as a virtual vector space. We define $\cX(X) \colonequals \dim H^*(X) = \sum_{i \in \Z} (-1)^i \dim H^i(X)$ to be the Euler characteristic of $X$. Also we define $H^{top}(X)$ to be $H^{2\dim X}(X)$.

\subsection{Springer fibers}
Let $\B \colonequals G/B$ be the flag variety of $G$, and for $N \in \g$ define $\B_N$ to be the Springer fiber of $N$, i.e.
$$\B_N \colonequals \{gB \in \B \mid \Ad(g)^{-1}N \in \Lie B\}.$$
Then the $\ell$-adic cohomology of $\B_N$ is naturally equipped with the action of $W$, called the Springer representation, originally defined by \cite{spr76}. Here we adopt the convention of \cite{lus81}, which differs from the former one by tensoring with the sign representation of $W$. We also call $H^*(\B_N)$, as a representation of $W$, the total Springer representation corresponding to $N$. As Springer fibers have vanishing odd cohomology \cite{dclp}, in fact it is an actual representation of $W$.

\subsection{Characters of a finite group} \label{subsec:char}
For a finite group $\mathcal{G}$, we define $id \in \mathcal{G}$ to be the identity element of $\mathcal{G}$ and also define $\Id_\mathcal{G}$ to be the trivial representation of $\mathcal{G}$. Also for a Weyl group $W$, we define $\sgn_W$ to be the sign representation of $W$.

Let $\cf{\mathcal{G}}$ be the ring of $\Q$-valued class functions of $\mathcal{G}$. Also for a $\Q$-algebra $A$, we write $\cf{\mathcal{G}}_A \colonequals \cf{\mathcal{G}}\otimes_\Q A$. Then there exists a homomorphism from the Grothendieck group of the category of finite dimensional complex representations of $\mathcal{G}$ to $\cf{\mathcal{G}}_\C$ which sends a (virtual) representation of $\mathcal{G}$ to its trace function, i.e. its character. We denote such homomorphism by $\ch$. In other words, for any (virtual) representation $M$ and for any $g\in \mathcal{G}$ we have $\ch M(g) = \tr(g, M).$
If $\mathcal{G}$ is the Weyl group $W$ of some reductive group, then the image of $\ch$ actually lands in $\cf{W}$ since all the complex irreducible representations of $W$ are defined over $\Q$.

There exists a usual scalar product $\br{\ , \ }: \cf{\mathcal{G}}_\C \times \cf{\mathcal{G}}_\C \mapsto \C : (f,g) \mapsto \frac{1}{\#\mathcal{G}}\sum_{x\in \mathcal{G}} f(x)\overline{g(x)}$. For any virtual representations $M,N$ of $\mathcal{G}$, we also define $\br{M,N} \colonequals \br{\ch M, \ch N}$. For a subgroup $\mathcal{H} \subset \mathcal{G}$, we define $\Res_\mathcal{H}^\mathcal{G}$ and $\Ind_\mathcal{H}^\mathcal{G}$ to be the usual restriction and induction functor. They are both well-defined on finite dimensional complex representations and class functions. Also, they are adjoint to each other with respect to the scalar product defined above.

For $g \in \mathcal{G}$ which commutes with all elements in $\mathcal{H}$, we define $\Res_{g\cdot\mathcal{H}}^\mathcal{G}$ to be the restriction of a class function of $\mathcal{G}$ to $g\mathcal{H}=\mathcal{H}g$ (the coset of $\mathcal{H}$), considered as a class function of $\mathcal{H}$ under the identification $\mathcal{H} \xrightarrow{\simeq} g\mathcal{H} : h \mapsto gh$. (Thus here it is necessary to specify $g$ in the notation.) Note that even when we choose $M$ to be an actual representation of $\mathcal{G}$, $\Res^\mathcal{G}_{g \cdot \mathcal{H}}\ch M$ need not be a character of an actual one of $\mathcal{H}$, e.g. $\Res^{\Z/2}_{-1\cdot \{id\}}\ch\sgn_{\Z/2}$ is a virtual representation of the trivial group.

\subsection{Miscellaneous}
For a vector space $V$ and a subset $X \subset V$, we define $\spn{X}$ to be the span of $X$ in $V$.

\section{Conjugacy classes in Weyl groups}\label{sec:param}
Here we recall the parametrizations of conjugacy classes in Weyl groups and nilpotent conjugacy classes in Lie algebras for classical types.

\subsection{Type $A$} Let $G=GL_n(\C)$. Then its Weyl group is isomorphic to $\sym_n$. Its conjugacy classes are parametrized by partitions of $n$; to each conjugacy class we associate its cycle type. For $\rho \vdash n$, we write $w_\rho \in W$ to be the element of cycle type $\rho$, which is well-defined up to conjugation.

Also, there is a natural bijection between the set of nilpotent conjugacy classes in $\g$ and the set of partitions of $n$; to each conjugacy class we associate its Jordan type. For $\lambda \vdash n$, we write $N_\lambda \in \g$ to be the nilpotent element of Jordan type $\lambda$, which is again well-defined up to conjugation. For $N=N_\lambda$, we define $\tsp{\lambda}\colonequals H^*(\B_{N})$  as a $W$-module.

\subsection{Type $B$ and $C$} Let $G=SO_{2n+1}(\C)$ or $G=Sp_{2n}(\C)$. Then its Weyl group is isomorphic to $\sym_n \ltimes (\Z/2)^n.$ Thus its conjugacy classes are parametrized by the set of ordered pairs of partitions $(\rho, \sigma)$ such that $|\rho|+|\sigma|=n$, where $\rho$ and $\sigma$ correspond to $1, -1 \in \Z/2$, respectively. (ref. \cite[pp.170--171, 178]{mac95}) For such pair $(\rho, \sigma)$, we write $w_{(\rho, \sigma)} \in W$ to be the element which is contained in the conjugacy class parametrized by $(\rho, \sigma)$. It is clearly well-defined up to conjugation.

If $G$ is of type $B$, by taking the Jordan type we have a bijection between nilpotent conjugacy classes in $\g$ and the set of partitions $\lambda \vdash 2n+1$ each of whose even parts has even multiplicity. Similarly, if $G$ is of type $C$, by taking the Jordan type we have a bijection between nilpotent conjugacy classes in $\g$ and the set of partitions $\lambda \vdash 2n$ each of whose odd parts has even multiplicity. We similarly write $N_\lambda \in \g$ for such $\lambda$ to denote the nilpotent element of Jordan type $\lambda$, which is well-defined up to conjugation. We define $\tsp{\lambda}$ similarly to type $A$.

%

\subsection{Type $D$} \label{subsec:typeDparam} Let $G=SO_{2n}(\C)$. Then its Weyl group can be regarded as a subgroup of $\sym_n \ltimes (\Z/2)^n$ of order 2. Then a conjugacy class of $\sym_n \ltimes (\Z/2)^n$ parametrized by $(\rho, \sigma)$ has a nontrivial intersection with $W$ if and only if $l(\sigma)$ is even. If so, the intersection is again a conjugacy class in $W$ except when $\sigma=\emptyset$ and $\rho$ is even, in which case it splits into two conjugacy classes. Thus there exists a parametrization of the conjugacy classes of $W$ by the set
$$\{ (\rho, \sigma) \mid |\rho|+|\sigma|=n, l(\sigma) \in 2\N, \rho \textup{ is not even if } \sigma=\emptyset\} \cup \{ (\rho, \emptyset)\pm  \mid \rho \vdash n, \rho \textup{ is even})  \}.$$
Here $(\rho, \emptyset)+$ and $(\rho, \emptyset)-$ correspond to two conjugacy classes in $W$ contained in the conjugacy class of $\sym_n \ltimes (\Z/2)^n$ parametrized by $(\rho, \emptyset)$. We write $w_{(\rho, \sigma)}, w_{(\rho, \emptyset)\pm}$ to denote the corresponding elements in $W$ which are well-defined up to conjugation. 

For a partition $\lambda \vdash 2n$ each of whose even parts has even multiplicity, there exists a unique nilpotent conjugacy class in $\g$ whose Jordan type is $\lambda$ unless $\lambda$ is very even; if so, then there exist two such classes. Thus the set of nilpotent conjugacy classes in $\g$ is parametrized by
$$\{ \lambda=(1^{m_1}2^{m_2} \cdots) \vdash 2n \mid m_{2i} \in 2\N \textup{ for } i \in \Z_{\geq 1}, \lambda \textup{ is not very even}\} \cup \{ \lambda\pm \mid \lambda \vdash 2n \textup{ is very even}\}.$$
We write $N_{\lambda}, N_{\lambda\pm}$ to denote the corresponding nilpotent elements in $\g$ which are well-defined up to conjugation. If $N=N_\lambda$ for $\lambda$ not very even, we define $\tsp{\lambda}$ similarly to type $A$ and also put $\atsp{\lambda}\colonequals \tsp{\lambda}\oplus \tsp{\lambda}$. If $\lambda$ is very even, we define $\tsp{\lambda\pm} = H^*(\B_{N})$ for $N=N_{\lambda\pm}$ and put $\atsp{\lambda}\colonequals \tsp{\lambda+}\oplus \tsp{\lambda-}$.


In order to remove ambiguity from the choice of the sign $\pm$, we proceed as follows. We construct embeddings
$$SO_4(\C) \hookrightarrow SO_6(\C) \hookrightarrow SO_8(\C) \hookrightarrow \cdots \hookrightarrow SO_{2n}(\C) \hookrightarrow SO_{2n+2}(\C) \hookrightarrow \cdots$$
which induces embeddings of Weyl groups
$$W(D_2) \hookrightarrow W(D_3) \hookrightarrow W(D_4) \hookrightarrow \cdots \hookrightarrow W(D_n) \hookrightarrow W(D_{n+1}) \hookrightarrow \cdots$$
such that for any $i<j$, the image of $W(D_i) \hookrightarrow W(D_j)$ is a parabolic subgroup of $W(D_j)$. Thus the set of simple reflections of $W(D_i)$ can be regarded as a subset of that of $W(D_j)$ for $i<j$. Here we also require that the simple reflections of $W(D_2)$ correspond to two extremal points in the Dynkin diagram of $D_n$ for any $n\geq 2$.

There exists an involution $\tau_n$ on $SO_{2n}(\C)$ which gives the nontrivial involution on the Dynkin diagram and also on $W(D_n)$. We may choose $\tau_n$ for each $n \geq 2$ such that for any $m <n $ we have $\tau_n|_{SO_{2m}(\C)} = \tau_m$ under the embeddings above. Thus we drop the subscript and simply write $\tau$ to be this nontrivial involution on any $SO_{2n}(\C)$. (Even when $n=4$, there is no ambiguity since the embedding $SO_{8}(\C) \hookrightarrow SO_{10}(\C)$ completely determines $\tau_4$.) 

Now once and for all, we fix the labeling $s_+, s_-$ of two simple reflections in $W(D_2) \simeq \Z/2 \times \Z/2$, and regard them as simple reflections in $W(D_n)$ for any $n \in \Z_{\geq 2}$. We write $\sym_{n+}$ (resp. $\sym_{n-}$) to be the parabolic subgroup of $W$ generated by all the simple reflections but $s_-$ (resp. $s_+$). Similarly, we write $L_+$ (resp. $L_-$) to be the Levi subgroup (which contains $T$) of a parabolic subgroup (which contains $B$) whose Weyl group is $\sym_{n+}$ (resp. $\sym_{n-}$).

For a conjugacy class which corresponds to either $(\rho, \emptyset)+$ or $(\rho, \emptyset)-$ for some $\rho$ even, we say that it is parametrized by $(\rho, \emptyset)+$ (resp. $(\rho, \emptyset)-$) if its intersection with $\sym_{n+}$ (resp. $\sym_{n-}$) is nonempty. Also, for a nilpotent conjugacy class whose Jordan type is $\lambda \vdash 2n$ very even, we say that it is parametrized by $\lambda+$ (resp. $\lambda-$) if its intersection with the Lie algebra of $L_+$ (resp. $L_-$) is nonempty. Then one can show that these notions are indeed well-defined. Also note that $\tau$ swaps $+$ and $-$ in both parametrizations.

These parametrizations have a following advantage. For $\lambda$ very even, by \cite{lus04} $\tsp{\lambda+}$ (resp. $\tsp{\lambda-}$) is an induced character from a parabolic subgroup contained in $\sym_{n+}$ (resp. $\sym_{n-}$). As conjugacy classes in $W$ parametrized by $(\rho, \emptyset)+$ (resp. $(\rho, \emptyset)-$) does not intersect with $\sym_{n-}$ (resp. $\sym_{n+}$), the following observation is a direct consequence of Mackey's formula.
\begin{lem} \label{lem:pm} Let $G=SO_{2n}(\C)$. For $\rho \vdash n$ even and $\lambda \vdash 2n$ very even, we have
$$\ch \tsp{\lambda-}(w_{(\rho, \emptyset)+})=\ch  \tsp{\lambda+}(w_{(\rho, \emptyset)-})=0.$$
\end{lem}

\section{Irreducible representations of Weyl groups and the Springer correspondence} \label{sec:parspr}
Here we recall the parametrizations of irreducible representations of Weyl groups for classical types and how they are related to nilpotent orbits in $\g$ under the Springer correspondence. One may refer to \cite{sho79}, \cite{lus79}, \cite{car93} for more information.

\subsection{Type $A$} Let $W=\sym_n$ be the Weyl group of type $A_{n-1}$. The irreducible representations of $W$ (up to isomorphism) are parametrized by partitions of $n$. For $\lambda \vdash n$, we write $\chi^\lambda$ to denote such representation. Thus for example $\chi^{(n)}=\Id_W$ and $\chi^{(1^n)}=\sgn_{\sym_n}$. Then the Springer correspondence is simply given by $H^{top}(\B_{N}) \simeq \chi^\lambda$ for $N=N_\lambda$.

\subsection{Type $B$ and $C$} Let $W=\sym_n \ltimes (\Z/2)^n$ be the Weyl group of type $BC_n$. The irreducible representations of $W$ (up to isomorphism) are parametrized by ordered pairs of partitions $(\lambda, \mu)$ such that $|\lambda|+|\mu|=n$. For such $(\lambda,\mu)$, we write $\chi^{(\lambda,\mu)}$ to denote such representation. Thus for example $\chi^{((n),\emptyset)}=\Id_W$ and $\chi^{(\emptyset,(1^n))}=\sgn_W$. 

If $G$ is of type $B$, then for a partition $\lambda \vdash 2n+1$ with minimal 2-core, we define $\chi^\lambda \colonequals \chi^{(\cq{0}\lambda,\cq{1}\lambda)}$. If $G$ is of type $C$, then for a partition $\lambda \vdash 2n$ with minimal 2-core, we also define $\chi^\lambda \colonequals \chi^{(\cq{1}\lambda,\cq{0}\lambda)}$. (Note the difference from type $B$.) If its 2-core is not minimal, then we simply put $\chi^\lambda=0$. Then the Springer correspondence is given by $H^{top}(\B_{N})^{Z_G(N)} \simeq \chi^{\lambda}$ for $N=N_\lambda$. Here, $Z_G(N)$ is the stabilizer of $N$ in $G$.



\subsection{Type $D$} Let $W$ be the Weyl group of type $D_n$. The irreducible representations of $W$ (up to isomorphism) are parametrized by unordered pairs of partitions $\{\lambda, \mu\}$ such that $|\lambda|+|\mu|=n$ except when $\lambda=\mu$, in which case there are two irreducible representations corresponding to $\{\lambda, \lambda\}$. For such $(\lambda,\mu)$ with $\lambda\neq \mu$,, we write $\chi^{\{\lambda,\mu\}}$ to denote such representation. Thus for example $\chi^{\{(n),\emptyset\}}=\Id_W$ and $\chi^{\{\emptyset,(1^n)\}}=\chi^{\{(1^n), \emptyset\}}=\sgn_W$.

In the case when $\lambda=\mu$, we write $\chi^{\{\lambda, \lambda\}+}, \chi^{\{\lambda, \lambda\}-}$ to denote such two representations corresponding to $\{\lambda,\lambda\}$. The choice of the sign is characterized by the property that $\br{\Ind_{\sym_{n\pm}}^W \Id_{\sym_{n\pm}}, \chi^{\{\lambda, \lambda\}\mp}}=0.$ Here $\sym_{n+}$ (resp. $\sym_{n-}$) is the parabolic subgroup of $W$ generated by all the simple reflections but $s_-$ (resp. $s_+$), see Section \ref{subsec:typeDparam}. 

For a partition $\lambda \vdash 2n$ with minimal 2-core, if $\lambda$ is not very even then $\cq{0}\lambda\neq \cq{1}\lambda$ and we define $\chi^\lambda \colonequals \chi^{\{\cq{0}{\lambda},\cq{1}{\lambda}\}}$. If $\lambda$ is very even, then $\cq{0}{\lambda}=\cq{1}{\lambda}$ and we define $\chi^{\lambda\pm} \colonequals \chi^{\{\cq{0}{\lambda},\cq{0}{\lambda}\}\pm}$. Then the Springer correspondence is described as follows. For $\lambda$ not very even, we have $H^{top}(\B_{N})^{Z_G(N)} \simeq \chi^{\lambda}$ for $N=N_\lambda$. Here, again $Z_G(N)$ is the stabilizer of $N$ in $G$. If $\lambda$ is very even, then we have $H^{top}(\B_{N}) \simeq \chi^{\lambda\pm}$ for $N=N_{\lambda\pm}$. (In this case $Z_G(N)$ is connected and thus $H^{top}(\B_{N})^{Z_G(N)} = H^{top}(\B_{N})$.)

\section{Main theorems} \label{sec:mainthm}
We are ready to state the main results of this papar.
\begin{mainthm} \label{mainthm1}For partitions $\lambda$ and $\mu$ such that $|\lambda|=|\mu|$, let $K_{\mu, \lambda}(t)$ be the Kostka-Foulkes polynomial, ref. \cite[Chapter III.6]{mac95}.
\begin{enumerate} 
\item Let $G=SO_{2n+1}(\C)$ or $G=Sp_{2n}(\C)$ and $\lambda$ be a Jordan type of some nilpotent element in $\g$. Then 
$$\ch \tsp{\lambda} = \sum_{\mu\vdash |\lambda|} \sg{\lambda}\sg{\mu}K_{\mu,\lambda}(-1)\ch\chi^\mu.$$
\item Let $G=SO_{2n}(\C)$ and $\lambda \vdash 2n$ be a Jordan type of some nilpotent element in $\g$ whicn is not very even. Then (note that $\sg{\mu}=1$ if $\mu$ is very even)
$$
\ch \tsp{\lambda}=\frac{1}{2}\sum_{\substack{\mu \vdash 2n\\\textup{not very even}}}\sg{\lambda}\sg{\mu}K_{\mu,\lambda}(-1)\ch\chi^\mu+\frac{1}{2}\sum_{\substack{\mu \vdash 2n\\\textup{very even}}}\sg{\lambda}K_{\mu,\lambda}(-1)\ch(\chi^{\mu+}\oplus \chi^{\mu-}).$$
If $\lambda$ is very even, then set $\tilde{\lambda} \vdash n/2$ to be such that $\lambda = 2\tilde{\lambda} \cup 2\tilde{\lambda}$. Define $\tilde{\mu}\vdash n/2$ similarly if $\mu$ is very even. Then, (note that $\sg{\lambda}=1$, and also $\sg{\mu}=1$ if $\mu$ is very even)
\begin{align*}
\ch \tsp{\lambda+} =&\frac{1}{2}\sum_{\substack{\mu \vdash 2n\\\textup{not very even}}}\sg{\mu}K_{\mu,\lambda}(-1)\ch\chi^\mu
\\&+\sum_{\substack{\mu \vdash 2n\\\textup{very even}}}\left(\frac{K_{\mu,\lambda}(-1)+K_{\tilde{\mu},\tilde{\lambda}}}{2}\ch\chi^{\mu+}+\frac{K_{\mu,\lambda}(-1)-K_{\tilde{\mu},\tilde{\lambda}}}{2}\ch\chi^{\mu-}\right), 
\\\ch \tsp{\lambda-} =&\frac{1}{2}\sum_{\substack{\mu \vdash 2n\\\textup{not very even}}}\sg{\mu}K_{\mu,\lambda}(-1)\ch\chi^\mu
\\&+\sum_{\substack{\mu \vdash 2n\\\textup{very even}}}\left(\frac{K_{\mu,\lambda}(-1)-K_{\tilde{\mu},\tilde{\lambda}}}{2}\ch\chi^{\mu+}+\frac{K_{\mu,\lambda}(-1)+K_{\tilde{\mu},\tilde{\lambda}}}{2}\ch\chi^{\mu-}\right).
\end{align*}
Here $K_{\tilde{\mu},\tilde{\lambda}} = K_{\tilde{\mu},\tilde{\lambda}}(1)$ is the Kostka number.
\end{enumerate}
\end{mainthm}
\begin{rmk} For $G=GL_n(\C)$ there is a graded analogue
$$\sum_{i \in \Z} (-1)^it^{i/2}\ch H^{i}(\B_{N_\lambda}) = \sum_{\mu\vdash n} t^{b(\lambda)}K_{\mu,\lambda}(t^{-1}) \ch \chi^\mu.$$
\end{rmk}
\begin{rmk} One needs to be careful when trying to calculate $\br{\tsp{\lambda}, \chi^\mu}$ for type $D$. If $\mu$ is not very even, then there exists $\nu \vdash |\lambda|$ such that $\mu\neq \nu, \cq{0}{\mu}=\cq{1}{\nu},$ and $\cq{1}{\mu}=\cq{0}{\nu}$. Thus we have
$$\br{\tsp{\lambda}, \chi^\mu} =\br{\tsp{\lambda}, \chi^\nu}= \frac{1}{2} \sg{\lambda}\left(\sg{\mu}K_{\mu,\lambda}(-1)+\sg{\nu}K_{\nu,\lambda}(-1)\right).$$
Indeed, the author conjectures that $\sg{\mu}K_{\mu,\lambda}(-1)=\sg{\nu}K_{\nu,\lambda}(-1)$. Thus it is likely that
$$\br{\tsp{\lambda}, \chi^\mu} =\br{\tsp{\lambda}, \chi^\nu}=  \sg{\lambda}\sg{\mu}K_{\mu,\lambda}(-1)=\sg{\lambda}\sg{\nu}K_{\nu,\lambda}(-1).$$
\end{rmk}
The following theorem may be regarded as a corollary of Main Theorem \ref{mainthm1}.
\begin{mainthm} \label{mainthm2}For partitions $\lambda$ and $\rho$ such that $|\lambda|=|\rho|$, let $\gr^\lambda_\rho(t)$ be the Green polynomial defined in \cite{gre55}. 
\begin{enumerate}
\item Let $G=SO_{2n+1}(\C)$ and $\lambda \vdash 2n+1$ be a Jordan type of some nilpotent element in $\g$. Then 
$$\ch \tsp{\lambda}(w_{(\rho,\emptyset)}) = \gr^\lambda_{2\rho\cup(1)}(-1), \qquad \textup{ thus }\qquad \cX(\B_{N_\lambda}) = \gr^\lambda_{(1^12^n)}(-1).$$
\item Let $G=Sp_{2n}(\C)$ and $\lambda \vdash 2n$ be a Jordan type of some nilpotent element in $\g$. Then 
$$\ch \tsp{\lambda}(w_{(\rho,\emptyset)}) = \gr^\lambda_{2\rho}(-1), \qquad \textup{ thus }\qquad \cX(\B_{N_\lambda})  = \gr^\lambda_{(2^n)}(-1).$$
\item Let $G=SO_{2n}(\C)$ and $\lambda \vdash 2n$ be a Jordan type of some nilpotent element in $\g$. If $\lambda$ is not very even and $\rho$ is not even, then
$$\ch \tsp{\lambda}(w_{(\rho,\emptyset)})= \frac{1}{2}\gr^\lambda_{2\rho}(-1), \qquad \textup{ thus }\qquad  \cX(\B_{N_\lambda}) = \frac{1}{2}\gr^\lambda_{(2^n)}(-1).$$
If $\lambda$ is not very even and $\rho$ is even, then 
$$\ch \tsp{\lambda}(w_{(\rho,\emptyset)\pm})= \frac{1}{2}\gr^\lambda_{2\rho}(-1).$$
If $\lambda$ is very even and $\rho \vdash n$ is not even, then
$$\ch \tsp{\lambda\pm}(w_{(\rho,\emptyset)}) = \frac{1}{2}\gr^\lambda_{2\rho}(-1), \qquad \textup{ thus }\qquad  \cX(\B_{N_{\lambda\pm}}) = \frac{1}{2}\gr^\lambda_{(2^n)}(-1).$$
If $\lambda$ is very even and $\rho$ is even, then
$$\ch \tsp{\lambda\pm}(w_{(\rho,\emptyset)\pm}) = \gr^\lambda_{2\rho}(-1), \qquad \textup{ and} \qquad \ch\tsp{\lambda\mp}(w_{(\rho,\emptyset)\pm}) = 0.$$
\end{enumerate}
\end{mainthm}
\begin{rmk} There is also a graded analogue for $G=GL_n(\C)$ 
$$\sum_{i \in \Z} (-1)^i t^{i/2}\ch H^i(\B_{N_\lambda})(w_\rho) = \gr^\lambda_\rho(t), \qquad \textup{ thus }\qquad  \chi (\B_{N_\lambda}) = \gr^\lambda_{(1^n)}(1)$$
that was first proved in \cite{spr71}.
\end{rmk}

\section{Fibration, stratification and Springer representation} \label{sec:prelim}
We first recall \cite[Theorem 5.1]{hs:green}. For a parabolic subgroup $P \supset B$ of $G$, let $P=L\cdot U_P$ be its Levi decomposition such that $T \subset L$. We also let $\cP \colonequals G/P$ be the partial flag variety of $G$ corresponding to $P$. For a nilpotent element $N \in \g$, we define 
$$\cP_N \colonequals \{gP \in \cP \mid \Ad(g)^{-1}N \in \Lie P\}.$$
Then there exists a natural morphism $\pi_P : \B_N \rightarrow \cP_N: gB \mapsto gP$. We define $W_P\subset W$ to be the Weyl group of $P$ considered as a parabolic subgroup of $W$.

\begin{lem} \label{lem:hs} Let $\qlbar_{\B_N}$ be the constant $\qlbar$-sheaf on $\B_N$. Then,
\begin{enumerate}[label={(\arabic*)}]
\item $R^j \pi_! (\qlbar_{\B_N}) \in Sh(\cP_N)$ for $j \in \Z$ has a structure of sheaves of $W_P$-modules.
\item Let $M \colonequals \Ad(g)^{-1}N|_{\Lie L}$ be the image of $\Ad(g)^{-1}N \subset \Lie P$ under the projection $\Lie P \twoheadrightarrow \Lie P / \Lie U_P \simeq \Lie L$. Then for any $gP \in \cP_N$, $R^j \pi_! (\qlbar_{\B_N})|_{gP} \simeq H^j(\B(L)_M)$ as $W_P$-modules. Here, $\B(L) \colonequals L/L\cap B$ is the flag variety of $L$, and the $W_P$-module structure on $H^j(\B(L)_{M})$ is defined as the Springer representation associated to $L$.
\item $H^i(\cP_N, R^j \pi_! (\qlbar_{\B_N})) \Rightarrow H^{i+j}(\B_N)$ is a spectral sequence of $W_P$-modules where the $W_P$-module structure on $H^{i+j}(\B_N)$ is given by the restriction of the Springer representation of $W$ to $W_P$.
\end{enumerate}
\end{lem}

Suppose $P \subset G$ is a parabolic subgroup of $G$ such that $W_P \simeq \sym_{k} \times W'$ for some $1\leq k \leq \rk G$ where $W'$ is a parabolic subgroup of the same type as $W$. Also suppose that there exists a stratification $\cP_N = \sqcup_{\alpha} X_\alpha$ such that on each $X_\alpha$, $H^*(\pi_P^{-1}(x))$ at any $x \in X_\alpha$ is isomorphic as $W_P$-modules. Then by Lemma \ref{lem:hs}, for any $w\in W_P$ and any $x_\alpha \in X_\alpha$ for each $\alpha$, we have
$$\ch H^*(\B_N)(w) = \sum_{\alpha} \cX(X_\alpha) \ch H^*(\pi_P^{-1}(x_\alpha))(w).$$

We may find two reductive groups $\tilde{L}, L'$ where $\tilde{L}$ is of type $A_{k-1}$ and $L'$ is of the same type as $G$, such that $\Lie L=\Lie \tilde{L}\oplus \Lie L'$. We define $\B(L), \B(\tilde{L}), \B(L')$ to be the flag varieties of $L, \tilde{L}, L'$, respectively. Then $\B(L) = \B(\tilde{L}) \times \B(L')$. Also we let $\tilde{N} \colonequals \Ad(g)^{-1}N|_{\Lie \tilde{L}}, N' \colonequals \Ad(g)^{-1}N|_{\Lie L'}$. For $w\in W_P$, we write $w = c w'$ for some $c \in \sym_k$ and $w' \in W'$. Then for any $gP \in \cP_N$ we have
$$\pi_P^{-1}(gP)\simeq \B(\tilde{L})_{\tilde{N}} \times \B(L')_{N'},$$
therefore
$$\ch H^*(\pi^{-1}(gP))(w)= \ch H^*(\B(\tilde{L})_{\tilde{N}})(c) \cdot\ch H^*(\B(L')_{N'})(w').$$

We further assume that $c\in \sym_k$ is a Coxeter element, i.e. a $k$-cycle. If $\tilde{N}$ is not regular, then $\ch H^*(\B(\tilde{L})_{\tilde{N}})(c)=0$ since $H^*(\B(\tilde{L})_{\tilde{N}})$ is an induced representation of $\sym_k$ from a proper parabolic subgroup (cf. \cite{lus04}). On the other hand, if $\tilde{N}$ is regular then $\ch H^*(\B(\tilde{L})_{\tilde{N}})(c)=1$.
Thus we have
\begin{lem} \label{lem:calc} Keep the notations and assumptions above. For any choice of $g_\alpha \in G$ such that $g_\alpha P \in X_\alpha$, we have
$$\ch H^*(\B_N)(w) = \sum_{\alpha} \cX(X_\alpha) \ch H^*(\B(L')_{N'})(w').$$
where the sum is over $\alpha$ such that $\tilde{N} \in \Lie \tilde{L}$ is regular.
\end{lem}

This lemma is our main tool. Later we choose a suitable stratification $\cP_N = \sqcup_{\alpha} X_\alpha$ and use Lemma \ref{lem:calc} to express $\ch H^*(\B_N)(w)$ in terms of the total Springer representations associated to simple groups of rank strictly smaller than $\rk G$.

\section{Geometric properties for type $B$ and $C$}
In this section, we use the symbol \tyB (resp. \tyC) when $G=SO_{2n+1}(\C)$ (resp. $G=Sp_{2n}(\C)$). Then $G$ naturally acts on $V=  \tyB \C^{2n+1}\ \tyC \C^{2n}$ that is equipped with a nondegenerate \tyB symmetric \tyC symplectic bilinear form $\br{\ ,\ }$. We identify $\B=G/B$ with the full isotropic flag variety of $V$. Also we set $P\subset G$ in Lemma \ref{lem:calc} to be a maximal parabolic subgroup such that $W_P \simeq \sym_k \times W'$ for some $1\leq k \leq n$.

\begin{rmk} Most calculations in this section are still valid for $G=SO_{2n} (n\geq 2), 1\leq k\leq n-2$. But we postpone this case until the next section.
\end{rmk}

\subsection{$k=1$ case} This part is a reformulation of \cite[Section 5]{kim:euler}. First we assume $k=1$, thus $W'$ is the maximal parabolic subgroup of $W$ of the same type. We identify $\cP=G/P$ with the Grassmannian of isotropic lines in $V$. Let $N\in \g$ be a nilpotent element of Jordan type $\lambda = (1^{m_1}2^{m_2} \cdots)$. First we recall \cite[Lemma 2.3.3]{vanleeuwen}.
\begin{lem} \label{lem:decomp} There exists an orthogonal decomposition $V = \bigoplus_{i \geq 1} V_i$ such that each $V_i$ is $N$-invariant and $N|_{V_i}$ has Jordan type $(i^{m_i})$.
Furthermore, for a given isotropic line $l \subset \bP(\ker N)$ such that $N|_{l^\perp/l}$ has Jordan type either $\lambda\pch{i}{i-2}$ or $\lambda\pch{i,i}{i-1,i-1}$, we may choose such a decomposition so that $l \subset V_i$.
\end{lem}

\begin{figure}
\begin{tikzpicture}[]	
\tikzset{bound/.style={
	thick
}} 
\tikzset{ins/.style={
	dotted
}} 

	\bluefill{0,0}{1,-8}

	\draw[bound] (0,0) -- (6,0);
	\draw[bound] (4,-1) -- (6,-1);
	\draw[bound] (3,-4) -- (4,-4);
	\draw[bound] (2,-6) -- (3,-6);
	\draw[bound] (0,-8) -- (2,-8);
	\draw[bound] (0,0) -- (0,-8);
	\draw[bound] (2,-6) -- (2,-8);
	\draw[bound] (3,-4) -- (3,-6);
	\draw[bound] (4,-1) -- (4,-4);
	\draw[bound] (6,0) -- (6,-1);
	
	\draw[ins] (0,-1) -- (4,-1);
	\draw[ins] (0,-2) -- (4,-2);
	\draw[ins] (0,-3) -- (4,-3);
	\draw[ins] (0,-4) -- (3,-4);
	\draw[ins] (0,-5) -- (3,-5);
	\draw[ins] (0,-6) -- (2,-6);
	\draw[ins] (0,-7) -- (2,-7);
	\draw[ins] (0,-8) -- (2,-8);
	\draw[ins] (1,0) -- (1,-8);
	\draw[ins] (2,0) -- (2,-6);
	\draw[ins] (3,0) -- (3,-4);
	\draw[ins] (4,0) -- (4,-1);
	\draw[ins] (5,0) -- (5,-1);

	\orangebox{0,0}{6,-1}
	\orangebox{0,-1}{4,-4}
	\orangebox{0,-4}{3,-6}
	\orangebox{0,-6}{2,-8}
	
	\node[] at (3,-0.5) {\LARGE $V_6$};
	\node[] at (2,-2.5) {\LARGE $V_4$};
	\node[] at (1.5,-5) {\LARGE $V_3$};
	\node[] at (1,-7) {\LARGE $V_2$};

	\node[] at (-0.7,-0.5) {$\ker_6 N$};
	\node[] at (-0.7,-2.5) {$\ker_4 N$};
	\node[] at (-0.7,-5) {$\ker_3 N$};
	\node[] at (-0.7,-7) {$\ker_2 N$};

\end{tikzpicture}
\caption{$V=\bigoplus_i V_i$ and $\ker N = \bigoplus_i \ker_i N$ for $G=Sp_{28}(\C)$ and $\lambda=(2^23^24^36^1)$}
(Here each orange box is $V_i$ for some $i \in \N,$ and the blue area corresponds to $\ker N$.)
\label{fig:decomp}
\end{figure}

We fix such a decomposition and define $\ker_i N\colonequals  \ker N \cap V_i$. (See Figure \ref{fig:decomp}.) For $v \in \ker N$, we let $v_i \in \ker_i N$ be the projection of $v$ onto $\ker_i N$ with respect to the decomposition above so that $v = \sum_{i\geq 1} v_i$.  The following lemma can be deduced from simple calculation with \cite[Section 2]{sho83} or \cite[Section 5]{kim:euler}.
\begin{lem} \label{lem:determ}
For $0\neq v\in \ker N$ such that $\br{v,v}=0$, let $r\in\N$ be the smallest integer such that $v_r \neq 0$.
Then the Jordan type of $N|_{\spn{v}^\perp/\spn{v}}$ only depends on $\spn{v_r} \in \bP(\ker_r N)$. Indeed, if the Jordan type of $N$ is $\lambda$, then $N|_{\spn{v}^\perp/\spn{v}} \textup{ has Jordan type } \lambda\pch{r}{r-2} \textup{ or } \lambda\pch{r,r}{r-1,r-1}.$
Furthermore, if $r$ is \tyB even \tyC odd, then the Jordan type of $N|_{\spn{v}^\perp/\spn{v}}$ is always $\lambda\pch{r,r}{r-1,r-1}$. If $r$ is \tyB odd \tyC even, then there exists a smooth quadric hypersurface $H \subset \bP(\ker_r N)$ such that
\begin{equation} \label{quadH}
\begin{aligned}
\spn{v_r} \in H &\Leftrightarrow \textup{ the Jordan type of } N|_{\spn{v}^\perp/\spn{v}} \textup{ is }\lambda\pch{r,r}{r-1,r-1},
\\\spn{v_r} \notin H &\Leftrightarrow \textup{ the Jordan type of } N|_{\spn{v}^\perp/\spn{v}} \textup{ is }\lambda\pch{r}{r-2}.
\end{aligned}
\end{equation}
\end{lem}
This lemma is used to prove the following proposition.
\begin{prop} \label{prop:k=1} For $\mu \vdash \dim V-2$, we define
$$X_\mu \colonequals \{l \in \bP(\ker N) \mid \br{l,l}=0, N|_{l^\perp/l} \textup{ has Jordan type } \mu \}.$$
Then $X_\mu$ is empty unless $\mu =\lambda\pch{r,r}{r-1,r-1}$ for some $r \geq 1$ or $\lambda\pch{r}{r-2}$ for some $ r\geq 2$. Also we have
\begin{align*} 
\cX(X_{\lambda\pch{r,r}{r-1,r-1}}) = 2\floor{\frac{m_r}{2}}, \qquad \cX(X_{\lambda\pch{r}{r-2}}) = \frac{1-(-1)^{m_r}}{2}.
\end{align*}
\end{prop}
\begin{proof} The first part is a direct consequence of Lemma \ref{lem:determ}. For the second claim, we first assume $r\geq2$ and note that $X_{\lambda\pch{r,r}{r-1,r-1}}\sqcup X_{\lambda\pch{r}{r-2}}=\bP(\bigoplus_{i\geq r} \ker_i N) - \bP(\bigoplus_{i>r} \ker_i N)$. If $r$ is \tyB even \tyC odd then $X_{\lambda\pch{r,r}{r-1,r-1}} = \bP(\bigoplus_{i\geq r} \ker_i N) - \bP(\bigoplus_{i>r} \ker_i N)$ and $X_{\lambda\pch{r}{r-2}} = \emptyset$ also by Lemma \ref{lem:determ}. Since $m_r$ is always even in this case, we have
\begin{gather*}
\cX(X_{\lambda\pch{r,r}{r-1,r-1}}) = \cX(\bP(\bigoplus_{i\geq r} \ker_i N)) - \cX(\bP(\bigoplus_{i>r} \ker_i N)) = m_r = 2\floor{\frac{m_r}{2}},
\\\cX(X_{\lambda\pch{r}{r-2}}) =0= \frac{1-(-1)^{m_r}}{2},
\end{gather*}
thus the result follows.

Now suppose $r$ is \tyB odd \tyC even. We have a projection map
$$\phi_r : \bP(\bigoplus_{i\geq r} \ker_i N) - \bP(\bigoplus_{i>r} \ker_i N) \rightarrow \bP(\ker_r N) : \spn{v} \mapsto \spn{v_r}.$$
which is a well-defined affine bundle with fiber isomorphic to $\bigoplus_{i>r} \ker_i N$. As $\phi_r(X_{\lambda\pch{r,r}{r-1,r-1}}) = H$ and $\phi_r(X_{\lambda\pch{r}{r-2}}) = \bP(\ker_r N) - H$ where $H$ is as in (\ref{quadH}), $\cX(X_{\lambda\pch{r,r}{r-1,r-1}})=\cX(H)$ and $\cX(X_{\lambda\pch{r}{r-2}}) = \cX(\bP(\ker_r N) ) - \cX(H)$. Then we use the fact that the Euler characteristic of a smooth quadric hypersurface in $\bP^{n-1}$ is $2\floor{\frac{n}{2}}.$ If $r=1$, then we argue similarly with the projection
$$\phi_1 : X_{\lambda\pch{1,1}{0,0}}\rightarrow \bP(\ker_1 N): \spn{v}\mapsto \spn{v_1}$$
whose image is $H \subset \bP(\ker_1 N)$ defined in (\ref{quadH}).
\end{proof}
From this proposition and Lemma \ref{lem:calc} we deduce the following result.
\begin{prop} Let $W'\subset W$ be the maximal parabolic subgroup of the same type. For $\lambda=(1^{m_1}2^{m_2}\cdots)$ we have
\begin{align*} \Res^W_{W'} \ch\tsp{\lambda} &= \sum_{i\geq 1, m_i \geq 2} 2\floor{\frac{m_i}{2}}\ch \tsp{\lambda\pch{i,i}{i-1,i-1}}+\sum_{i\geq 2, m_i \textup{ odd}} \ch \tsp{\lambda\pch{i}{i-2}}.
\end{align*}
Here $\tsp{\lambda\pch{i,i}{i-1,i-1}},\tsp{\lambda\pch{i}{i-2}}$ are total Springer representations of $W'$.
\end{prop}

\subsection{$k=2$ case} Here, we assume $k=2$ and identify $\cP=G/P$ with the Grassmannian of two-dimensional isotropic subspaces in $V$. (This is a crucial step for generalizing our statement to the case for arbitrary $k$.) We also assume that an isotropic line $l \subset \ker N$ is given and that $N|_{l^\perp/l}$ has Jordan type $\lambda\pch{i}{i-2}$ or $\lambda\pch{i,i}{i-1,i-1}$ for some fixed $i\geq 2$. (The reason to exclude $i=1$ case will be apparent when we state Proposition \ref{prop:k=2}.) We use Lemma \ref{lem:decomp} to fix a decomposition $V = \bigoplus_{j\geq 1} V_j$ such that $l \subset V_i$. We write $\wt{N} \colonequals N|_{l^\perp/l}$ and denote:
\begin{enumerate}[label=$\bullet$\textup{ Case }$\Alph*.$,leftmargin=1cm,itemindent=1cm]
\item if $\wt{N}$ has Jordan type $\lambda\pch{i}{i-2}$, and
\item if $\wt{N}$ has Jordan type $\lambda\pch{i,i}{i-1,i-1}$.
\end{enumerate}
We start with improving Lemma \ref{lem:decomp}. (See Figure \ref{fig:decomp2}.)

\begin{lem} Keep the assumption in Lemma \ref{lem:decomp} and suppose $l \subset V_i$. 
\begin{enumerate}[label=$\bullet$\textup{ Case }$\Alph*.$,leftmargin=1cm,itemindent=1cm]
\item There exists an orthogonal decomposition $V_i = V_i' \oplus V_i''$ such that
\begin{enumerate}[label=$\bullet$,leftmargin=0.5cm]
\item $V_i'$ and $V_i''$ are $N$-invariant,
\item $N|_{V_i'}, N|_{V_i''}$ have Jordan types $(i^{m_i-1})$, $(i^1)$, respectively, and
\item $l \subset V_i''$. (In particular, $l=\ker N \cap V_i''$)
\end{enumerate}
Thus in particular, $ V_i' \perp l$ and $N|_{(l^\perp \cap V_i'')/l}$ has Jordan type $((i-2)^1)$.
\item There exists an orthogonal decomposition $V_i = V_i' \oplus V_i''$ such that
\begin{enumerate}[label=$\bullet$]
\item $V_i'$ and $V_i''$ are $N$-invariant,
\item $N|_{V_i'}, N|_{V_i''}$ have Jordan types $(i^{m_i-2})$, $(i^2)$, respectively, and
\item $l \subset  V_i''$.
\end{enumerate}
Thus in particular, $V_i'\perp l$ and $N|_{(l^\perp \cap V_i'')/l}$ has Jordan type $((i-1)^2)$.
\end{enumerate}
\end{lem}

\begin{figure}
\begin{subfigure}{0.5\textwidth}
\begin{tikzpicture}[scale=1]
\tikzset{bound/.style={
	thick
}} 
\tikzset{ins/.style={
	dotted
}} 

	\yellowfill{0,-3}{1,-4}
	\hatch{0,0}{6,-1}
	\hatch{0,-1}{4,-3}
	\hatch{0,-3}{3,-6}
	\hatch{0,-6}{2,-8}
	\draw[line width=0.5mm] (3.3,-3.3)--(3.7,-3.7);
	\draw[line width=0.5mm] (3.3,-3.7)--(3.7,-3.3);

	\draw[bound] (0,0) -- (6,0);
	\draw[bound] (4,-1) -- (6,-1);
	\draw[bound] (3,-4) -- (4,-4);
	\draw[bound] (2,-6) -- (3,-6);
	\draw[bound] (0,-8) -- (2,-8);
	\draw[bound] (0,0) -- (0,-8);
	\draw[bound] (2,-6) -- (2,-8);
	\draw[bound] (3,-4) -- (3,-6);
	\draw[bound] (4,-1) -- (4,-4);
	\draw[bound] (6,0) -- (6,-1);
	
	\draw[ins] (0,-1) -- (4,-1);
	\draw[ins] (0,-2) -- (4,-2);
	\draw[ins] (0,-3) -- (4,-3);
	\draw[ins] (0,-4) -- (3,-4);
	\draw[ins] (0,-5) -- (3,-5);
	\draw[ins] (0,-6) -- (2,-6);
	\draw[ins] (0,-7) -- (2,-7);
	\draw[ins] (0,-8) -- (2,-8);
	\draw[ins] (1,0) -- (1,-8);
	\draw[ins] (2,0) -- (2,-6);
	\draw[ins] (3,0) -- (3,-4);
	\draw[ins] (4,0) -- (4,-1);
	\draw[ins] (5,0) -- (5,-1);

	\orangebox{0,0}{6,-1}
	\orangebox{0,-1}{4,-4}
	\orangebox{0,-4}{3,-6}
	\orangebox{0,-6}{2,-8}
	
	\draw[line width=1mm, orange, dotted] (0,-3) -- (4,-3);
	\node[] at (3,-0.5) {\LARGE $V_6$};
	\node[] at (2,-2) {\LARGE $V_4'$};
	\node[] at (2,-3.5) {\LARGE $V_4''$};
	\node[] at (1.5,-5) {\LARGE $V_3$};
	\node[] at (1,-7) {\LARGE $V_2$};

	\node[] at (0.5,-3.5) {\LARGE $l$};

\end{tikzpicture}
\caption{Case $A$, $i=4$}
\end{subfigure}%
\begin{subfigure}{0.5\textwidth}
\begin{tikzpicture}[scale=1]
\tikzset{bound/.style={
	thick
}} 
\tikzset{ins/.style={
	dotted
}} 

	\yellowfill{0,-3}{1,-4}
	\hatch{0,0}{6,-1}
	\hatch{0,-1}{4,-2}
	\hatch{0,-2}{3,-3}
	\hatch{0,-3}{4,-4}
	\hatch{0,-4}{3,-6}
	\hatch{0,-6}{2,-8}
	\draw[line width=0.5mm] (3.3,-2.3)--(3.7,-2.7);
	\draw[line width=0.5mm] (3.3,-2.7)--(3.7,-2.3);

	\draw[bound] (0,0) -- (6,0);
	\draw[bound] (4,-1) -- (6,-1);
	\draw[bound] (3,-4) -- (4,-4);
	\draw[bound] (2,-6) -- (3,-6);
	\draw[bound] (0,-8) -- (2,-8);
	\draw[bound] (0,0) -- (0,-8);
	\draw[bound] (2,-6) -- (2,-8);
	\draw[bound] (3,-4) -- (3,-6);
	\draw[bound] (4,-1) -- (4,-4);
	\draw[bound] (6,0) -- (6,-1);
	
	\draw[ins] (0,-1) -- (4,-1);
	\draw[ins] (0,-2) -- (4,-2);
	\draw[ins] (0,-3) -- (4,-3);
	\draw[ins] (0,-4) -- (3,-4);
	\draw[ins] (0,-5) -- (3,-5);
	\draw[ins] (0,-6) -- (2,-6);
	\draw[ins] (0,-7) -- (2,-7);
	\draw[ins] (0,-8) -- (2,-8);
	\draw[ins] (1,0) -- (1,-8);
	\draw[ins] (2,0) -- (2,-6);
	\draw[ins] (3,0) -- (3,-4);
	\draw[ins] (4,0) -- (4,-1);
	\draw[ins] (5,0) -- (5,-1);

	\orangebox{0,0}{6,-1}
	\orangebox{0,-1}{4,-4}
	\orangebox{0,-4}{3,-6}
	\orangebox{0,-6}{2,-8}
	
	\draw[line width=1mm, orange, dotted] (0,-2) -- (4,-2);
	\node[] at (3,-0.5) {\LARGE $V_6$};
	\node[] at (2,-1.5) {\LARGE $V_4'$};
	\node[] at (2,-3) {\LARGE $V_4''$};
	\node[] at (1.5,-5) {\LARGE $V_3$};
	\node[] at (1,-7) {\LARGE $V_2$};
	
	\node[] at (0.5,-3.5) {\LARGE $l$};
\end{tikzpicture}
\caption{Case $B$, $i=4$}
\end{subfigure}
\caption{$V=(\bigoplus_{j\neq i}V_j)\oplus V_i'\oplus V_i''$ for $G=Sp_{28}(\C)$ and $\lambda=(2^23^24^36^1)$}
(Here the yellow box is $l$ and the blue-dotted area corresponds to $l^\perp$.)
\label{fig:decomp2}
\end{figure}

\begin{proof} It suffices to consider the case when $\lambda=(i^{m_i})$ which we assume here. First assume Case $A$. Then we have an orthogonal decomposition $l^\perp/ l = \wt{V_{i}}\oplus \wt{V_{i-2}}$
which satisfies the property in Lemma \ref{lem:decomp} with respect to $\wt{N}$. Now take the inverse image of $\wt{V_{i}}$ under the map $l^\perp\twoheadrightarrow l^\perp/l$, say $W_i \subset \C^{2n}$, and take a complementary space $V_i'$ of $l$ in $W_i$. We also set $V_i'' = (V_i')^\perp$. Then it is easy to see that they satisfy the desired properties. Case $B$ is also similar.
\end{proof}

We want to have a similar decomposition of $l^\perp/l$ similar to Lemma \ref{lem:decomp} with respect to $\wt{N}$, in accordance with $V=(\bigoplus_{j\neq i}V_j)\oplus V_i'\oplus V_i''$. The following lemma is an easy exercise.
\begin{lem} Keep the assumptions above. 
\begin{enumerate}[label=$\bullet$\textup{ Case }$\Alph*.$,leftmargin=1cm,itemindent=1cm]
\item The decomposition $l^\perp/l = \bigoplus_{j \geq 1}\wt{V_j}$ where
$$\wt{V_j} = (V_j\oplus l)/l \textup{ for } j\neq i, i-2, \qquad \wt{V_i} = (V_i'\oplus l)/l,  \qquad \wt{V_{i-2}}=\left((V_{i}''\cap l^\perp)\oplus V_{i-2}\right)/l$$
gives the decomposition of $l^\perp/l$ which satisfies the criteria in Lemma \ref{lem:decomp} with respect to $\wt{N}$.
\item The decomposition $l^\perp/l = \bigoplus_{j \geq 1}\wt{V_j}$ where
$$\wt{V_j} = (V_j\oplus l)/l \textup{ for } j\neq i, i-1, \qquad \wt{V_i} = (V_i'\oplus l)/l, \qquad \wt{V_{i-1}}=\left((V_{i}''\cap l^\perp)\oplus V_{i-1}\right)/l$$
gives the decomposition of $l^\perp/l$ which satisfies the criteria in Lemma \ref{lem:decomp} with respect to $\wt{N}$.
\end{enumerate}
\end{lem}

Previously we defined $\ker_j N \colonequals \ker N \cap V_j$. Similarly, we also let $\ker_i' N \colonequals \ker N \cap V_i', \ker_i'' \colonequals \ker N \cap V_i''$ and $\ker_j \wt{N} \colonequals \ker\wt{N} \cap \wt{V_j}$. Then it is easy to see that
\begin{enumerate}[label=$\bullet$\textup{ Case }$\Alph*.$,leftmargin=1cm,itemindent=1cm]
\item 
\begin{gather*}
\ker_{j} \wt{N} = (\ker_j N\oplus l)/l \textup{ for } j\neq i, i-2,
\\\ker_{i} \wt{N} = (\ker_i' N\oplus l)/l, \qquad \ker_{i-2} \wt{N} = \left((N^{-1}l\cap V_i'')\oplus V_{i-2}\right)/l.
\end{gather*}
\item 
\begin{gather*}
ker_{j} \wt{N} = (\ker_j N\oplus l)/l \textup{ for } j\neq i, i-1,
\\ \qquad \ker_{i} \wt{N} = (\ker_i' N\oplus l)/l, \qquad \ker_{i-1} \wt{N} = \left((N^{-1}l\cap V_i'')\oplus V_{i-1}\right)/l.
\end{gather*}
\end{enumerate}

From now on we denote by $\ol{v} \colonequals v+l\in V/l$ the image of $v \in V$ under the projection $V\twoheadrightarrow V/l$.
We choose $v_0 \in N^{-1}l \cap V_i''$ such that $\spn{Nv_0}=l$. 
 For any $v\in \ker N$, we write $\tilde{v}\colonequals\overline{v_0+v}$. Also we let $v=\sum_{j\geq 1} v_j$ such that $v_j \in \ker_j N$ and let $v_i = v_i'+v_i''$ where $v_i' \in \ker_i' N$ and $v_i''\in \ker_i'' N$. (See Figure \ref{fig:vi}.) Then $\tilde{v} = \sum_{j\geq 1} \tilde{v}_j, \tilde{v}_j \in \ker_j \wt{N}$ where 
\begin{enumerate}[label=$\bullet$\textup{ Case }$\Alph*.$,leftmargin=1cm,itemindent=1cm]
\item $\tilde{v}_j = \ol{v_j} \textup{ for } j\neq i-2 \textup{ (in particular, } \tilde{v}_i= \ol{v_i'} = \ol{v_i}) \qquad\tilde{v}_{i-2} = \ol{v_0+v_{i-2}}$,
\item $\tilde{v}_j = \ol{v_j} \textup{ for } j\neq i, i-1, \qquad \tilde{v}_i= \ol{v_i'}, \qquad \tilde{v}_{i-1} = \ol{v_0+v_{i}''+v_{i-1}}$.
\end{enumerate}

\begin{figure}
\begin{subfigure}{0.5\textwidth}
\begin{tikzpicture}[scale=1]
\tikzset{bound/.style={
	thick
}} 
\tikzset{ins/.style={
	dotted
}} 

	\bluefill{0,0}{1,-3}
	\bluefill{0,-4}{1,-8}
	\bluefill{1,-3}{2,-4}
	\yellowfill{0,-3}{1,-4}
	\hatch{0,0}{6,-1}
	\hatch{0,-1}{4,-3}
	\hatch{0,-3}{3,-6}
	\hatch{0,-6}{2,-8}
	\draw[line width=0.5mm] (3.3,-3.3)--(3.7,-3.7);
	\draw[line width=0.5mm] (3.3,-3.7)--(3.7,-3.3);

	\draw[bound] (0,0) -- (6,0);
	\draw[bound] (4,-1) -- (6,-1);
	\draw[bound] (3,-4) -- (4,-4);
	\draw[bound] (2,-6) -- (3,-6);
	\draw[bound] (0,-8) -- (2,-8);
	\draw[bound] (0,0) -- (0,-8);
	\draw[bound] (2,-6) -- (2,-8);
	\draw[bound] (3,-4) -- (3,-6);
	\draw[bound] (4,-1) -- (4,-4);
	\draw[bound] (6,0) -- (6,-1);
	
	\draw[ins] (0,-1) -- (4,-1);
	\draw[ins] (0,-2) -- (4,-2);
	\draw[ins] (0,-3) -- (4,-3);
	\draw[ins] (0,-4) -- (3,-4);
	\draw[ins] (0,-5) -- (3,-5);
	\draw[ins] (0,-6) -- (2,-6);
	\draw[ins] (0,-7) -- (2,-7);
	\draw[ins] (0,-8) -- (2,-8);
	\draw[ins] (1,0) -- (1,-8);
	\draw[ins] (2,0) -- (2,-6);
	\draw[ins] (3,0) -- (3,-4);
	\draw[ins] (4,0) -- (4,-1);
	\draw[ins] (5,0) -- (5,-1);

	\orangebox{0,0}{6,-1}
	\orangebox{0,-1}{4,-4}
	\orangebox{0,-4}{3,-6}
	\orangebox{0,-6}{2,-8}
	
	\draw[line width=1mm, orange, dotted] (0,-3) -- (4,-3);

	\node[] at (0.5,-0.5) {$v_6$};
	\node[] at (0.5,-2) {$v_4'$};
	\node[] at (0.5,-3.5) {$v_4''$};
	\node[] at (0.5,-5) {$v_3$};
	\node[] at (0.5,-7) {$v_2$};
	\node[] at (1.5,-3.5) {$v_0$};

\end{tikzpicture}
\caption{Case $A$, $i=4$}
\end{subfigure}%
\begin{subfigure}{0.5\textwidth}
\begin{tikzpicture}[scale=1]
\tikzset{bound/.style={
	thick
}} 
\tikzset{ins/.style={
	dotted
}} 

	\bluefill{0,0}{1,-3}
	\bluefill{0,-4}{1,-8}
	\bluefill{1,-3}{2,-4}
	\yellowfill{0,-3}{1,-4}
	\hatch{0,0}{6,-1}
	\hatch{0,-1}{4,-2}
	\hatch{0,-2}{3,-3}
	\hatch{0,-3}{4,-4}
	\hatch{0,-4}{3,-6}
	\hatch{0,-6}{2,-8}
	\draw[line width=0.5mm] (3.3,-2.3)--(3.7,-2.7);
	\draw[line width=0.5mm] (3.3,-2.7)--(3.7,-2.3);

	\draw[bound] (0,0) -- (6,0);
	\draw[bound] (4,-1) -- (6,-1);
	\draw[bound] (3,-4) -- (4,-4);
	\draw[bound] (2,-6) -- (3,-6);
	\draw[bound] (0,-8) -- (2,-8);
	\draw[bound] (0,0) -- (0,-8);
	\draw[bound] (2,-6) -- (2,-8);
	\draw[bound] (3,-4) -- (3,-6);
	\draw[bound] (4,-1) -- (4,-4);
	\draw[bound] (6,0) -- (6,-1);
	
	\draw[ins] (0,-1) -- (4,-1);
	\draw[ins] (0,-2) -- (4,-2);
	\draw[ins] (0,-3) -- (4,-3);
	\draw[ins] (0,-4) -- (3,-4);
	\draw[ins] (0,-5) -- (3,-5);
	\draw[ins] (0,-6) -- (2,-6);
	\draw[ins] (0,-7) -- (2,-7);
	\draw[ins] (0,-8) -- (2,-8);
	\draw[ins] (1,0) -- (1,-8);
	\draw[ins] (2,0) -- (2,-6);
	\draw[ins] (3,0) -- (3,-4);
	\draw[ins] (4,0) -- (4,-1);
	\draw[ins] (5,0) -- (5,-1);

	\orangebox{0,0}{6,-1}
	\orangebox{0,-1}{4,-4}
	\orangebox{0,-4}{3,-6}
	\orangebox{0,-6}{2,-8}
	
	\draw[line width=1mm, orange, dotted] (0,-2) -- (4,-2);

	\node[] at (0.5,-0.5) {$v_6$};
	\node[] at (0.5,-1.5) {$v_4'$};
	\node[] at (0.5,-3) {$v_4''$};
	\node[] at (0.5,-5) {$v_3$};
	\node[] at (0.5,-7) {$v_2$};
	\node[] at (1.5,-3.5) {$v_0$};

\end{tikzpicture}
\caption{Case $B$, $i=4$}
\end{subfigure}
\caption{$v=(\sum_{j\neq i} v_j) +v_i'+v_i''$ and $v_0 \in N^{-1}l$ for $G=Sp_{28}(\C)$ and $\lambda=(2^23^24^36^1)$}
(Here the yellow box is $l$ and the blue area corresponds to $\ker \wt{N}$.)
\label{fig:vi}
\end{figure}

Let $r, r'\in \N$ be the smallest integers such that $v_r \neq 0$ and $\tilde{v}_{r'} \neq 0$. Then from the observation above it is clear that
\begin{enumerate}[label=$\bullet$\textup{ Case }$\Alph*.$,leftmargin=1cm,itemindent=1cm]
\item $r' = \min(r,i-2)$.
\item $r' = \min(r,i-1)$.
\end{enumerate}
Therefore, for $\V \subset V$ such that $\dim \V=2$ and $N\V=l\subset \V$, using Lemma \ref{lem:determ} we have
\begin{enumerate}[label=$\bullet$\textup{ Case }$\Alph*.$,leftmargin=1cm,itemindent=1cm]
\item the Jordan type of $N|_{\V^\perp/\V}$ is either 
$$ \quad\lambda\pch{i}{i-4},\quad \lambda\pch{i,i-2}{i-3,i-3},\quad \lambda\pch{i,r}{i-2,r-2}, \quad \textup{ or }\lambda\pch{i,r,r}{i-2,r-1,r-1},$$
for some $r<i-2$.
\item the Jordan type of $N|_{\V^\perp/\V}$ is either 
$$\lambda\pch{i,i}{i-1,i-3},\quad \lambda\pch{i,i}{i-2,i-2}, \quad \lambda\pch{i,i,r}{i-1,i-1,r-2},\quad \textup{ or }\lambda\pch{i,i,r,r}{i-1,i-1,r-1,r-1},$$
for some $r<i-1$.
\end{enumerate}
Now we state a key step to calculate the formula in Lemma \ref{lem:calc} for $k=2$ case.
\begin{prop} \label{prop:k=2} Let $l \in \bP(\ker N)$ be an isotropic line and define
$$X_{\mu,l} \colonequals \{ \V \subset V\mid \br{\V,\V}=0,\dim \V =2, N\V=l\subset \V, N|_{\V^\perp / \V} \textup{ has Jordan type } \mu\}.$$
\begin{enumerate}[label=$\bullet$\textup{ Case }$\Alph*.$,leftmargin=1cm,itemindent=1cm]
\item We have
\begin{enumerate}[label=(\alph*)]
\item $\cX(X_{\mu,l}) = 1-(-1)^{m_{i-2}}$ if $\mu = \lambda\pch{i,i-2}{i-3,i-3}$ for $i\geq 3$,
\item $\cX(X_{\mu,l}) = (-1)^{m_{i-2}}$ if $\mu = \lambda\pch{i}{i-4}$ for $i \geq 4$, and
\item $\cX(X_{\mu,l}) =0$ otherwise.
\end{enumerate}
\item We have
\begin{enumerate}[label=(\alph*)]
\item $\cX(X_{\mu,l}) = 1$ if $\mu = \lambda\pch{i,i}{i-2,i-2}$ for $i \geq 2$,
\item $\cX(X_{\mu,l})  = 0$ otherwise.
\end{enumerate}
\end{enumerate}
\end{prop}
\begin{rmk} If $i=1$, then there does not exist any $\V \in V$ such that $N\V=l\subset\V$. Thus it is enough to assume $i\geq 2$ as we did at the beginning.
\end{rmk}
\begin{proof}
First we consider Case $A(c)$. For $\mu$ different from $\lambda\pch{i,i-2}{i-3,i-3}$ and $\lambda\pch{i}{i-4}$, it suffices to assume that $\mu = \lambda\pch{i,r}{i-2,r-2}$ or $\lambda\pch{i,r,r}{i-2,r-1,r-1}$ for some $1\leq r<i-2$. Moreover, it is enough to assume that $i\geq 4$. We briefly write $\mu^1 = \lambda\pch{i,r}{i-2,r-2}$ and $\mu^2 = \lambda\pch{i,r,r}{i-2,r-1,r-1}$.

We can choose $v_0 \in N^{-1}l \cap V_i''$ such that $\spn{Nv_0}=l$. Then the map
$$\psi: \ker N \rightarrow \{\V \subset V  \mid \br{\V,\V}=0, \dim \V=2, N\V = l\subset \V\}: v \mapsto \spn{v_0+v}\oplus l$$
is an affine bundle with fiber $l$. By Lemma \ref{lem:determ}, we have
$$\psi(v) \in X_{\mu^1,l}\sqcup X_{\mu^2,l} \Leftrightarrow \tilde{v} \in \bigoplus_{j\geq r}\ker_j \wt{N} - \bigoplus_{j>r}\ker_j \wt{N} \Leftrightarrow v\in \bigoplus_{j\geq r}\ker_j N - \bigoplus_{j>r}\ker_j N$$
where $\tilde{v} = \ol{v_0+v}$. (The condition $\br{\psi(v),\psi(v)}=0$ is automatic for $i\geq 4$.) Thus we have
$$\psi^{-1}(X_{ \mu^1,l} \sqcup X_{\mu^2,l}) = \bigoplus_{j\geq r}\ker_j N - \bigoplus_{j>r}\ker_j N.$$
For such $r$, consider
$$\phi_r : \bigoplus_{j\geq r}\ker_j N - \bigoplus_{j>r}\ker_j N \rightarrow \bP(\ker_r \wt{N}): v \mapsto \spn{\tilde{v}_r}.$$
This map is a fiber bundle with fiber isomorphic to $ \left(\bigoplus_{j>r}\ker_j N\right)\times (\A^1-\{0\})$. As the Jordan type of $N|_{\psi(v)^\perp/\psi(v)}$ is determined by $\spn{\tilde{v}_r} \in \bP(\ker_r \wt{N})$, there exists $Y_{\mu^1,l}, Y_{\mu^2,l} \subset \bP(\ker_r \wt{N})$ such that
\begin{gather*}
\phi_r^{-1}(Y_{\mu^1,l}) =\psi^{-1}(X_{\mu^1,l}),
\qquad \phi_r^{-1}(Y_{\mu^2,l}) =\psi^{-1}(X_{\mu^2,l}).
\end{gather*}
But since the Euler characteristic of a fiber of $\phi_r$ is 0, it follows that $\cX(X_{\mu^1}) = \cX(\psi^{-1}(X_{ \mu^1,l})) = \cX(\phi_r^{-1}(Y_{\mu^1,l})) = 0$ and similarly $\cX(X_{\mu^2,l})= 0$. This proves Case $A(c).$ Case $B(b)$ when $\mu \neq \lambda\pch{i, i}{i-1,i-3}$ is also similar.

Now consider Case $A(a)$ and $A(b)$, thus it is enough to assume that $i\geq 3$. Then $\tilde{v}_{i-2}=\ol{v_0+v_{i-2}}$ is always nonzero. This time we set $\mu^1 = \lambda\pch{i,i-2}{i-3,i-3}$ and $\mu^2 =  \lambda\pch{i}{i-4}$. Here we only deal with the case when $i\geq 4$, but $i=3$ case is totally analogous. (If $i=3$, $\mu^2$ is not well-defined and one needs to be careful about the condition $\br{\psi(v), \psi(v)}=0$ which is automatic for $i\geq 4$.)

First we observe the following. (Also see Figure \ref{fig:rest}.)
\begin{lem} \label{lem:indep1}$\cX(X_{\mu^1,l})$ and $\cX(X_{\mu^2,l})$ become unchanged if we put $V_j = 0$ for $j\neq i, i-2$ and $V_i'=0$, i.e. if $V=V_i'' \oplus V_{i-2}$.
\end{lem}
\begin{proof}[Proof of Lemma \ref{lem:indep1}]
Again we define
$$\psi: \ker N \rightarrow \{\V \subset V  \mid \br{\V,\V}=0,\dim \V=2, N\V = l\subset \V\}: v \mapsto \spn{v_0+v}\oplus l.$$
Then similarly to above we have
$$\psi^{-1}(X_{\mu^1,l}\sqcup X_{\mu^2,l}) = \bigoplus_{j\geq i-2}\ker_{j} N.$$
Also we let
$$\phi' : \bigoplus_{j\geq i-2} \ker_j N \rightarrow \bP(\ker_{i-2} \wt{N}) : v \mapsto \spn{\tilde{v}_{i-2}}.$$
The image of $\phi'$ is $\bP(\ker_{i-2} \wt{N}) - \bP((\ker_{i-2} N \oplus l)/l)$ which is isomorphic to an affine space, and when restricted to the image $\phi'$ is a fiber bundle with fiber isomorphic to $\bigoplus_{j>i-2}\ker_j N$. Also by the same reason as above there exists $Y_{\mu^1,l}, Y_{\mu^2,l} \subset \bP(\ker_{i-2} \wt{N})$ such that
\begin{gather*}
\phi'^{-1}(Y_{\mu^1,l}) =\psi^{-1}(X_{\mu^1,l}), \qquad \phi'^{-1}(Y_{\mu^2,l}) =\psi^{-1}(X_{\mu^2,l}).
\end{gather*}
Then $\cX(X_{\mu^1,l}) = \cX(Y_{\mu^1,l})$ and $\cX(X_{\mu^2,l}) = \cX(Y_{\mu^2,l})$. We recall the orthogonal decomposition $V= \left(\bigoplus_{j\neq i}V_j\right) \oplus V_i' \oplus V_i''$. It is clear that the varieties $Y_{\mu^1,l}, Y_{\mu^2,l}$ only depend (up to isomorphism) on the restriction of $N$ to $V_i''\oplus V_{i-2}$ and $l \subset V_i''$. As $\cX(X_{\mu^1,l}) = \cX(Y_{\mu^1,l})$ and $\cX(X_{\mu^2,l}) = \cX(Y_{\mu^2,l})$, the claim follows.
\end{proof}
In other words, we may assume that $N$ has Jordan type $\lambda=((i-2)^{m_{i-2}}i^1)$ and $l \subset \ker_i N$. Note that $i$ should be \tyB odd \tyC even in this case. Following the idea of \cite[E.93]{springer-steinberg}, we choose vectors $e_0, e_1, \cdots, e_{m_{i-2}}\in V$ that satisfy the following.
\begin{enumerate}[label=$\bullet$]
\item $\{ N^j e_0 \mid 0 \leq j \leq i-1\} \cup \{N^{j'} e_{p} \mid  0\leq j' \leq i-3, 1\leq p \leq m_{i-2}\}$ is a basis of $V$.
\item The value of $\br{\ ,\ }$ on any pair of the elements above are zero except
$$(-1)^j\br{N^j e_0, N^{i-1-j}e_0}= (-1)^{j'}\br{N^{j'} e_p, N^{i-3-j'}e_p}=1$$
for $0 \leq j \leq i-1, 0\leq j' \leq i-3, 1\leq p \leq m_{i-2}$.
\item $l = \spn{N^{i-1}e_0}$.
\end{enumerate}
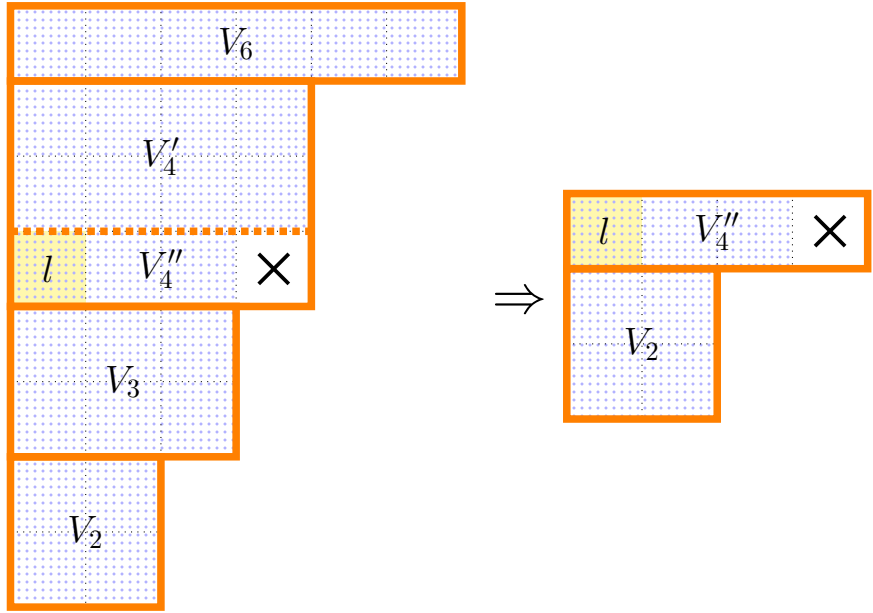
\begin{figure}
\centering
\begin{subfigure}{0.45\textwidth}
\centering
\begin{tikzpicture}[scale=1]
\tikzset{bound/.style={
	thick
}} 
\tikzset{ins/.style={
	dotted
}} 

	\yellowfill{0,-3}{1,-4}
	\hatch{0,0}{6,-1}
	\hatch{0,-1}{4,-3}
	\hatch{0,-3}{3,-6}
	\hatch{0,-6}{2,-8}
	\draw[line width=0.5mm] (3.3,-3.3)--(3.7,-3.7);
	\draw[line width=0.5mm] (3.3,-3.7)--(3.7,-3.3);

	\draw[bound] (0,0) -- (6,0);
	\draw[bound] (4,-1) -- (6,-1);
	\draw[bound] (3,-4) -- (4,-4);
	\draw[bound] (2,-6) -- (3,-6);
	\draw[bound] (0,-8) -- (2,-8);
	\draw[bound] (0,0) -- (0,-8);
	\draw[bound] (2,-6) -- (2,-8);
	\draw[bound] (3,-4) -- (3,-6);
	\draw[bound] (4,-1) -- (4,-4);
	\draw[bound] (6,0) -- (6,-1);
	
	\draw[ins] (0,-1) -- (4,-1);
	\draw[ins] (0,-2) -- (4,-2);
	\draw[ins] (0,-3) -- (4,-3);
	\draw[ins] (0,-4) -- (3,-4);
	\draw[ins] (0,-5) -- (3,-5);
	\draw[ins] (0,-6) -- (2,-6);
	\draw[ins] (0,-7) -- (2,-7);
	\draw[ins] (0,-8) -- (2,-8);
	\draw[ins] (1,0) -- (1,-8);
	\draw[ins] (2,0) -- (2,-6);
	\draw[ins] (3,0) -- (3,-4);
	\draw[ins] (4,0) -- (4,-1);
	\draw[ins] (5,0) -- (5,-1);

	\orangebox{0,0}{6,-1}
	\orangebox{0,-1}{4,-4}
	\orangebox{0,-4}{3,-6}
	\orangebox{0,-6}{2,-8}
	
	\draw[line width=1mm, orange, dotted] (0,-3) -- (4,-3);
	\node[] at (3,-0.5) {\LARGE $V_6$};
	\node[] at (2,-2) {\LARGE $V_4'$};
	\node[] at (2,-3.5) {\LARGE $V_4''$};
	\node[] at (1.5,-5) {\LARGE $V_3$};
	\node[] at (1,-7) {\LARGE $V_2$};

	\node[] at (0.5,-3.5) {\LARGE $l$};

\end{tikzpicture}
\end{subfigure}{\Huge $\Rightarrow$}%
\begin{subfigure}{0.3\textwidth}
\centering
\begin{tikzpicture}[scale=1]
\tikzset{bound/.style={
	thick
}} 
\tikzset{ins/.style={
	dotted
}} 

	\yellowfill{0,0}{1,-1}
	\hatch{0,0}{3,-1}
	\hatch{0,-1}{2,-3}

	\draw[line width=0.5mm] (3.3,-0.3)--(3.7,-0.7);
	\draw[line width=0.5mm] (3.3,-0.7)--(3.7,-0.3);

	\draw[bound] (0,0) -- (4,0);\draw[bound] (2,-1) -- (4,-1);\draw[bound] (0,-3) -- (2,-3);
	\draw[bound] (0,0) -- (0,-3);\draw[bound] (2,-1) -- (2,-3);\draw[bound] (4,0) -- (4,-1);
	
	\draw[ins] (0,-1) -- (2,-1);\draw[ins] (0,-2) -- (2,-2);
	\draw[ins] (1,0) -- (1,-3);\draw[ins] (2,0) -- (2,-1);\draw[ins] (3,0) -- (3,-1);

	\orangebox{0,0}{4,-1}
	\orangebox{0,-1}{2,-3}

	\node[] at (2,-0.5) {\LARGE $V_4''$};
	\node[] at (1,-2) {\LARGE $V_2$};

	\node[] at (0.5,-0.5) {\LARGE $l$};
\end{tikzpicture}
\end{subfigure}
\caption{Setting $V_6=V_4'=V_3=0$ in Case $A$, $i=4$ for $G=Sp_{28}(\C)$ and $\lambda=(2^23^24^36^1)$}
\label{fig:rest}
\end{figure}
Also we may assume that $v_0 = N^{i-2}e_0$. Then $\psi$ induces an isomorphism of varieties
$$\psi|_{\ker_{i-2} N} : \ker_{i-2} N \rightarrow \{\V \subset V \mid \br{\V, \V}=0, \dim \V = 2, N\V =l\subset \V\} : v \mapsto \spn{v_0+v}\oplus l.$$
We identify $\ker_{i-2} N$ with $\A^{m_{i-2}}$ using the basis $\{N^{i-3}e_p\}_{1\leq p \leq m_{i-2}}$. Then by direct calculation, for any $v \in \ker_{i-2} N$ such that $v = \sum_{p=1}^{m_{i-2}}a_p N^{i-3}e_p$ we have (see also \cite[p.249]{sho83})
$$\psi(v) \in X_{\mu^1,l} \Leftrightarrow 1+a_1^2+\cdots+a_{m_{i-2}}^2=0,\qquad \psi(v) \in X_{\mu^2,l} \Leftrightarrow 1+a_1^2+\cdots+a_{m_{i-2}}^2\neq0.$$
But for any $m \in \N$ it is easy to see that
$$\cX\left(\{ (a_1, \cdots, a_m) \in \A^{m} \mid a_1^2+\cdots+a_m^2+1=0\}\right) = 1-(-1)^{m}.$$
Therefore we have $\cX(X_{\mu^1,l}) = 1-(-1)^{m_{i-2}}$ and $\cX(X_{\mu^2,l}) = \cX(\A^{m_{i-2}}) - \cX(X_{\mu^1,l}) = (-1)^{m_{i-2}}$.

We proceed with Case $B(a)$ and $B(b)$ for $i \geq 2$ which are similar to above. Again $\tilde{v}_{i-1}=\ol{v_0+v_{i-1}}$ is always nonzero. This time we set $\mu^1 = \lambda\pch{i,i}{i-2,i-2}$ and $\mu^2 =  \lambda\pch{i,i}{i-1,i-3}$. Here we only deal with the case when $i\geq 3$, but $i=2$ case is totally analogous. (If $i=2$, $\mu^2$ is not well-defined and one needs to be careful about the condition $\br{\psi(v), \psi(v)}=0$, which is automatic when $i\geq 3$.)

First note that we may put $V_j = 0$ for $j \neq i, i-1$ and $V_i'=0$, which can be similarly justified to Lemma \ref{lem:indep1}. Therefore, it suffices to consider the case when $\lambda = ((i-1)^{m_{i-1}}i^2)$ and $l \subset \ker_i N$. If $i$ is \tyB odd \tyC even, then for any line $\tilde{l} \subset \ker_{i-1} \wt{N}$, $\wt{N}|_{\tilde{l}^\perp/\tilde{l}}$ has Jordan type $(\lambda\pch{i,i}{i-1,i-1})\pch{i-1,i-1}{i-2,i-2} = \lambda\pch{i,i}{i-2,i-2}$ by Lemma \ref{lem:determ}. Thus $X_{\mu^2,l} = \emptyset$ and $\psi(\ker N) = X_{\mu^1,l}$, i.e.
$$\psi: \ker N \rightarrow X_{\mu^1,l}: v \mapsto \spn{v_0+v}\oplus l.$$
is an affine bundle with fiber $l$. Thus
$$\cX(X_{\mu^1,l})  = \cX\left(\ker N\right) = 1,$$
which proves the claim in this case.

Now we assume that $i$ is \tyB even \tyC odd. Using the argument in \cite[E.93]{springer-steinberg}, we choose $f_0, g_0, e_1, \cdots, e_{m_{i-1}}$ such that the following conditions hold.
\begin{enumerate}[label=$\bullet$]
\item $\{ N^j f_0 \mid 0 \leq j \leq i-1\} \cup\{ N^j g_0 \mid 0 \leq j \leq i-1\} \cup \{N^{j'} e_{p} \mid  0\leq j' \leq i-2, 1\leq p \leq m_{i-1}\}$ is a basis of $V$.
\item The value of $\br{\ ,\ }$ on any pair of the elements above is zero except
$$(-1)^j\br{N^j f_0, N^{i-1-j}g_0}=\epsilon(-1)^j\br{N^{i-1-j} g_0, N^{i}f_0}= (-1)^{j'}\br{N^{j'} e_p, N^{i-2-{j'}}e_p}=1$$
for $0 \leq j \leq i-1, 1\leq p \leq m_{i-1}, 0\leq j' \leq i-2$. (Here $\epsilon$ equals \tyB 1 \tyC -1.)
\item $l = \spn{N^{i-1}f_0}$.
\end{enumerate}
Also we may assume that $v_0 = N^{i-2} f_0$. Then $\psi$ induces an isomorphism of varieties
\begin{align*}
\psi|_{\ker_{i-1} N\oplus \spn{N^{i-1} g_0}} &: \ker_{i-1} N\oplus \spn{N^{i-1} g_0} \rightarrow \{\V \subset V \mid \br{\V,\V}=0, \dim\V = 2, N\V =l\subset \V\} 
\\&: v \mapsto \spn{v_0+v}\oplus l.
\end{align*}
We identify $\ker_{i-1} N\oplus \spn{N^{i-1} g_0}$ with $\A^{m_{i-1}+1}$ using the basis $\{N^{i-2}e_p\}_{1\leq p \leq m_{i-2}}\cup\{N^{i-1} g_0\}$. Then by direct calculation, for any $v \in \ker_{i-1} N\oplus \spn{N^{i-1} g_0}$ such that $v = \sum_{p=1}^{m_{i-1}}a_p N^{i-2}e_p+bN^{i-1} g_0$ we have
$$\psi(v) \in X_{\mu^1,l} \Leftrightarrow 2b+a_1^2+\cdots+a_{m_{i-2}}^2=0, \qquad \psi(v) \in X_{\mu^2,l} \Leftrightarrow 2b+a_1^2+\cdots+a_{m_{i-2}}^2\neq 0.$$
But it is easy to see that for $m \in \N$ we have
$$\cX\left(\{ (b, a_1, \cdots, a_m) \in \A^{m+1} \mid a_1^2+\cdots+a_m^2+2b=0\}\right) = 1.$$
Therefore we have $\cX(X_{\mu^1,l}) = 1$ and $\cX(X_{\mu^2,l}) = \cX(\A^{m_{i-1}+1}) - \cX(X_{\mu^1,l}) = 0$. It completes the proof of the proposition.
\end{proof}

We combine Proposition \ref{prop:k=1} and \ref{prop:k=2}. For $\mu \vdash \dim V-4$, define
$$X_\mu \colonequals \{ \V \subset V \mid \br{\V,\V}=0, \dim \V = 2, N\V \subset \V, N|_{\V} \textup{ is regular, } N|_{\V^\perp/\V} \textup{ has Jordan type } \mu\}.$$
Then we have a decomposition $X_\mu = \sqcup_{\mu' \vdash \dim V-2}X_{\mu, \mu'}$ where
$$X_{\mu,\mu'} \colonequals \{ \V \in X_\mu \mid N|_{(N\V)^\perp/N\V} \textup{ has Jordan type } \mu'\}.$$
For $\mu' \vdash \dim V-2$ we similarly define
\begin{gather*}
X_{\mu'} = \{ l \subset V \mid \br{l,l}=0, \dim l = 1,l \subset \ker N, N|_{l^\perp/l} \textup{ has Jordan type } \mu'\},
\end{gather*}
then the natural map $X_{\mu, \mu'} \rightarrow X_{\mu'} : \V \mapsto N\V$ is a fiber bundle whose fiber at $l \in X_{\mu'}$ is $X_{\mu, l}$ defined in Proposition \ref{prop:k=2}. Thus we have
$$\cX(X_{\mu, \mu'}) = \cX(X_{\mu'})\cX(X_{\mu, l_{\mu'}}), \qquad\textup{ i.e. }\qquad \cX(X_\mu)=\sum_{\mu' \vdash \dim V -2} \cX(X_{\mu'})\cX(X_{\mu, l_{\mu'}})$$
where $l_{\mu'}$ is any element in $X_{\mu'}$. Using Proposition \ref{prop:k=1} and \ref{prop:k=2}, we get the following result.
\begin{prop} For $\mu \vdash \dim V -4$,
\begin{align*}
&\textup{if } \mu = \lambda \pch{i}{i-4} \textup{ for $i \geq 4$}, &&\cX(X_\mu) = \frac{1-(-1)^{m_i}}{2}(-1)^{m_{i-2}}
\\&\textup{if } \mu = \lambda \pch{i,i-2}{i-3,i-3} \textup{ for $i \geq 3$}, && \cX(X_\mu) = \frac{1-(-1)^{m_i}}{2}(1-(-1)^{m_{i-2}})
\\&\textup{if } \mu = \lambda \pch{i,i}{i-2,i-2} \textup{ for $i \geq 2$},&& \cX(X_\mu) = 2\floor{\frac{m_i}{2}}
\\&\textup{otherwise}, && \cX(X_\mu) = 0.
\end{align*}
\end{prop}
Combined with Lemma \ref{lem:calc}, this is now a natural consequence.
\begin{prop} Let $\sym_2 \times W' \subset W$ be the maximal parabolic subgroup where $W'$ is of the same type as $W$. If $c \in \sym_2$ is a 2-cycle, then for $\lambda=(1^{m_1}2^{m_2}\cdots)$ we have
\begin{align*}  
&\Res^W_{c\cdot W'}\ch\tsp{\lambda} = \sum_{i \geq 2, m_i \geq 2} 2\floor{\frac{m_i}{2}} \ch\tsp{\lambda\pch{i,i}{i-2, i-2}}
\\&\qquad + \sum_{i\geq 4, m_i \textup{ odd}} (-1)^{m_{i-2}}\ch \tsp{\lambda\pch{i}{i-4}} + \sum_{i\geq 3, m_i, m_{i-2} \textup{ odd}} 2 \ch\tsp{\lambda\pch{i,i-2}{i-3,i-3}}.
\end{align*}
Here  $\tsp{\lambda\pch{i,i}{i-2, i-2}}, \tsp{\lambda\pch{i}{i-4}}, \tsp{\lambda\pch{i,i-2}{i-3,i-3}}$ are total Springer representations of $W'$.
\end{prop}

\subsection{General case} \label{subsec:general} We consider the general case when $1\leq k\leq n$. For $\mu \vdash 2n-2k$, we define
$$X_\mu = \{ \V \subset V \mid \br{\V, \V}=0, \dim \V = k, N\V \subset \V, N|_\V \textup{ is regular}, N|_{\V^\perp/\V} \textup{ has Jordan type } \mu\}.$$
We again use Proposition \ref{prop:k=1} and \ref{prop:k=2}. Let $\mathcal{I}=\{(\mu^1, \mu^2, \cdots, \mu^{k-1}) \mid \mu^i \vdash 2n-2i\}.$ Then there exists a decomposition $X_\mu = \bigsqcup_{(\mu^1, \cdots, \mu^{k-1}) \in \mathcal{I}}X_{\mu, (\mu^1, \cdots, \mu^{k-1})}$ where
$$X_{\mu, (\mu^1, \cdots, \mu^{k-1})} = \{ \V \in X_{\mu} \mid N|_{(N^{k-i}\V)^\perp/(N^{k-i}\V)} \textup{ has Jordan type } \mu^i \textup{ for } 1 \leq i \leq k-1\}.$$
Also there exist natural morphisms
$$X_{\mu, (\mu^1, \cdots, \mu^{k-1})}\xrightarrow{N} X_{\mu^{k-1}, (\mu^1, \cdots, \mu^{k-2})}\xrightarrow{N} \cdots \xrightarrow{N} X_{\mu^{2}, (\mu^1)}\xrightarrow{N} X_{\mu^1}$$
each of which is a fiber bundle with fiber isomorphic to a variety defined in Proposition \ref{prop:k=2}, and $X_{\mu^1}$ is as in Proposition \ref{prop:k=1}. Now the following proposition is an easy combinatorial exercise.
\begin{prop}\label{prop:general} Let $\mu \vdash 2n-2k$. Then,
\begin{align*}
&\mu = \lambda \pch{i}{i-2k}, &&\cX(X_\mu) = \frac{1-(-1)^{m_i}}{2}(-1)^{m_{i-2}+m_{i-4}+\cdots+m_{i-2k+2}}
\\&\begin{aligned}
&\mu = \lambda \pch{i,j}{\frac{i+j}{2}-k,\frac{i+j}{2}-k} \textup{ for } i, j \textup{ s.t.}
\\&i-j \in 2\Z, 0<\frac{i-j}{2}<k\leq \frac{i+j}{2},
\end{aligned}&& \cX(X_\mu) = \frac{1-(-1)^{m_i}}{2}(-1)^{m_{i-2}+m_{i-4}+\cdots+m_{j+2}}(1-(-1)^{m_{j}})
\\&\mu = \lambda \pch{i,i}{i-k,i-k}, && \cX(X_\mu) = 2\floor{\frac{m_i}{2}}
\\&\textup{otherwise}, && \cX(X_\mu) = 0.
\end{align*}
\end{prop}

For \tyB odd \tyC even $i,j$ such that $i>j$, $\lambda/\left(\lambda\pch{i}{j}\right)$ is a horizontal strip and ${\h{\lambda/\left(\lambda\pch{i}{j}\right)}}\equiv {m_{i-2}+m_{i-4}+\cdots+m_{j+2}} \mod 2$ since $m_{i-1}, m_{i-3}, \cdots, m_{j+3},m_{j+1}$ are even. Therefore together with Lemma \ref{lem:calc} we have the following main result of this section.
\begin{thm} \label{thm:geocalc} Assume $1\leq k \leq n$. Let $\sym_k \times W' \subset W$ be the maximal parabolic subgroup where $W'$ is of the same type as $W$. If $c\in \sym_k$ is a $k$-cycle, then for $\lambda=(1^{m_1}2^{m_2}\cdots)$ we have
\begin{align*} \Res^W_{c\cdot W'} \ch\tsp{\lambda} =& \sum_{i\geq k, m_i \geq 2} 2\floor{\frac{m_i}{2}}\ch \tsp{\lambda\pch{i,i}{i-k, i-k}}
\\&+ \sum_{i\geq 2k, m_i \textup{ odd}} (-1)^{\h{\lambda/\left(\lambda\pch{i}{i-2k}\right)}}\ch \tsp{\lambda\pch{i}{i-2k}}
\\&+ \sum_{\substack{0<i-j<2k\leq i+j, \\m_i, m_j \textup{ odd}}} (-1)^{\h{\lambda/ \left(\lambda\pch{i}{j}\right)}}2\ch \tsp{\lambda\pch{i,j}{\frac{i+j}{2}-k,\frac{i+j}{2}-k}}.
\end{align*}
Here $\tsp{\lambda\pch{i,i}{i-k, i-k}}, \tsp{\lambda\pch{i}{i-2k}}, \tsp{\lambda\pch{i,j}{\frac{i+j}{2}-k,\frac{i+j}{2}-k}}$ are total Springer representations of $W'$.
\end{thm}
%

\section{Geometric properties for type $D$}\label{sec:typeD}
Here we assume $G=SO_{2n}(\C)$ for $n\geq 2$, acting on $V=\C^{2n}$ equipped with a nondegenerate symmetric bilinear form $\br{\ , \ }$. Almost every argument in the previous section is still valid in this case, but one needs to be careful since
\begin{enumerate}[label=$\bullet$]
\item $\B=G/B$ is no longer the full isotropic flag variety of $V$, and
\item for a nilpotent element $N \in \g$, the Jordan type $\lambda$ of $N$ does not determine $H^*(\B_N)$ if $\lambda$ is very even.
\end{enumerate}
Thus we need a minor modification of Theorem \ref{thm:geocalc}. From now on we state and prove such modifications.
\begin{prop}\label{thm:geocalcD1}
 Assume that $k$ is odd and $1\leq k \leq n-2$. Let $\sym_k \times W' \subset W$ be the maximal parabolic subgroup where $W'$ is of the same type as $W$ which contains both $s_+$ and $s_-$. (See Section \ref{subsec:typeDparam} for the definition of $s_+$ and $s_-$.) If $c\in \sym_k$ is a $k$-cycle and $\lambda=(1^{m_1}2^{m_2}\cdots)$, then Theorem \ref{thm:geocalc} is still valid except the following cases: if $\lambda$ is very even, then 
\begin{align*}
\Res^W_{c\cdot W'}\ch\tsp{\lambda+}=\Res^W_{c\cdot W'} \ch\tsp{\lambda-} &= \sum_{i\geq k, m_i \geq 2} m_i\ch \tsp{\lambda\pch{i,i}{i-k, i-k}}.
\end{align*}
Otherwise if $\lambda\pch{e, e}{e-k,e-k}$ is very even for some $e\geq k$, then
\begin{align*}
\Res^W_{c\cdot W'} \ch\tsp{\lambda} =&
 \sum_{i\geq k, i \neq e, m_i \geq 2} m_i \ch\tsp{\lambda\pch{i,i}{i-k,i-k}}
\\&+\ch \tsp{\lambda\pch{e,e}{e-k,e-k}+}+\ch \tsp{\lambda\pch{e,e}{e-k,e-k}-}.
\end{align*}
Here the total Springer representations on the right hand side are with respect to $W'$.
\end{prop}
\begin{proof} If there is no very even partition which appears in the formula of Theorem \ref{thm:geocalc}, then its proof is still valid. Also note that these two cases above are the only ones where a very even partition can appear in the formula for $k$ odd. But the right hand side on the first equation does not contain any very even partition, thus the formula still holds by the same reason.

Now consider the second case. Then $m_e=2$ and $m_i \equiv 0 \mod 2$ for any $i$. Also $\lambda\pch{i,i}{i-k,i-k}$ are not very even unless $i=e$. Thus by the same reason as the previous section at least we have
\begin{align*}
\Res^W_{c\cdot W'} \ch\tsp{\lambda} =&
 \sum_{i\geq k,i \neq e, m_i \geq 2} m_i\ch \tsp{\lambda\pch{i,i}{i-k,i-k}}
\\&+\alpha\ch \tsp{\lambda\pch{e,e}{e-k,e-k}+}+\beta\ch \tsp{\lambda\pch{e,e}{e-k,e-k}-}
\end{align*}
for some $\alpha, \beta$ such that $\alpha+\beta =\cX(X_{\lambda\pch{e,e}{e-k,e-k}})=2$. (See Proposition \ref{prop:general}.) Recall the involution $\tau$ defined in Section \ref{subsec:typeDparam}. Since $|\lambda| \equiv 2 \mod 4$ in our case, $\rk G$ is odd and $\tau$ fixes each conjugacy class in $W$. Thus by applying $\tau$ on both sides we have
\begin{align*}
\Res^W_{c\cdot W'} \ch\tsp{\lambda} =&
 \sum_{i\geq k, i \neq e, m_i \geq 2} m_i\ch \tau\circ \tsp{\lambda\pch{i,i}{i-k,i-k}}
\\&+\alpha \ch \tau\circ \tsp{\lambda\pch{e,e}{e-k,e-k}+}+\beta\ch \tau\circ \tsp{\lambda\pch{e,e}{e-k,e-k}-}
\\=&\sum_{i\geq k, i \neq e, m_i \geq 2} m_i\ch \tsp{\lambda\pch{i,i}{i-k,i-k}}
\\&+\alpha\ch \tsp{\lambda\pch{e,e}{e-k,e-k}-}+\beta\ch \tsp{\lambda\pch{e,e}{e-k,e-k}+}.
\end{align*}
This only happens when $\alpha=\beta=1$ since $\ch \tsp{\lambda\pch{e,e}{e-k,e-k}+}$ and $\ch\tsp{\lambda\pch{e,e}{e-k,e-k}-}$ are linearly independent. Thus the result follows.
\end{proof}

\begin{prop}\label{thm:geocalcD2}
 Assume that $k$ is even and $1\leq k \leq n-2$. Let $\sym_k \times W' \subset W$ be the maximal parabolic subgroup where $W'$ is of the same type as $W$ which contains both $s_+$ and $s_-$. (See Section \ref{subsec:typeDparam} for the definition of $s_+$ and $s_-$.) If $c\in \sym_k$ is a $k$-cycle and  $\lambda=(1^{m_1}2^{m_2}\cdots)$,  then Theorem \ref{thm:geocalc} is still valid except the following cases: if $\lambda\pch{e_1,e_2}{\frac{e_1+e_2}{2}-k,\frac{e_1+e_2}{2}-k}$ is very even for some odd $e_1,e_2$ such that $0<e_1-e_2<2k\leq e_1+e_2$, then 
\begin{align*} &\Res^W_{c\cdot W'}\ch \tsp{\lambda} = \sum_{i\geq k, m_i \geq 2} 2\floor{\frac{m_i}{2}}\ch\tsp{\lambda\pch{i,i}{i-k, i-k}}
\\&+ \sum_{i\in \{e_1, e_2\}, i\geq 2k} (-1)^{\h{\lambda/\left(\lambda\pch{i}{i-2k}\right)}}\ch \tsp{\lambda\pch{i}{i-2k}}
\\&+ (-1)^{\h{\lambda/ \left(\lambda\pch{e_1}{e_2}\right)}}\ch\left( \tsp{\lambda\pch{e_1,e_2}{\frac{e_1+e_2}{2}-k,\frac{e_1+e_2}{2}-k}+}\oplus \tsp{\lambda\pch{e_1,e_2}{\frac{e_1+e_2}{2}-k,\frac{e_1+e_2}{2}-k}-}\right).
\end{align*}
Otherwise if $\lambda$ is very even, then 
\begin{align*} \Res^W_{c\cdot W'}\ch \tsp{\lambda+} =& \sum_{i\geq k, m_i \geq 2} m_i \ch \tsp{\lambda\pch{i,i}{i-k, i-k}+},
\\\Res^W_{c\cdot W'} \ch\tsp{\lambda-} =& \sum_{i\geq k, m_i \geq 2} m_i \ch  \tsp{\lambda\pch{i,i}{i-k, i-k}-}.
\end{align*}
Here the total Springer representations on the right hand side are with respect to $W'$.
\end{prop}
\begin{proof} If there is no very even partition which appears in the formula of Theorem \ref{thm:geocalc}, then its proof is still valid. Also note that these two cases above are the only ones where a very even partition can appear in the formula if $k$ is even. The proof for the first equation is totally analogous of that of the second part in Proposition \ref{thm:geocalcD1}.

Now consider the second case for $\tsp{\lambda+}$. (Then the formula for $\tsp{\lambda-}$ can be obtained by applying $\tau$.) Then $m_i \equiv 0 \mod 2$ for any $i$ and $\lambda\pch{i,i}{i-k, i-k}$ are all very even for $i\geq k$ if $m_i \geq 2$. Thus by the same reason as the previous section at least we have
\begin{align*} \Res^W_{c\cdot W'} \ch\tsp{\lambda+} =& \sum_{i\geq k, m_i \geq 2} \left(\alpha_i \ch\tsp{\lambda\pch{i,i}{i-k, i-k}+}+\beta_i \ch \tsp{\lambda\pch{i,i}{i-k, i-k}-}\right)
\end{align*}
for some $\alpha_i, \beta_i$ such that $\alpha_i+\beta_i =\cX(X_{\lambda\pch{i,i}{i-k,i-k}})=m_i$. (See Proposition \ref{prop:general}.) Now for $w=w_{(\rho,\emptyset)-} \in W'$ for some even $\rho \vdash n-k$, the argument in Section \ref{subsec:typeDparam} implies that $cw\in W$ is parametrized by $(\rho\cup\{k\}, \emptyset)-$.

Thus if we evaluate the equation at $w=w_{(\rho, \emptyset)-}$ then
\begin{align*} 
0=\tr(w_{(\rho\cup\{k\},\emptyset)-}, \tsp{\lambda+}) =\sum_{i\geq k, m_i \geq 2}\beta_i \tr(w_{(\rho,\emptyset)-}, \tsp{\lambda\pch{i,i}{i-k, i-k}-})
\end{align*}
by Lemma \ref{lem:pm}. We set 
$$\Xi \colonequals \sum_{i\geq k, m_i \geq 2}\beta_i\left(\ch \tsp{\lambda\pch{i,i}{i-k, i-k}+}- \ch \tsp{\lambda\pch{i,i}{i-k, i-k}-}\right)$$
and show $\Xi=0$. As $\ch \tsp{\lambda\pch{i,i}{i-k, i-k}+},  \ch \tsp{\lambda\pch{i,i}{i-k, i-k}-}$ for different $i$ are linearly independent, it implies $\beta_i=0$, from which the claim follows. Indeed, if $w\in W$ is not of the form $w_{(\rho, \emptyset)\pm}$ for some $\rho \vdash n-k$ even, then obviously $\Xi(w)=0$. If $w=w_{(\rho, \emptyset)-}$ for some even $\rho$, then as $\ch \tsp{\lambda\pch{i,i}{i-k, i-k}+}(w)=0$ for any $i$, we still have $\Xi(w)=0$ from the equation above. Finally if $w=w_{(\rho, \emptyset)+}$ for some even $\rho$, then as $\Xi\circ\tau = -\Xi$, we have $\Xi(w) = (\Xi\circ\tau)(\tau(w)) = -\Xi(w_{(\rho, \emptyset)-})=0$. It suffices for the proof. 
\end{proof}

If $k=n-1$, then there is no \emph{maximal} parabolic subgroup of the form $\sym_{n-1}\times W' \subset W$. But we can still calculate the character value of total Springer representations at an $(n-1)$-cycle using the same method; we let $P \subset G$ be a (non-maximal) parabolic subgroup of $G$ such that $W_P \subset W$ is a parabolic subgroup of $W$ generated by all the simple reflections except both $s_+$ and $s_-$. Then $\cP=G/P$ can be regarded as the Grassmannian of isotropic $(n-1)$-dimensional subspaces in $V$ and there exists a natural projection $\pi_P : \B \rightarrow \cP$. Also for any $\V \in \cP_N$ such that $N|_{\V}$ is regular, $\pi_P^{-1}(\V)$ consists of only one point. Therefore we deduce the following proposition.
\begin{prop} \label{thm:geocalcD3} Let $\sym_{n-1} \subset W$ be a parabolic subgroup of $W$ and $c \in \sym_{n-1}$ be a $(n-1)$-cycle. If $n \geq 3$, then we have
\begin{align*} \ch \tsp{\lambda}(c) =& 2\delta_{\lambda, (n-1,n-1,1,1)}+2\delta_{\lambda, (n,n)}+ \delta_{\lambda,(2n-1,1)}
\\&+ \sum_{i>j\geq 1 \textup{ odd}, i+j=2n-2} 2\delta_{\lambda,(i,j,1,1)}+ \sum_{i>j>1 \textup{ odd}, i+j=2n} 2\delta_{\lambda,(i,j)}.
\end{align*}
This formula holds even when $\lambda$ is very even. If $n=2$, then $c=id$ and we have
\begin{align*} \ch \tsp{\lambda}(id) = \dim \tsp{\lambda} = 4\delta_{\lambda, (1,1,1,1)}+2\delta_{\lambda, (2,2)}+ \delta_{\lambda,(3,1)}.
\end{align*}
This formula holds even when $\lambda$ is very even, i.e. when $\lambda=(2,2)$.
\end{prop}
\begin{proof} The same proof as Theorem \ref{thm:geocalc} applies here.
\end{proof}
\begin{rmk}This theorem can be interpreted in an analogous way to Theorem \ref{thm:geocalc} if we decree that $\ch\tsp{1,1}=1$.
\end{rmk}

If $k=n$, then there are two maximal parabolic subgroups of $W$ isomorphic to $\sym_n$, i.e. $\sym_{n\pm}$ defined in Section \ref{subsec:typeDparam} (In fact, if $n=4$ then there are three possible choices. But we discard the choice of the parabolic subgroup containing both $s_+$ and $s_-$.) We define $P_+, P_- \subset G$ to be two parabolic subgroups whose Weyl groups are $\sym_{n+}, \sym_{n-}$, respectively. Then we may identify $G/P_{+}, G/P_-$ with two components of the Lagrangian Grassmannian of $V$. Also there exists a natural map $\pi_+: \B \rightarrow G/P_{+}$, $\pi_- : \B \rightarrow G/P_-$ such that for any isotropic flag $\cF \in \B$, $\pi_+(\cF), \pi_-(\cF)$ are two Lagrangian subspaces containing $\cF$.

Suppose $c_\pm \in \sym_{n\pm}$ is a $n$-cycle. If $n=2$, then by direct calculation we have
\begin{gather*}
\ch \tsp{(3,1)}(c_\pm) = 1,\qquad \ch\tsp{(2,2)\pm}(c_\pm)=2
\\\ch \tsp{(2,2)\mp}(c_\pm)=0,\qquad \ch \tsp{(1,1,1,1)}(c_\pm)=4.
\end{gather*}
Now suppose $n\geq 3$. If we follow the argument in the previous section it is required to calculate the Euler characteristic of
$$X_{\pm} \colonequals \{ \V \in G/P_\pm \mid N\V \subset \V, N|_\V \textup{ is regular}\}.$$ To that end we mimic Section \ref{subsec:general}. let $\mathcal{I}=\{(\mu^1, \mu^2, \cdots, \mu^{n-2}) \mid \mu^i \vdash 2n-2i\}.$ Then there exists a decomposition $X_\pm = \bigsqcup_{(\mu^1, \cdots, \mu^{n-2}) \in \mathcal{I}}X_{\pm, (\mu^1, \cdots, \mu^{n-2})}$ where
$$X_{\pm, (\mu^1, \cdots, \mu^{n-2})} \colonequals \{ \V \in X_{\pm} \mid N|_{(N^{k-i}\V)^\perp/(N^{k-i}\V)} \textup{ has Jordan type } \mu^i \textup{ for } 1 \leq i \leq n-2\}.$$
Also there exist natural morphisms
$$X_{\pm, (\mu^1, \cdots, \mu^{n-2})}\xrightarrow{N^2} X_{\mu^{n-2}, (\mu^1, \cdots, \mu^{n-3})}\xrightarrow{N} \cdots \xrightarrow{N} X_{\mu^{2}, (\mu^1)}\xrightarrow{N} X_{\mu^1}$$
each of which is a fiber bundle with fiber isomorphic to a variety defined in Proposition \ref{prop:k=2} except the first map, and $X_{\mu^1}$ is as in Proposition \ref{prop:k=1}. 

In order to proceed induction, it remains to calculate the Euler characteristic of a fiber of the first map. To that end, it is enough to assume $G=SO_6(\C)$ and calculate the Euler characteristic of
$$X_{\pm,l}\colonequals \{\V \in G/P_\pm \mid N\V \subset \V, N|_{\V} \textup{ is regular}, N^2\V=l\}$$
for any isotropic line $l \subset \ker N$. We list all the possibilities below.
\begin{enumerate}[label=$\bullet$]
\item $N$ has Jordan type $(5,1)$, $N|_{l^\perp/l}$ has Jordan type $(3,1)$: $\cX(X_{\pm,l})=1$.
\item $N$ has Jordan type $(3,3)$, $N|_{l^\perp/l}$ has Jordan type $(2,2)$: there are exactly two such isotropic lines and we have $\cX(X_{+,l})+\cX(X_{-,l})=1$.
\item Otherwise $\cX(X_{\pm,l})=0$.
\end{enumerate}
Now we use combinatorial induction to deduce the following.
\begin{prop} \label{thm:geocalcD4} Let $\sym_{n\pm} \subset W$ be a maximal parabolic subgroup of $W$ and $c_\pm \in \sym_{n\pm}$ be a $n$-cycle. If $n$ is even, then
\begin{gather*}
\ch \tsp{(i,j)}(c_{\pm}) = 1 \quad \textup{ if } i>j \textup{ are odd and } i+j=2n,
\\\ch \tsp{(n,n)\pm}(c_{\pm}) = 2, \qquad \ch \tsp{(n,n)\mp}(c_{\pm}) =0,
\end{gather*}
and $\ch\tsp{\lambda}(c_{\pm})=0$ otherwise. If $n$ is odd, then $c_+$ and $c_-$ are conjugate to each other and we have
\begin{gather*}
\ch \tsp{(i,j)}(c_{\pm}) = 1 \textup{  if } i>j \textup{ are odd and } i+j=2n, \qquad \ch \tsp{(n,n)}(c_{\pm}) = 1,
\end{gather*}
and $\ch \tsp{\lambda}(c_{\pm})=0$ otherwise.
\end{prop}
\begin{proof} We first calculate $\ch \tsp{\lambda}(c_+)+\ch \tsp{\lambda}(c_-)$. By the observation above and using the same method as Section \ref{subsec:general}, we see that
\begin{gather*}
\ch \tsp{\lambda}(c_+)+\ch \tsp{\lambda}(c_-) = 2\left(\delta_{\lambda,(n,n)}+\sum_{i>j \textup{ odd}, i+j=2n} \delta_{\lambda,(i,j)}\right)
\end{gather*}
which holds even when $\lambda$ is very even. Now the theorem follows from two observations: (1) if $\lambda$ is not very even, then $\ch \tsp{\lambda}(c_+)=\ch \tsp{\lambda}(c_-)$, and (2) if $\lambda$ is very even, then $\ch \tsp{\lambda\mp}(c_{\pm})=0$. (cf. Lemma \ref{lem:pm})
\end{proof}

We summarize previous propositions as follows.
\begin{thm} \label{thm:geocalcD5} Assume $1\leq k \leq n$. Let $\sym_k \times W' \subset W$ be a parabolic subgroup where $W'$ is of the same type as $W$. If $1\leq k \leq n-2$, then we choose $W'$ such that $\sym_k \times W'$ becomes a maximal parabolic subgroup of $W$. ($W'$ is trivial if $k =n-1,n$.) If $c\in \sym_k$ is a $k$-cycle and $\lambda=(1^{m_1}2^{m_2}\cdots)$ then we have
\begin{align*} 
\Res^W_{c\cdot W'} \ch\atsp{\lambda} =& \sum_{i\geq k, m_i \geq 2} 2\floor{\frac{m_i}{2}} \ch\atsp{\lambda\pch{i,i}{i-k, i-k}}
\\&+ \sum_{i\geq 2k, m_i \textup{ odd}} (-1)^{\h{\lambda/\left(\lambda\pch{i}{i-2k}\right)}}\ch \atsp{\lambda\pch{i}{i-2k}}
\\&+ \sum_{0<i-j<2k\leq i+j, m_i, m_j \textup{ odd}} (-1)^{\h{\lambda/ \left(\lambda\pch{i}{j}\right)}}2 \ch\atsp{\lambda\pch{i,j}{\frac{i+j}{2}-k,\frac{i+j}{2}-k}}.
\end{align*}
Here $\atsp{\lambda\pch{i,i}{i-k, i-k}}, \atsp{\lambda\pch{i}{i-2k}}, \atsp{\lambda\pch{i,j}{\frac{i+j}{2}-k,\frac{i+j}{2}-k}}$ are with respect to $W'$. Also, we use the convention that
$$\ch\atsp{(1,1)}=2 \qquad \text{and} \qquad \ch \atsp{\emptyset}=1.$$
\end{thm}
\begin{proof} This is just a reformulation of Proposition \ref{thm:geocalcD1}--\ref{thm:geocalcD4}.
\end{proof}

\section{Symmetric functions} \label{sec:symm}
Here, we review the theory of symmetric functions and how it is related to the representation theory of Weyl groups for type $A$ and $BC$. For more details, we refer readers to \cite{zel81}, \cite{sta86}, \cite{mac95}, etc.

\subsection{Symmetric functions and symmetric groups} Let $\Lambda = \Lambda(x)$ be the ring of symmetric functions of infinitely many variables $x=(x_1, x_2, \cdots)$ over $\Q$. For any $\Q$-algebra $R$, we write $\Lambda_R \colonequals \Lambda\otimes_\Q R$. Define $\Lambda^n$ to be the homogeneous part of degree $n$ in $\Lambda$ so that $\Lambda = \bigoplus_{n\geq 0}\Lambda^n$. There are series of symmetric functions parametrized by partitions, some of which we recall in the following.
\begin{enumerate}[label=$\bullet$]
\item $h_\lambda=h_\lambda(x)$: (complete) homogeneous symmetric functions \cite[Chapter 7.5]{sta86}.
\item $p_\lambda=p_\lambda(x)$: power sum symmetric functions \cite[Chapter 7.7]{sta86}.
\item $s_\lambda=s_\lambda(x)$: Schur functions \cite[Chapter 7.10, 7.15]{sta86}.
\end{enumerate}
Each series forms a $\Q$-basis of $\Lambda$.

There exists a non-degenerate symmetric bilinear form $\br{\ ,\ }: \Lambda \times \Lambda \rightarrow \Q$ such that $\br{p_\lambda, p_\mu} = \delta_{\lambda, \mu} z_\lambda$ and $\br{s_\lambda, s_\mu} = \delta_{\lambda, \mu}$ (ref. \cite[Chapter 7.9]{sta86}). Also there exists a ring involution $\omega: \Lambda \rightarrow \Lambda$ satisfying $\omega(p_\lambda) =(-1)^{|\lambda|-l(\lambda)}p_\lambda$ and $\omega(s_\lambda) = s_{\lambda'}$ (ref. \cite[Chapter 7.6]{sta86}). It is obvious that $\omega$ is an isometry with respect to $\br{\ , \ }$.

Recall the Frobenius characteristic map (ref. \cite[Chapter 7.18]{sta86})
$$\Psi : \bigoplus_{n\geq 0} \cf{\sym_n} \rightarrow \Lambda$$
which sends $f\in \cf{\sym_n}$  to $\frac{1}{n!}\sum_{w\in \sym_n}f(w)p_{\rho(w)}\in \Lambda^n$, where $\rho(w)$ is the cycle type of $w$. It is a $\Q$-linear isomorphism and also an isometry. Moreover, it satisfies
$$\Psi(\Ind^{\sym_{m+n}}_{\sym_m \times \sym_n}f \times g)=\Psi(f)\Psi(g)$$
for any $f \in \cf{\sym_m}$ and $g \in \cf{\sym_n}$, and
\begin{enumerate}[label=$\bullet$]
\item $\Psi^{-1}(h_\lambda)= \ch \left(\Ind^{\sym_{|\lambda|}}_{\sym_{\lambda_1}\times \cdots \times \sym_{\lambda_{l(\lambda)}}} \Id_{\sym_{\lambda_1}\times \cdots \times \sym_{\lambda_{l(\lambda)}}}\right)$,
\item $\Psi^{-1}(p_\lambda)(w_\rho) = \delta_{\lambda,\rho}z_\lambda$, and
\item $\Psi^{-1}(s_\lambda)=\ch \chi^\lambda$.
\end{enumerate}

\subsection{Hall-Littlewood functions, Green and Kostka-Foulkes polynomials}
Here, we consider $\Lambda_{\Q[t]}$ and recall the Hall-Littlewood functions. (ref. \cite[Chapter III.2]{mac95})
\begin{enumerate}[label=$\bullet$]
\item $P_\lambda(t)=P_\lambda(x;t)$: the Hall-Littlewood $P$-functions
\item $Q_\lambda(t)=Q_\lambda(x;t)$: the Hall-Littlewood $Q$-functions
\end{enumerate}
Then $\{P_{\lambda}(t)\}$ forms a $\Q[t]$-basis of $\Lambda_{\Q[t]}$, but $\{Q_{\lambda}(t)\}$ does not. (It becomes a basis if we base change $\lambda_{\Q[t]}$ to $\Q(t)$.) There exists a non-degenerate symmetric bilinear form $\brt{\ ,\ } : \Lambda_{\Q[t]}\times \Lambda_{\Q[t]} \rightarrow \Q(t)$ defined by $\brt{P_\lambda(t), Q_\mu(t)}=\delta_{\lambda,\mu}$ and $\brt{p_\lambda, p_\mu} = z_\lambda(t) \delta_{\lambda,\mu}$ where 
$$z_\lambda(t) \colonequals \prod_{i\geq 1}\frac{i^{m_i}m_i!}{ (1-t^{i})^{m_i}} \textup{ for } \lambda = (1^{m_1}2^{m_2}\cdots).$$
As $z_\lambda(0) = z_\lambda$, the former scalar product $\br{\ ,\ }$ can be regarded as $\brt{\ ,\ }_{t=0}$.

There is an endomorphism $\mho: \Lambda_{\Q(t)} \rightarrow \Lambda_{\Q(t)}$
which maps $f(x)$ to $f(x/(1-t))$ in Lambda-ring notation, i.e. which substitutes variables $\{x_i \mid i\geq 1\}$ with $\{x_i t^j\mid i\geq 1, j\geq 0\}.$ In other words, $\mho$ is characterized by $\mho(t)=t$ and $\mho(p_k) = \frac{1}{1-t^k}p_k$. We define  $Q'_\lambda(t) = Q'_\lambda(x;t) \colonequals \mho(Q_\lambda(t))$, called the modified Hall-Littlewood $Q'$-functions (ref. \cite[pp. 234--236]{mac95}). Then in fact $Q'_\lambda(t) \in \Q[t]$ and $\br{P_\lambda(t), Q'_\mu(t)}= \delta_{\lambda, \mu}$, i.e. $\{P_\lambda(t)\}$ and $\{Q'_\lambda(t)\}$ are dual $\Q[t]$-bases of $\Lambda_{\Q[t]}$ with respect to the scalar product $\br{\ ,\ }$. 

The function $Q'_{\lambda}(t)$ enjoys the following properties. For partitions $\lambda$ and $\rho$, we define a polynomial $X^\lambda_\mu(t) \in \Q[t]$ to be the ``transition matrix" between bases $\{p_\rho\}$ and $\{P_\lambda(t)\}$, i.e. such that $p_\rho = \sum_\lambda X^\lambda_\rho(t) P_\lambda(t)$. Then we have $\br{p_\rho, Q'_\lambda(t)} = X^\lambda_\rho(t)$, i.e. $Q'_\lambda(t) = \sum_\rho z_\rho^{-1} X^\lambda_\rho(t)p_\rho.$ If we define $\gr^\lambda_\rho(t) \colonequals t^{b(\lambda)}X^\lambda_\rho(t^{-1})$, then indeed $\gr^\lambda_\rho(t)\in \Q[t]$, which we call the Green polynomial originally defined in \cite{gre55}. Now the equation above reads
$$Q'_\lambda(t) = t^{b(\lambda)} \sum_\rho z_\rho^{-1} \gr^\lambda_\rho(t^{-1})p_\rho.$$
 On the other hand, if we write $K_{\mu, \lambda}(t)$ to be the Kostka-Foulkes polynomial, then we have $s_\mu = \sum_{\lambda}K_{\mu, \lambda}(t) P_\lambda(t)$ and thus $\br{s_\mu, Q'_\lambda(t)} = K_{\mu, \lambda}(t)$, i.e. 
$$Q'_\lambda(t) = \sum_\mu K_{\mu,\lambda}(t)s_\mu.$$

\subsection{Product, coproduct, Frobenius and Verschiebung operator}
Let $y=(y_1, y_2, \cdots)$ be another set of indeterminates. 
Define the coproduct map $\Delta: \Lambda \rightarrow \Lambda(x) \otimes \Lambda(y) : f \mapsto f(x,y)$ which substitutes $x$ with $x \sqcup y$. Also define the product map $\nabla :  \Lambda(x) \otimes \Lambda(y)  \rightarrow \Lambda$ which sends $f(x)g(y)$ to $fg=f(x)g(x)$. Note that there is a naturally induced scalar product $\br{\ , \ }$ on $\Lambda(x) \otimes \Lambda(y)$ defined by $\br{a(x)b(y), c(x)d(y)} = \br{a,c}\br{b,d}$. Then we have an adjunction $\br{\Delta(f), g} = \br{f, \nabla(g)}$. A similar results also holds if we consider $\Delta: \Lambda_{\Q[t]} \rightarrow \Lambda(x)_{\Q[t]} \otimes \Lambda(y)_{\Q[t]}, \nabla: \Lambda(x)_{\Q[t]} \otimes \Lambda(y)_{\Q[t]} \rightarrow \Lambda_{\Q[t]} $ and the scalar product $\brt{\ , \ }$ instead.

We call $\Lambda \rightarrow \Lambda: f \mapsto f[p_2]$ the (second) Frobenius operator where $f[g]$ denotes the plethysm, that is $f[p_2]$ is obtained by substituting $\{x_1, x_2, \cdots\}$ with $\{x_1^2, x_2^2, \cdots\}$. Also we define the (second) Verschiebung operator $\versch: \Lambda \rightarrow \Lambda$ which is characterized by $\versch(h_{2n}) = h_n$ and $\versch(h_{2n+1})=0$. Then there is an adjunction $\br{f[p_2],g}=\br{f,\versch(g)}$ for $f,g \in \Lambda$.

\subsection{Hyperoctahedral groups} Suppose $W_n$ is the Weyl group of type $BC_n$, i.e. $W_n=\sym_n \ltimes (\Z/2)^n$. Then similar to symmetric groups there exists the Frobenius characteristic map
$$\Psi: \bigoplus_{n\geq 0} \cf{W_n} \rightarrow \Lambda(x) \otimes \Lambda(y)$$
defined as follows. For partitions $\alpha$ and $\beta$ we define
$$p_{(\alpha, \beta)} = p_{(\alpha, \beta)}(x,y) \colonequals \left(\prod_{i\geq 1}(p_{\alpha_i}(x)+p_{\alpha_i}(y))\right)\left(\prod_{j\geq 1}(p_{\beta_j}(x)-p_{\beta_j}(y))\right).$$
Then $f \in \cf{W_n}$ is mapped under $\Psi$ to $\frac{1}{n!2^n} \sum_{w\in W_n}f(w)p_{\rho(w)}$, where $\rho(w)$ is the bipartition which parametrizes the conjugacy class of $w \in W_n$ (see Section \ref{sec:param}). It is a $\Q$-linear isomorphism and also an isometry. Moreover, it satisfies
$$\Psi(\Ind^{W_{m+n}}_{W_m \times W_n}f \times g)=\Psi(f)\Psi(g)$$
for any $f \in \cf{W_m}$ and $g \in \cf{W_n}$, and
\begin{enumerate}[label=$\bullet$]
\item $\Psi^{-1}(h_\lambda(x))= \ch \left(\Ind^{W_{|\lambda|}}_{W_{\lambda_1}\times \cdots \times W_{\lambda_{l(\lambda)}}} \Id_{W_{\lambda_1}\times \cdots \times W_{\lambda_{l(\lambda)}}}\right)$, 
\item $\Psi^{-1}(h_\lambda(y))= \ch \left(\Ind^{W_{|\lambda|}}_{W_{\lambda_1}\times \cdots \times W_{\lambda_{l(\lambda)}}} \widetilde{\textup{Sign}}_{W_{\lambda_1}\times \cdots \times W_{\lambda_{l(\lambda)}}}\right)$, where $\widetilde{\textup{Sign}}_{W_{\lambda_1}\times \cdots \times W_{\lambda_{l(\lambda)}}}$ is the tensor product of $\sgn_{W_{\lambda_1}\times \cdots \times W_{\lambda_{l(\lambda)}}}$ and the pullback of $\sgn_{\sym_{\lambda_1}\times \cdots \times \sym_{\lambda_{l(\lambda)}}}$ under the canonical surjection $W_{\lambda_1}\times \cdots \times W_{\lambda_{l(\lambda)}} \twoheadrightarrow \sym_{\lambda_1}\times \cdots \times \sym_{\lambda_{l(\lambda)}}$,
\item $\Psi^{-1}(p_{(\alpha, \beta)})(w_{(\rho, \sigma)}) = \delta_{\alpha,\rho}\delta_{\beta,\sigma}z_\alpha z_\beta 2^{l(\alpha)+l(\beta)}$, and
\item $\Psi^{-1}(s_\alpha(x)s_\beta(y))=\ch \chi^{(\alpha, \beta)}$.
\end{enumerate}

We often identify $\sym_n \ltimes \{id\} \subset W_n$ with $\sym_n$. Then for $\varrho \in \mathcal{R}(\sym_n)$ and $\varphi \in \mathcal{R}(W_n)$ we have (ref. \cite[7.10]{zel81})
$$\Psi\left(\Ind_{\sym_n}^{W_n} \varrho\right) = \Delta(\Psi(\varrho)), \quad \Psi\left(\Res_{\sym_n}^{W_n} \varphi\right) = \nabla(\Psi(\varphi)).$$
Thus in particular, $\Psi\left( \Ind_{\sym_n}^{W_n} \ch\Id_{\sym_n}\right) = \Delta(h_n) = \sum_{i=0}^n h_i(x)h_{n-i}(y)$ and $\Psi\left(\Res_{\sym_n}^{W_n} \ch\chi^{(\alpha, \beta)}\right) = \nabla(s_\alpha(x)s_\beta(y)) = s_\alpha s_\beta$.

\section{Green and Kostka-Foulkes polynomials} \label{sec:greenkostka}
Our goal in this section is to prove Theorem \ref{thm:grcalc} and \ref{thm:grcalc2} regarding Green and Kostka-Foulkes polynomials.
Let $p_k^{*,\brt{,}}: \Lambda_{\Q(t)} \rightarrow \Lambda_{\Q(t)}$ be the operator satisfying $\brt{p_k^{*,\brt{,}}f,g}=\brt{f,p_kg}$, that is the adjoint to multiplication by $p_k$ with respect to $\brt{\ ,\ }$. It is indeed easy to check that $p_k^{*,\brt{,}} = \frac{k}{1-t^k}\frac{\del}{\del p_k}$ if we identify $\Lambda_{\Q(t)}$ with $\Q(t)[p_1, p_2, \cdots]$. 

For an indeterminate $u$ we define
$$\genQ(u) \colonequals \prod_{i \geq 1} \frac{1-x_i t u}{1-x_i u}, \qquad \genF(u) \colonequals \frac{1-u}{1-tu}$$
as elements in $\Lambda_{\Q[t]}[[u]]$. Then by \cite[Chapter III.2, (2.15)]{mac95}, $Q_\lambda(t)$ is the coefficient of $u_1^{\lambda_1}u_2^{\lambda_2} \cdots$ in the expansion of $\prod_{i\geq 1} \genQ(u_i) \prod_{i<j} \genF(u_i^{-1}u_j).$
Even when $\lambda$ is a finite sequence of integers which is not necessarily a partition, we may define $Q_{\lambda}(t)$ as a coefficient of $u_1^{\lambda_1}u_2^{\lambda_2} \cdots$ in its expansion.

Now we claim (cf. \cite[p.237]{mac95})
$$\genQ(u) = \exp\left(\sum_{n\geq1} \frac{1-t^n}{n}p_n u^n\right), \textup{ i.e. } \log \genQ(u) = \sum_{n\geq1} \frac{1-t^n}{n}p_n u^n$$
as elements in $\Lambda_{\Q(t)}[[u]]$. But it is clear since
\begin{align*}
\log \genQ(u) &= \sum_{i \geq 1} \left(\log(1-x_i t u)- \log(1-x_i u)\right) = \sum_{i \geq 1} \left(\sum_{j \geq 1} \frac{(x_i u)^j}{j}-\frac{(x_i t u)^j}{j}\right)
\\&= \sum_{i \geq 1} \sum_{j \geq 1} \frac{(x_iu)^j}{j} (1-t^j) = \sum_{j \geq 1} \frac{p_i u^j}{j} (1-t^j).
\end{align*}
Thus we have $p_k^{*,\brt{,}}  \genQ(u)=u^n\genQ(u)$ and it follows that
$$p_k^{*,\brt{,}} \prod_{i\geq 1} \genQ(u_i) \prod_{i<j} \genF(u_i^{-1}u_j)=p_k(u_1, u_2, \cdots)\prod_{i\geq 1} \genQ(u_i) \prod_{i<j} \genF(u_i^{-1}u_j),$$
where $p_k(u_1, u_2, \cdots)$ is the power symmetric function of variables $u_1, u_2, \cdots$. Therefore,
$$p_k^{*,\brt{,}} Q_\lambda(t) = \sum_{i\geq 1} Q_{\lambda-ke_i}(t).$$
Here, we regard $\lambda=(\lambda_1, \lambda_2, \cdots)$ as an element of $\bigoplus_{i\geq 1}\Z$ and let $e_i \in \bigoplus_{i\geq 1}\Z$ be an element whose $k$-th coordinate is 1 and 0 elsewhere. (Note that $\lambda-ke_i$ is not necessarily a partition. Also the right side of the identity is a finite sum since $Q_{\lambda-ke_i}(t)=0$ when $i>l(\lambda)$.)

There is a way to inductively define $Q_\lambda(t)$ for any finite integer sequence $\lambda\in\bigoplus_{i\geq 1}\Z$ if it is known for any partition. According to \cite[p.214]{mac95}, for any $s<r$ we have
\begin{align*}
\textup{ if $r-s$ is odd, }\quad Q_{(s,r)}(t) =& tQ_{(r,s)}(t)+ \sum_{i=1}^{\frac{r-s-1}{2}}(t^{i+1}-t^{i-1})Q_{(r-i, s+i)}(t), \textup{ and}
\\\textup{ if $r-s$ is even, }\quad Q_{(s,r)}(t) =& tQ_{(r,s)}(t)+ \sum_{i=1}^{\frac{r-s-2}{2}}(t^{i+1}-t^{i-1})Q_{(r-i, s+i)}(t)
\\&+(t^{\frac{r-s}{2}}-t^{\frac{r-s-2}{2}})Q_{(\frac{r-s}{2}, \frac{r-s}{2})}(t).
\end{align*}
A similar formula holds for any finite sequence of integers. From this, the following lemma can be easily verified. From now on, we write $\dummy$ to denote some $\Q[t]$-linear combination of the Hall-Littlewood $Q$-functions that we do not specify.

\begin{lem} Let $\lambda$ be a partition and $k\geq 1$.
\begin{enumerate}[label=\textup{(\arabic*)}]
\item Let $\lambda = (r^{m_r})$ for some $m_r \geq 0$. Then 
\begin{align*}
\textup{$m_r$ even}, &&&\sum_{i=1}^{m_r} Q_{\lambda-2ke_i}(t) = (-1)^km_rQ_{\lambda\pch{r,r}{r-k,r-k}}(t) +(t+1)\dummy
\\\textup{$m_r$ odd}, &&&\sum_{i=1}^{m_r} Q_{\lambda-2ke_i}(t) = (-1)^k(m_r-1)Q_{\lambda\pch{r,r}{r-k,r-k}}(t) + Q_{\lambda\pch{r}{r-2k}}(t) + (t+1)\dummy
\end{align*}
\item Let $\lambda = (r^{m_r+1})$ for some $m_r \geq 0$. Then,
\begin{align*}
\textup{$m_r$ even}, &&&Q_{\lambda-2ke_1}(t) = Q_{\lambda\pch{r}{r-2k}}(t)  +(t+1)\dummy
\\&&&Q_{\lambda-(2k-1)e_1}(t) = Q_{\lambda\pch{r}{r-2k+1}}(t)  +(t+1)\dummy
\\\textup{$m_r$ odd}, &&&Q_{\lambda-2ke_1}(t) = (-1)^{k}2 Q_{\lambda\pch{r,r}{r-k,r-k}}(t) - Q_{\lambda\pch{r}{r-2k}}(t)+ (t+1) \dummy
\\&&&Q_{\lambda-(2k-1)e_1}(t) = -Q_{\lambda\pch{r}{r-2k+1}}(t)  +(t+1)\dummy
\end{align*}
\item Let $\lambda = (r^{m_r+2})$ for some $m_r\geq 0$. Then,
\begin{align*}
Q_{\lambda-k(e_1+e_2)}(t) = Q_{\lambda\pch{r,r}{r-k,r-k}}(t) +(t+1) \dummy
\end{align*}
\end{enumerate}
\end{lem}

By applying the above lemma repeatedly, for any partition $\lambda$ we easily obtain
\begin{align*}
&p_{2k}^{*,\brt{,}}Q_\lambda(t) = \sum_{i\geq 1} Q_{\lambda-2ke_i}(t)
\\=& \sum_{i\geq k, m_i \geq 2} (-1)^k2\floor{\frac{m_i}{2}}Q_{\lambda\pch{i,i}{i-k, i-k}}(t)+\sum_{i\geq 2k, m_i \textup{ odd}} (-1)^{\h{\lambda/ \lambda\pch{i}{i-2k}}}Q_{\lambda\pch{i}{i-2k}}(t)
\\&+ \sum_{\substack{i-j\in 2\Z, 0<i-j<2k\leq i+j, \\m_i, m_j \textup{ odd}}} (-1)^{\h{\lambda/ \lambda\pch{i}{j}}}(-1)^{\frac{2k-(i-j)}{2}}2Q_{\lambda\pch{i,j}{\frac{i+j}{2}-k,\frac{i+j}{2}-k}}(t) 
\\&+(t+1)\dummy.
\end{align*}
We need the following lemma to simplify the signs that appear in the equation above.
\begin{lem} Let $\lambda$ be a partition and $\lambda' = (\lambda'_1, \lambda'_2, \cdots)$ be its conjugate. Also assume $k\geq 1$. 
\begin{enumerate}[label=\textup{(\arabic*)}]
\item For $i\geq k$, $b(\lambda) \equiv b(\lambda\pch{i}{i-k}) + \sum_{a=i-k+1}^i \lambda_a' +k\mod 2$.
\item For $i\geq k$, $b(\lambda) \equiv b(\lambda\pch{i,i}{i-k,i-k}) +k \mod 2$.
\item For $0<i-j<2k\leq i+j$ such that $i-j \in 2\Z_{>0}$, 
$$b(\lambda) \equiv b\left(\lambda\pch{i,j}{\frac{i+j}{2}-k,\frac{i+j}{2}-k}\right) +\sum_{a=j+1}^i \lambda'_a+\frac{2k-(i-j)}{2} \mod 2.$$
\end{enumerate}
\end{lem}
\begin{proof} Since $\binom{n}{2}-\binom{n-1}{2} = n-1, $ we have
$$b(\lambda)- b(\lambda\pch{i}{i-k}) = \sum_{a=i-k+1}^i \binom{\lambda_a'}{2}-\binom{\lambda_a'-1}{2}= \sum_{a=i-k+1}^i (\lambda_a'-1) \equiv \sum_{a=i-k+1}^i \lambda_a'+k\mod 2.$$
Now the rest follows directly from (1).
\end{proof}
Therefore, we have
\begin{align*}
&p_{2k}^{*,\brt{,}}\sg{\lambda}Q_\lambda(t) = \sg{\lambda}\sum_{i\geq 1} Q_{\lambda-2ke_i}(t)
\\=& \sum_{i\geq k, m_i \geq 2} \sg{\lambda\pch{i,i}{i-k, i-k}}2\floor{\frac{m_i}{2}}Q_{\lambda\pch{i,i}{i-k, i-k}}(t)
\\&+\sum_{i\geq 2k, m_i \textup{ odd}} \sg{\lambda\pch{i}{i-2k}}(-1)^{\h{\lambda/ \lambda\pch{i}{i-2k}}+\sum_{a=i-2k+1}^i \lambda_a'}Q_{\lambda\pch{i}{i-2k}}(t)
\\&+ \sum_{\substack{i-j \in 2\Z, 0<i-j<2k\leq i+j, \\m_i, m_j \textup{ odd}}}\sg{\lambda\pch{i,j}{\frac{i+j}{2}-k,\frac{i+j}{2}-k}} (-1)^{\h{\lambda/ \lambda\pch{i}{j}}+\sum_{a=j+1}^i \lambda_a'}2Q_{\lambda\pch{i,j}{\frac{i+j}{2}-k,\frac{i+j}{2}-k}}(t) 
\\&+(t+1)\dummy.
\end{align*}
We apply $\mho$ on both sides (see Section \ref{sec:symm} for the definition of $\mho$). As mentioned previously, the Hall-Littlewood $Q$-functions become the modified Hall-Littlewood $Q'$-functions. We define $p_k^{*,\br{,}}$ to be the adjoint operator to the multiplication by $p_k$ with respect to $\br{\ ,\ }$. Then it is easy to see that $p_k^{*,\br{,}} = k \frac{\partial}{\partial p_k}$ and we have $\mho\circ p_{k}^{*, \brt{,}} =p_{k}^{*, \br{,}}\circ \mho$. (It easily follows from the chain rule.)
Therefore, we have
\begin{align*}
&p_{2k}^{*,\br{,}}\sg{\lambda}Q'_\lambda(t) 
\\=& \sum_{i\geq k, m_i \geq 2} \sg{\lambda\pch{i,i}{i-k, i-k}}2\floor{\frac{m_i}{2}}Q'_{\lambda\pch{i,i}{i-k, i-k}}(t)
\\&+\sum_{i\geq 2k, m_i \textup{ odd}} \sg{\lambda\pch{i}{i-2k}}(-1)^{\h{\lambda/ \lambda\pch{i}{i-2k}}+\sum_{a=i-2k+1}^i \lambda_a'}Q'_{\lambda\pch{i}{i-2k}}(t)
\\&+ \sum_{\substack{i-j \in 2\Z, 0<i-j<2k\leq i+j, \\m_i, m_j \textup{ odd}}}\sg{\lambda\pch{i,j}{\frac{i+j}{2}-k,\frac{i+j}{2}-k}} (-1)^{\h{\lambda/ \lambda\pch{i}{j}}+\sum_{a=j+1}^i \lambda_a'}2Q'_{\lambda\pch{i,j}{\frac{i+j}{2}-k,\frac{i+j}{2}-k}}(t) 
\\&+(t+1)\dummyy
\end{align*}
where $\dummyy$ is some $\Q[t]$-linear combination of $Q'$-functions. Note that the equation above is defined in $\Lambda_{\Q[t]}$. Thus we can evaluate both sides at $t=-1$, which gives
\begin{equation}\label{eq:hlQ'}
\begin{aligned}
&p_{2k}^{*,\br{,}}\sg{\lambda}Q'_\lambda(-1) 
\\=& \sum_{i\geq k, m_i \geq 2} \sg{\lambda\pch{i,i}{i-k, i-k}}2\floor{\frac{m_i}{2}}Q'_{\lambda\pch{i,i}{i-k, i-k}}(-1)
\\&+\sum_{i\geq 2k, m_i \textup{ odd}} \sg{\lambda\pch{i}{i-2k}}(-1)^{\h{\lambda/ \lambda\pch{i}{i-2k}}+\sum_{a=i-2k+1}^i \lambda_a'}Q'_{\lambda\pch{i}{i-2k}}(-1)
\\&+ \sum_{\substack{i-j \in 2\Z, 0<i-j<2k\leq i+j, \\m_i, m_j \textup{ odd}}}\sg{\lambda\pch{i,j}{\frac{i+j}{2}-k,\frac{i+j}{2}-k}} (-1)^{\h{\lambda/ \lambda\pch{i}{j}}+\sum_{a=j+1}^i \lambda_a'}2Q'_{\lambda\pch{i,j}{\frac{i+j}{2}-k,\frac{i+j}{2}-k}}(-1).
\end{aligned}
\end{equation} 
If we take the scalar product $\br{\ ,\ }$ on both sides with $p_\rho$ for some $\rho \vdash |\lambda|-2k$, then the following proposition is a direct outcome.
\begin{prop} \label{prop:grcalc} For a partition $\lambda=(1^{m_1}2^{m_2}\cdots)$ with its transpose $\lambda' = (\lambda'_1, \lambda_2', \cdots)$ and $k \geq 1$, we have
\begin{align*}
\gr^\lambda_{\rho \cup (2k)}(-1)=& \sum_{i\geq k, m_i \geq 2} 2\floor{\frac{m_i}{2}}\gr^{\lambda\pch{i,i}{i-k, i-k}}_{\rho}(-1)
\\& + \sum_{i\geq 2k, m_i \textup{ odd}} (-1)^{\h{\lambda/ \lambda\pch{i}{i-2k}}+\sum_{a=i-2k+1}^i \lambda_a'}\gr^{\lambda\pch{i}{i-2k}}_\rho(-1)
\\& + \sum_{\substack{i-j \in 2\Z, 0<i-j<2k\leq i+j, \\m_i, m_j \textup{ odd}}} (-1)^{\h{\lambda/ \lambda\pch{i}{j}}+\sum_{a=j+1}^i \lambda_a'}2\gr^{\lambda\pch{i,j}{(i+j)/{2}-k,(i+j)/{2}-k}}_\rho(-1).
\end{align*}
\end{prop} 
On the other hand, if we take the scalar product $\br{\ ,\ }$ on both sides of (\ref{eq:hlQ'}) with $s_\nu$ for some $\nu\vdash |\lambda|-2k$, then 
\begin{align*}
&\sg{\lambda}\br{Q'_\lambda(-1),p_{2k}s_\nu}
\\=& \sum_{i\geq k, m_i \geq 2} \sg{\lambda\pch{i,i}{i-k, i-k}}2\floor{\frac{m_i}{2}}K_{\nu,\lambda\pch{i,i}{i-k, i-k}}(-1)
\\&+\sum_{i\geq 2k, m_i \textup{ odd}} \sg{\lambda\pch{i}{i-2k}}(-1)^{\h{\lambda/ \lambda\pch{i}{i-2k}}+\sum_{a=i-2k+1}^i \lambda_a'}K_{\nu,\lambda\pch{i}{i-2k}}(-1)
\\&+ \sum_{\substack{i-j \in 2\Z, 0<i-j<2k\leq i+j, \\m_i, m_j \textup{ odd}}}\sg{\lambda\pch{i,j}{\frac{i+j}{2}-k,\frac{i+j}{2}-k}} (-1)^{\h{\lambda/ \lambda\pch{i}{j}}+\sum_{a=j+1}^i \lambda_a'}2K_{\nu,\lambda\pch{i,j}{\frac{i+j}{2}-k,\frac{i+j}{2}-k}}(-1).
\end{align*}
By the Murnaghan-Nakayama rule (cf. \cite[7.17]{sta86}), we have
\begin{prop}For a partition $\lambda=(1^{m_1}2^{m_2}\cdots)$ with its transpose $\lambda' = (\lambda'_1, \lambda_2', \cdots)$ and $k \geq 1$, we have
\begin{align*}
&\sg{\lambda}\sum_{\mu \vdash n, \mu/\nu \textup{ border strip}} (-1)^{\h{\mu/\nu}}K_{\mu, \lambda}(-1)
\\=& \sum_{i\geq k, m_i \geq 2} \sg{\lambda\pch{i,i}{i-k, i-k}}2\floor{\frac{m_i}{2}}K_{\nu,\lambda\pch{i,i}{i-k, i-k}}(-1)
\\&+\sum_{i\geq 2k, m_i \textup{ odd}} \sg{\lambda\pch{i}{i-2k}}(-1)^{\h{\lambda/ \lambda\pch{i}{i-2k}}+\sum_{a=i-2k+1}^i \lambda_a'}K_{\nu,\lambda\pch{i}{i-2k}}(-1)
\\&+ \sum_{\substack{i-j \in 2\Z, 0<i-j<2k\leq i+j, \\m_i, m_j \textup{ odd}}}\sg{\lambda\pch{i,j}{\frac{i+j}{2}-k,\frac{i+j}{2}-k}} (-1)^{\h{\lambda/ \lambda\pch{i}{j}}+\sum_{a=j+1}^i \lambda_a'}2K_{\nu,\lambda\pch{i,j}{\frac{i+j}{2}-k,\frac{i+j}{2}-k}}(-1).
\end{align*}
\end{prop}

From now on, we suppose that $\lambda$ is a Jordan type of some nilpotent element in $\g$ of type $B, C, $ or $D$. Then case-by-case check shows that $\sum_{a=i-2k+1}^i \lambda_a' \equiv 0 \mod 2$ for any choice of $i$ such that $m_i$ is odd. Similarly, we always have $\sum_{a=j+1}^i \lambda_a' \equiv 0 \mod 2$ for $i,j$ such that $i>j$ and $m_i, m_j$ are odd. On the other hand, if $m_i$ and $m_j$ are odd then we always have $i-j \in 2\Z$. Therefore, we can simplify the above propositions a bit as follows. 
\begin{thm} \label{thm:grcalc} Assume $\lambda=(1^{m_1}2^{m_2}\cdots)$ is a partition of the Jordan type of some nilpotent element in $\g$ of type $B$, $C$, or $D$ and let $k \geq 1$. Then we have
\begin{align*}
\gr^\lambda_{\rho \cup (2k)}(-1)=& \sum_{i\geq k, m_i \geq 2} 2\floor{\frac{m_i}{2}}\gr^{\lambda\pch{i,i}{i-k, i-k}}_{\rho}(-1)
\\& + \sum_{i\geq 2k, m_i \textup{ odd}} (-1)^{\h{\lambda/ \lambda\pch{i}{i-2k}}}\gr^{\lambda\pch{i}{i-2k}}_\rho(-1)
\\& + \sum_{\substack{ 0<i-j<2k\leq i+j, \\m_i, m_j \textup{ odd}}} (-1)^{\h{\lambda/ \lambda\pch{i}{j}}}2\gr^{\lambda\pch{i,j}{\frac{i+j}{2}-k,\frac{i+j}{2}-k}}_\rho(-1).
\end{align*}
\end{thm} 
\begin{thm} \label{thm:grcalc2} Under the same assumptions above, we have
\begin{align*}
&\sg{\lambda}\sum_{\mu \vdash n, \mu/\nu \textup{ border strip}} (-1)^{\h{\mu/\nu}}K_{\mu, \lambda}(-1)
\\=& \sum_{i\geq k, m_i \geq 2} \sg{\lambda\pch{i,i}{i-k, i-k}}2\floor{\frac{m_i}{2}}K_{\nu,\lambda\pch{i,i}{i-k, i-k}}(-1)
\\&+\sum_{i\geq 2k, m_i \textup{ odd}} \sg{\lambda\pch{i}{i-2k}}(-1)^{\h{\lambda/ \lambda\pch{i}{i-2k}}}K_{\nu,\lambda\pch{i}{i-2k}}(-1)
\\&+ \sum_{\substack{ 0<i-j<2k\leq i+j, \\m_i, m_j \textup{ odd}}}\sg{\lambda\pch{i,j}{\frac{i+j}{2}-k,\frac{i+j}{2}-k}} (-1)^{\h{\lambda/ \lambda\pch{i}{j}}}2K_{\nu,\lambda\pch{i,j}{\frac{i+j}{2}-k,\frac{i+j}{2}-k}}(-1).
\end{align*}
\end{thm}

\section{Proof of Main Theorem \ref{mainthm2}} \label{sec:proofmain2}
We are ready to prove Main Theorem \ref{mainthm2}. Indeed it is easily deduced from Theorem \ref{thm:geocalc}, \ref{thm:geocalcD5}, and \ref{thm:grcalc}, as we see below.
\begin{thm} \label{thm:mainBC} Suppse $G=SO_{2n+1}(\C)$ or $Sp_{2n}(\C)$. Let $\sym_n \subset W$ be the maximal parabolic subgroup of $W$ and let $w_\rho \in \sym_n$ be an element of cycle type $\rho\vdash n$.
\begin{enumerate}[label=$\bullet$]
\item If $G=SO_{2n+1}(\C)$, then $\ch \tsp{\lambda}(w_\rho) = \gr^\lambda_{2\rho\cup(1)}(-1).$
\item If $G=Sp_{2n}(\C)$, then $\ch\tsp{\lambda}(w_\rho) = \gr^\lambda_{2\rho}(-1).$
\end{enumerate}
\end{thm}
\begin{proof} We proceed by induction on $l(\rho)$. If $\rho$ is a cycle, i.e. $\rho=(n)$, then Theorem \ref{thm:geocalc} and Theorem \ref{thm:grcalc} imply that
\begin{enumerate}[label=$\bullet$]
\item if $G=SO_{2n+1}(\C)$, then
$$
\ch \tsp{\lambda}(w_{(n)}) =  2\delta_{\lambda,(n,n,1)}+ \delta_{\lambda, (2n+1)}+ \sum_{i>j>1 \textup{ odd}, i+j=2n} 2\delta_{\lambda, (i,j,1)}=\gr^\lambda_{(2n,1)}(-1).
$$
\item if $G=Sp_{2n}(\C)$, then
$$
\ch \tsp{\lambda}(w_{(n)}) =  2\delta_{\lambda,(n,n)}+ \delta_{\lambda, (2n)}+ \sum_{i>j>0 \textup{ even}, i+j=2n} 2\delta_{\lambda, (i,j)}=\gr^\lambda_{(2n)}(-1).
$$
\end{enumerate}
Thus the theorem holds. Otherwise, we set $\rho = (k)\cup \rho'$ for some $\rho' \vdash n-k$ and $ 1\leq k <n$. If $G=Sp_{2n}(\C)$ then we have
\begin{align*} 
\ch\tsp{\lambda}(w_\rho) =& \sum_{i\geq k, m_i \geq 2} 2\floor{\frac{m_i}{2}}\ch \tsp{\lambda\pch{i,i}{i-k, i-k}}(w_{\rho'})
\\&+ \sum_{i\geq 2k, m_i \textup{ odd}} (-1)^{\h{\lambda/\left(\lambda\pch{i}{i-2k}\right)}}\ch \tsp{\lambda\pch{i}{i-2k}}(w_{\rho'})
\\&+ \sum_{\substack{0<i-j<2k\leq i+j, \\m_i, m_j \textup{ odd}}} (-1)^{\h{\lambda/ \left(\lambda\pch{i}{j}\right)}}2\ch \tsp{\lambda\pch{i,j}{\frac{i+j}{2}-k,\frac{i+j}{2}-k}}(w_{\rho'}).
\end{align*}
By induction assumption, it is equal to
\begin{align*}
& \sum_{i\geq k, m_i \geq 2} 2\floor{\frac{m_i}{2}}\gr^{\lambda\pch{i,i}{i-k, i-k}}_{2\rho'}(-1)+ \sum_{i\geq 2k, m_i \textup{ odd}} (-1)^{\h{\lambda/\left(\lambda\pch{i}{i-2k}\right)}}\gr^{\lambda\pch{i}{i-2k}}_{2\rho'}(-1)
\\&+ \sum_{\substack{ 0<i-j<2k\leq i+j, \\m_i, m_j \textup{ odd}}} (-1)^{\h{\lambda/ \left(\lambda\pch{i}{j}\right)}}2\gr^{\lambda\pch{i,j}{\frac{i+j}{2}-k,\frac{i+j}{2}-k}}_{2\rho'}(-1)
\\&=\gr^\lambda_{2\rho'\cup(2k)}(-1)=\gr^\lambda_{2\rho}(-1).
\end{align*}
If $G=SO_{2n+1}(\C)$, we simply replace $2\rho', 2\rho$ with $2\rho'\cup(1), 2\rho\cup(1)$, respectively, and the equation is still valid. It suffices for the proof.
\end{proof}
\begin{thm} \label{thm:mainD} Suppose $G=SO_{2n}(\C)$. Let $\sym_{n+}$ (resp. $\sym_{n-}$) be the maximal parabolic subgroup of $W$ which does not contain $s_{-}$ (resp. $s_{+}$) and let $w_\rho \in \sym_{n\pm}$ be an element of cycle type $\rho$. If $\rho$ is even, then we write $w_{\rho\pm} \in \sym_{n\pm}$ to avoid ambiguity. 
\begin{enumerate}
\item If $\lambda$ is not very even and $\rho$ is not even, then $\ch \tsp{\lambda}(w_\rho) = \frac{1}{2}\gr^\lambda_{2\rho}(-1).$
\item If $\lambda$ is very even and $\rho$ is not even, then $\ch \tsp{\lambda\pm}(w_\rho) = \frac{1}{2}\gr^\lambda_{2\rho}(-1).$
\item If $\lambda$ is not very even and $\rho$ is even, then $\ch \tsp{\lambda}(w_{\rho\pm}) = \frac{1}{2}\gr^\lambda_{2\rho}(-1).$
\item If $\lambda$ is very even and $\rho$ is even, then 
$$\ch\tsp{\lambda\pm}(w_{\rho\pm}) = \gr^\lambda_{2\rho}(-1) \quad\textup{ and }\quad \ch\tsp{\lambda\mp}(w_{\rho\pm}) =0.$$
\end{enumerate}
\end{thm}
\begin{proof} By the same argument as Theorem \ref{thm:mainBC}, using Theorem \ref{thm:geocalcD5} and Theorem \ref{thm:grcalc} we have 
$$\ch \atsp{\lambda}(w_\rho) = \gr^\lambda_{2\rho}(-1).$$
Now the statement follows from two observations: (1) if $\rho$ is not even or $\lambda$ is not very even, then $\ch\tsp{\lambda}(w_\rho) =\ch\tsp{\lambda}(\tau(w_\rho))$, and (2) if $\rho$ is even and $\lambda$ is very even, then $\ch \tsp{\lambda\mp}(w_{\rho\pm})=0$. (cf. Lemma \ref{lem:pm})
\end{proof}
By letting $\rho =(1^n)$ in Theorem \ref{thm:mainBC} and \ref{thm:mainD}, we obtain closed formulas for the Euler characteristic of Springer fibers. 
\begin{cor} \label{cor:main} Let $N \in \g$ be a nilpotent element of Jordan type $\lambda$.
\begin{enumerate}
\item If $G=SO_{2n+1}(\C)$, then $\cX(\B_N)=\gr^\lambda_{(1^12^n)}(-1).$
\item If $G=Sp_{2n}(\C)$, then $\cX(\B_N)=\gr^\lambda_{(2^n)}(-1).$
\item If $G=SO_{2n}(\C)$, then $\cX(\B_N)=\frac{1}{2}\gr^\lambda_{(2^n)}(-1).$
\end{enumerate}
\end{cor}

\section{Proof of Main Theorem \ref{mainthm1} for type $B$ and $C$} \label{sec:mainBC}
Assume $G$ is of type $B_n$ or $C_n$ and define 
$$\gue{\lambda}\colonequals \sum_{\mu \vdash |\lambda|} \sg{\lambda}\sg{\mu}K_{\mu, \lambda}(-1)\ch\chi^\mu.$$
In this section we show that $\ch \tsp{\lambda}=\gue{\lambda}$, hence Main Theorem \ref{mainthm1}. Our strategy to prove it is as follows.
\begin{enumerate}
\item $\gue{\lambda}$ satisfies the ``restriction property" similar to Theorem \ref{thm:geocalc}.
\item $\gue{\lambda}$ satisfies the ``induction property" similar to \cite{lus04}, see Proposition \ref{prop:indthm}.
\item Use ``(upper) triangularity properties'' of $\gue{\lambda}$ and total Springer representations.
\end{enumerate}
\subsection{Restriction property}
First we show the following proposition that is an analogue to Theorem \ref{thm:geocalc}.
\begin{prop} \label{prop:restgue}Assume $1\leq k \leq n$. Let $\sym_k \times W' \subset W$ be the maximal parabolic subgroup where $W'$ is of the same type as $W$. If $c\in \sym_k$ is a $k$-cycle, then for $\lambda=(1^{m_1}2^{m_2}\cdots)$ we have
\begin{align*} \Res^W_{c\cdot W'} \gue{\lambda} =& \sum_{i\geq k, m_i \geq 2} 2\floor{\frac{m_i}{2}} \gue{\lambda\pch{i,i}{i-k, i-k}}
+ \sum_{i\geq 2k, m_i \textup{ odd}} (-1)^{\h{\lambda/\left(\lambda\pch{i}{i-2k}\right)}} \gue{\lambda\pch{i}{i-2k}}
\\&+ \sum_{\substack{0<i-j<2k\leq i+j, \\m_i, m_j \textup{ odd}}} (-1)^{\h{\lambda/ \left(\lambda\pch{i}{j}\right)}}2 \gue{\lambda\pch{i,j}{\frac{i+j}{2}-k,\frac{i+j}{2}-k}}.
\end{align*}
Here $\gue{\lambda\pch{i,i}{i-k, i-k}}, \gue{\lambda\pch{i}{i-2k}}, \gue{\lambda\pch{i,j}{\frac{i+j}{2}-k,\frac{i+j}{2}-k}}$ are defined with respect to $W'$. (See \ref{subsec:char} for the definition of $\Res^W_{c\cdot W'}$.)
\end{prop}
To that end, first note that the statement is equivalent to that
\begin{align*} \br{\Res^W_{c\cdot W'} \gue{\lambda},\ch\chi^\nu} =& \sum_{i\geq k, m_i \geq 2} 2\floor{\frac{m_i}{2}} \sg{\nu}\sg{\lambda\pch{i,i}{i-k, i-k}}K_{\nu,\lambda\pch{i,i}{i-k, i-k}}(-1)
\\&+ \sum_{i\geq 2k, m_i \textup{ odd}} (-1)^{\h{\lambda/\left(\lambda\pch{i}{i-2k}\right)}} \sg{\nu}\sg{\lambda\pch{i}{i-2k}}K_{\nu,\lambda\pch{i}{i-2k}}(-1)
\\&+ \sum_{\substack{0<i-j<2k\leq i+j, \\m_i, m_j \textup{ odd}}} (-1)^{\h{\lambda/ \left(\lambda\pch{i}{j}\right)}}2 \sg{\nu}\sg{\lambda\pch{i,j}{\frac{i+j}{2}-k,\frac{i+j}{2}-k}}K_{\nu,\lambda\pch{i,j}{\frac{i+j}{2}-k,\frac{i+j}{2}-k}}(-1)
\end{align*}
for all $\nu \vdash |\lambda|-2k$ with minimal 2-core. On the other hand, we have
\begin{lem} For $f \in \mathcal{R}(W)$, we have
$$\Psi(  \Res^W_{c\cdot W'} f) = p^{*,\br{,}}_{((k),\emptyset)}\Psi( f)= (p_{k}(x)+p_{k}(y))^{*,\br{,}}\Psi(f)$$
where $\Psi$ is defined in Section \ref{sec:symm}.
\end{lem}
\begin{proof} It is an easy exercise.
\end{proof}
Thus we have
\begin{align*}
\br{\Res^W_{c\cdot W'} \gue{\lambda},\ch\chi^\nu} &=\br{\sum_{\mu\vdash |\lambda|} \sg{\lambda}\sg{\mu}K_{\mu,\lambda}(-1)\Res^W_{c\cdot W'} \ch\chi^\mu,\ch\chi^\nu}
\\&=\sum_{\mu\vdash |\lambda|} \sg{\lambda}\sg{\mu}K_{\mu,\lambda}(-1)\br{s_{\cq{0}{\mu}}(x)s_{\cq{1}{\mu}}(y),s_{\cq{0}{\nu}}(x)s_{\cq{1}{\nu}}(y)(p_{k}(x)+p_{k}(y))}.
\end{align*}
(As the formula is symmetric with respect to $x$ and $y$, this formula is valid for both type $B$ and $C$.) By the Murnaghan-Nakayama rule, we have
\begin{equation}\label{eq:resgue}
\br{\Res^W_{c\cdot W'} \gue{\lambda},\ch\chi^\nu}=\sum_{\mu} \sg{\lambda}\sg{\mu}(-1)^{\h{\cq{0}{\mu}/\cq{0}{\nu}}+\h{\cq{1}{\mu}/\cq{1}{\nu}}}K_{\mu,\lambda}(-1)
\end{equation}
where the sum is over all $\mu \vdash |\lambda|$ such that either $\cq{0}{\mu}/\cq{0}{\nu}$ is a border strip of size $k$ and $\cq{1}{\mu}=\cq{1}{\nu}$, or $\cq{1}{\mu}/\cq{1}{\nu}$ is a border strip of size $k$ and $\cq{0}{\mu}=\cq{0}{\nu}$. But this condition can be translated in a more direct way using the following lemma.

\begin{lem} Suppose that a partition $\nu\vdash |\lambda|-2k$ with minimal 2-core is given. 
\begin{enumerate}
\item The mapping $\mu \mapsto (\cq{0}{\mu},\cq{1}{\mu})$ gives a bijection
\begin{gather*}
\{ \mu \vdash |\lambda| \mid \mu/\nu \textup{ is a border strip of size }2k\} \rightarrow
\\\{ (\alpha, \cq{1}{\nu}) \mid \alpha/\cq{0}{\nu} \textup{ is a border strip of size }k\} \sqcup \{ (\cq{0}{\nu},\beta) \mid   \beta/\cq{1}{\nu} \textup{ is a border strip of size }k\}.
\end{gather*}
\item If $\mu/\nu$ is a border strip of size $2k$, then
$$b(\mu)+b(\nu)+\h{\mu/\nu}\equiv\h{\cq{0}{\mu}/\cq{0}{\nu}}+\h{\cq{1}{\mu}/\cq{1}{\nu}} \mod 2.$$
\end{enumerate}
\end{lem}
\begin{proof} The first part is well-known and dates back to \cite{sta50}. In order to prove the second claim, we first fix $\mu$ and regard it as an infinite 01-sequences; first consider the Young diagram of $\mu$ and label horizontal line segments with 1 and vertical ones with 0. Then we read the labels from southwest to northeast. For example, the partition $(6,4,2)$ is converted to 
$$\cdots 0,0,0,1,1,0,1,1,0,1,1,0,1,1,1,\cdots.$$
We let $a_1=1$ to be the 1 that appears first in the sequence corresponding to $\mu$ and successively define $a_2, a_3, \cdots$ by reading the sequence from left to right. Also we let $a_i=0$ for $i\leq 0$. For example, if $\mu=(6,4,2)$ then we have
$$1= a_1 =a_2=a_4=a_5=a_7=a_8=a_{10}=a_{11}=\cdots, \quad \cdots =a_{-1}=a_0= a_3 =a_6=a_9=0.$$
Then it is known that the 2-quotient of $\mu$ corresponds to the following sequences
$$\cdots , a_{-3}, a_{-1}, a_1, a_3, a_5, \cdots\quad  \textup{ and } \quad \cdots , a_{-4}, a_{-2}, a_0, a_2, a_4, \cdots.$$
If we remove a border strip of size $2k$ to obtain $\nu$, then there exists $x\geq 1$ such that $a_x=1$, $a_{x+2k}=0$ and $\nu$ corresponds to the 01-sequence
$$\cdots, a_{-2}, a_{-1}, a_0, \cdots, a_{x-1}, 0, a_{x+1}, \cdots, a_{x+2k-1}, 1, a_{x+2k+1}, \cdots.$$
In other words, we change $a_x, a_{x+2k}$ to $0, 1$, respectively, from the 01-sequence of $\mu$. 

Let $j_1< \cdots< j_r$ be the collection of all indices of zeroes in $\{a_{x+1}, \cdots, a_{x+2k-1}\}$. Then it is easy to see that $\h{\mu/\nu}=r$ and each row in the border strip $\mu/\nu$ has length $j_1-x, j_2-j_1, \cdots, j_{r}-j_{r-1}, x+2k-j_{r-1}$, respectively. Thus it is also clear that 
$$b(\mu)-b(\nu)\equiv (j_1-x)+(j_3-j_2)+\cdots = (j_2-j_1)+(j_4-j_3)+\cdots \mod 2$$
since $|\mu/\nu|=2k$ is even. It follows that
\begin{align*}
\textup{if } r \textup{ is even}, &\quad b(\mu)+b(\nu) \equiv x+\sum_{a=1}^r j_a + (x+2k) \equiv \sum_{a=1}^r j_a \mod 2,
\\\textup{if } r \textup{ is odd}, &\quad b(\mu)+b(\nu) \equiv x+\sum_{a=1}^r j_a  \mod 2.
\end{align*}
This is equivalent to
\begin{align*}
b(\mu)+b(\nu)+\h{\mu/\nu} \equiv (x+1)r+\sum_{a=1}^r j_a \mod 2.
\end{align*}
On the other hand, by a similar reason $\h{\cq{0}{\mu}/\cq{0}{\nu}}+\h{\cq{1}{\mu}/\cq{1}{\nu}}$ is the number of $j_a$ such that $j_a\equiv x \mod 2$. Thus we may write
$$\h{\cq{0}{\mu}/\cq{0}{\nu}}+\h{\cq{1}{\mu}/\cq{1}{\nu}} \equiv \sum_{a=1}^r (x+j_a+1) \equiv (x+1)r+\sum_{a=1}^r j_a \mod 2.$$
It suffices for the proof.
\end{proof}
Using the lemma above, we may simplify (\ref{eq:resgue}) to the following form.
$$\br{\Res^W_{c'W'} \gue{\lambda},\ch\chi^\nu}=\sum_\mu \sg{\lambda}\sg{\nu}(-1)^{\h{\mu/\nu}}K_{\mu,\lambda}(-1)$$
Here the sum is over all $\mu \vdash |\lambda|$ such that $\mu/\nu$ is a border strip of size $2k$. Thus in order to prove Proposition \ref{prop:restgue} it suffices to show that
\begin{align*} 
&\sum_{\mu \vdash |\lambda|, \mu/\nu \textup{ border strip}} \sg{\lambda}(-1)^{\h{\mu/\nu}}K_{\mu,\lambda}(-1)
\\=& \sum_{i\geq k, m_i \geq 2} 2\floor{\frac{m_i}{2}}\sg{\lambda\pch{i,i}{i-k, i-k}}K_{\nu,\lambda\pch{i,i}{i-k, i-k}}(-1)
\\&+ \sum_{i\geq 2k, m_i \textup{ odd}} (-1)^{\h{\lambda/\left(\lambda\pch{i}{i-2k}\right)}} \sg{\lambda\pch{i}{i-2k}}K_{\nu,\lambda\pch{i}{i-2k}}(-1)
\\&+ \sum_{\substack{0<i-j<2k\leq i+j,\\m_i, m_j \textup{ odd}}} (-1)^{\h{\lambda/ \left(\lambda\pch{i}{j}\right)}}2 \sg{\lambda\pch{i,j}{\frac{i+j}{2}-k,\frac{i+j}{2}-k}}K_{\nu,\lambda\pch{i,j}{\frac{i+j}{2}-k,\frac{i+j}{2}-k}}(-1).
\end{align*}
But this is exactly Theorem \ref{thm:grcalc2}.

\subsection{Induction property}
First we recall Lusztig's induction theorem for total Springer representations \cite[Theorem 1.3]{lus04}.
\begin{prop} \label{prop:indthm}Let $L \subset G$ be a Levi subgroup of some parabolic subgroup of $G$ and $W'$ be the Weyl group of $L$, naturally identified with a subgroup of $W$. For $N\in \Lie L$, let $\B'_{N}$ be the Springer fiber corresponding to $L$. Then we have an isomorphism
$$H^*(\B_N) \simeq \Ind_{W'}^W H^*(\B'_N)$$
of $W$-modules. Here we regard $H^*(\B'_N)$ as the total Springer representation with respect to $W'$.
\end{prop}
We also claim that $\gue{\lambda}$ enjoys a similar property as follows.
\begin{prop} \label{prop:indgue} Assume $1\leq k \leq n$. Let $\sym_k \times W' \subset W$ be the maximal parabolic subgroup where $W'$ is of the same type as $W$. Then we have
$$\gue{\lambda\cup (k,k)} = \Ind_{\sym_k\times W' }^W 1 \times \gue{\lambda}.$$
Here $1=\ch \Id_{\sym_k}$.
\end{prop}
To that end, it is enough to show that for any $\nu \vdash |\lambda|+2k$ with minimal 2-core we have
\begin{gather*}
\sg{\lambda\cup(k,k)}\sg{\nu}K_{\nu,\lambda\cup(k,k)}(-1) =\br{\gue{\lambda\cup (k,k)},\ch \chi^{\nu}}
\\= \br{\Ind_{\sym_k\times W' }^W 1 \times \gue{\lambda}, \ch\chi^{\nu}}=\br{1 \times \gue{\lambda}, \Res_{\sym_k \times W'}^W \ch\chi^{\nu}}.
\end{gather*}
We claim the following.
\begin{lem} \label{lem:yam} For $\mu \vdash |\lambda|$, $\br{\Res_{\sym_k \times W'}^W \ch\chi^{\nu}, 1\times\ch\chi^{\mu}}=1$ if $\mu \subset \nu$ and $\nu/\mu$ can be covered by dominoes with no two dominoes intersecting the same column, and 0 otherwise.
\end{lem}
\begin{proof}
Recall that if $\sym_n \subset W$ is the maximal parabolic subgroup of type $A$, then $\Psi(\ch \Ind_{\sym_n}^W \Id_{\sym_n})=\Delta h_n$ (see Section \ref{sec:symm} for the definition of $\Psi$). Thus we have
\begin{align*}
&\br{\Res_{\sym_k \times W'}^W \ch\chi^{\nu}, 1 \times\ch\chi^{\mu}}=\br{\ch\chi^{\nu}, \Ind_{\sym_k \times W'}^W1 \times\ch\chi^{\mu}}
\\&= \br{s_{\cq{0}{\nu}}(x)s_{\cq{1}{\nu}}(y),s_{\cq{0}{\mu}}(x)s_{\cq{1}{\mu}}(y)\Delta h_k}= \br{s_{\cq{0}{\nu}/\cq{0}{\mu}}(x)s_{\cq{1}{\nu}/\cq{1}{\mu}}(y), \Delta h_k}
\\&= \br{s_{\cq{0}{\nu}/\cq{0}{\mu}}s_{\cq{1}{\nu}/\cq{1}{\mu}}, h_k}.
\end{align*}
using the adjunction of $\Delta$ and $\nabla$. (The formula is symmetric with respect to $x$ and $y$, thus this formula is valid for both type $B$ and $C$.) By \cite[Theorem 2.2.2]{vlee00} (see also \cite{cale95}) this is the number of Yamanouchi domino tableaux of shape $\nu/\mu$ and weight $(k)$. One can easily check that this number is 1 if and only if the condition in the lemma is satisfied and otherwise 0, thus the result follows.
\end{proof}

Thus we have
$$ \br{1\times \gue{\lambda}, \Res_{\sym_k \times W'}^W \ch\chi^{\nu}}= \sum_{\mu} \sg{\lambda}\sg{\mu}K_{\mu,\lambda}(-1)$$
where the sum is over $\mu \vdash |\lambda|$ which satisfies the condition in Lemma \ref{lem:yam}. As $\sg{\lambda\cup(k,k)}=(-1)^k\sg{\lambda}$, in order to prove Proposition \ref{prop:indgue} it suffices to show that
\begin{equation}\label{eqnind}
(-1)^k\sg{\nu}K_{\nu,\lambda\cup(k,k)}(-1) =\sum_{\mu} \sg{\mu}K_{\mu,\lambda}(-1)
\end{equation}
where the sum is over $\mu$ which satisfies the same condition above. Now we recall the following result of Lascoux, Leclerc, and Thibon.
\begin{lem} Suppose $\lambda = \alpha\cup\beta\cup\beta$ for some partition $\alpha, \beta$. Then we have
$$Q'_{\lambda}(-1) = Q'_{\alpha}(-1)\prod_{i\geq 1} \left(Q'_{(i^2)}(-1)\right)^{\beta_i}.$$
\end{lem}
\begin{proof} This is a special case of \cite[Theorem 2.1]{llt94} for $k=2$. The statement therein only considers the case when $\alpha$ is multiplicity-free (for $k=2$), but it is in fact equivalent to the formula above.
\end{proof}
Therefore in particular, $Q'_{\lambda\cup(k,k)}(-1) = Q'_{\lambda}(-1) Q'_{(k,k)}(-1).$
By \cite[Theorem 2.2]{llt94}, $Q'_{(k,k)}(-1)=(-1)^k s_k[p_2]$. Thus we have
$$K_{\nu, \lambda\cup (k,k)}(-1)=\br{Q'_{\lambda\cup(k,k)}(-1), s_\nu} = (-1)^k\br{s_k[p_2]\sum_{\mu}K_{\mu, \lambda}(-1)s_\mu , s_\nu}$$
where the sum is over $\mu \vdash |\lambda|$ which satisfies the condition in Lemma \ref{lem:yam}.
But \cite[Theorem 4.1]{cale95} says that $\tsg{\nu/\mu}\br{s_k[p_2]s_\mu, s_\nu}$ is the number of Yamanouchi domino tableaux of shape $\nu/\mu$ and weight $(k)$, i.e. 1 if and only if the condition in Lemma \ref{lem:yam} is satisfied and otherwise 0. Therefore, 
$$K_{\nu, \lambda\cup(k,k)}(-1) = (-1)^k\sum_{\mu}\tsg{\nu/\mu}K_{\mu, \lambda}(-1)$$
where the sum is over $\mu \vdash |\lambda|$ which satisfies the condition in Lemma \ref{lem:yam}. Since $\tsg{\nu/\mu} = \sg{\nu}\sg{\mu},$ this reads
$$(-1)^k \sg{\nu}K_{\nu, \lambda\cup(k,k)}(-1) = \sum_{\mu}\sg{\mu}K_{\mu, \lambda}(-1)$$
which proves (\ref{eqnind}). Thus Proposition \ref{prop:indgue} is proved.

\subsection{Triangularity property} 
We proceed to the proof of Main Theorem \ref{mainthm1} for type $B$ and $C$ by induction on $n=\rk G$. As we observed already, $\tsp{\lambda}$ and $\gue{\lambda}$ enjoy similar induction properties. In particular, if $\lambda = \mu\cup(k,k)$ for some $\mu$ and $k\geq 1$, then by induction assumption we have
$$\ch\tsp{\lambda}=\Ind_{\sym_k \times W'}^W (1\times \ch\tsp{\mu})=\Ind_{\sym_k \times W'}^W (1\times \gue{\mu})=\gue{\lambda}.$$
Thus it remains to consider the case when $\lambda$ corresponds to a distinguished nilpotent element. We fix such $\lambda$ from now on.

Also, for any maximal proper parabolic subgroup $\sym_k \times W' \subset W$ for $k\geq 1$ and $c \in \sym_k$ a $k$-cycle, by restriction properties and induction assumption we have
\begin{align*} \Res^W_{c\cdot W'}\ch \tsp{\lambda} =& \sum_{i\geq k, m_i \geq 2} 2\floor{\frac{m_i}{2}} \ch\tsp{\lambda\pch{i,i}{i-k, i-k}}
\\&+ \sum_{i\geq 2k, m_i \textup{ odd}} (-1)^{\h{\lambda/\left(\lambda\pch{i}{i-2k}\right)}} \ch\tsp{\lambda\pch{i}{i-2k}}
\\&+ \sum_{\substack{0<i-j<2k\leq i+j, \\m_i, m_j \textup{ odd}}} (-1)^{\h{\lambda/ \left(\lambda\pch{i}{j}\right)}}2 \ch\tsp{\lambda\pch{i,j}{\frac{i+j}{2}-k,\frac{i+j}{2}-k}}.
\\=&  \sum_{i\geq k, m_i \geq 2} 2\floor{\frac{m_i}{2}} \gue{\lambda\pch{i,i}{i-k, i-k}}
\\&+ \sum_{i\geq 2k, m_i \textup{ odd}} (-1)^{\h{\lambda/\left(\lambda\pch{i}{i-2k}\right)}} \gue{\lambda\pch{i}{i-2k}}
\\&+ \sum_{\substack{0<i-j<2k\leq i+j, \\m_i, m_j \textup{ odd}}} (-1)^{\h{\lambda/ \left(\lambda\pch{i}{j}\right)}}2 \gue{\lambda\pch{i,j}{\frac{i+j}{2}-k,\frac{i+j}{2}-k}}.
\\=&\Res^W_{c\cdot W'} \gue{\lambda}.
\end{align*}
Thus $\ch \tsp{\lambda}(w) =  \gue{\lambda}(w)$ if $w \in W$ is contained in some maximal proper parabolic subgroup of $W$. In other words, $\ch \tsp{\lambda}-\gue{\lambda}$ is supported on the elements which are not contained in any proper parabolic subgroup. $w \in W$ is such an element if and only if the conjugacy class of $w$ is parametrized by $(\emptyset, \rho)$ for some partition $\rho\vdash n$. Thus for any $\mu \vdash n$ with minimal 2-core, we have
\begin{equation}\label{eq:resA}
\br{\ch \tsp{\lambda}-\gue{\lambda},\ch\chi^\mu}=\sum_{\rho \vdash n} \frac{1}{z_\rho 2^{l(\rho)}} \left(\ch \tsp{\lambda}(w_{(\emptyset,\rho)})-\gue{\lambda}(w_{(\emptyset,\rho)})\right) \ch\chi^\mu(w_{(\emptyset,\rho)}).
\end{equation}
Here $z_\rho 2^{l(\rho)}$ is the size of the centralizer of $w_{(\emptyset,\rho)}$ in $W$. We observe the following (upper) triangularity properties of $\tsp{\lambda}$ and $\gue{\lambda}$.
\begin{lem} \label{lem:triangular}Let $\lambda, \mu$ be partitions with minimal 2-cores such that $|\lambda|=|\mu|$.
\begin{enumerate}
\item If $\br{\tsp{\lambda}, \chi^{\mu}} \neq 0$, then $\lambda_1 \leq \mu_1$.
\item If $\br{\gue{\lambda},\ch \chi^{\mu}} \neq 0$, then $\lambda_1 \leq \mu_1$.
\end{enumerate}
\end{lem}
\begin{proof} The second part directly follows from the definition of $\gue{\lambda}$ and the fact that $K_{\mu,\lambda}(t)\neq 0 \Rightarrow \lambda \leq \mu$. To show the first part, we let $\tilde{\mu}\vdash n$ be a partition such that $\chi^{\mu}$ corresponds to some local system on the nilpotent orbit of $\g$ parametrized by $\tilde{\mu}$ under Springer correspondence. (Thus if $\chi^{\mu}$ corresponds to the trivial local system on some nilpotent orbit, then $\mu=\tilde{\mu}$.) By upper-triangularity property of Springer representations, we have $\lambda \leq \tilde{\mu}$, so that $\lambda_1 \leq \tilde{\mu}_1$. Now it is an easy combinatorial exercise to see that $\tilde{\mu}_1 \leq \mu_1$ by exploiting construction of Lusztig's symbols and its relation to Springer correspondence. (e.g. \cite[Chapter 13.3]{car93})
\end{proof}

Therefore, for $\mu \vdash n$ such that $\mu_1 < \lambda_1$, we have $\br{\ch\tsp{\lambda}-\gue{\lambda},\ch \chi^{\mu}}=0$. Now we define an ad-hoc notation 
$$\Res^W_{``\!\sym_n\!\!"} : \mathcal{R}(W) \rightarrow \mathcal{R}(\sym_n) \quad \textup{ by } \quad (\Res^W_{``\!\sym_n\!\!"} f)(w_\rho) = f(w_{(\emptyset, \rho)}).$$
Then combined with (\ref{eq:resA}), we only need to show the following statement.
\begin{equation}\label{eq:claim1}
 \{\Res_{``\!\sym_n\!\!"}^W\ch \chi^\mu \in \mathcal{R}(\sym_n) \mid\mu \vdash n, \mu_1 < \lambda_1 \} \textup{ spans } \mathcal{R}(\sym_n).
\end{equation}
Recall that $\lambda$ is a Jordan type of some distinguished nilpotent element in $\g$. For a partition $\mu$, we define $m_\mu \colonequals \max((\cq{0}{\mu})_1, (\cq{1}{\mu})_1).$
\begin{lem} \label{lem:estim}Suppose $\mu$ is a partition with minimal 2-core. 
\begin{enumerate}
\item Suppose $|\mu|$ is odd. If $\mu_1$ is even, then $\mu_1 \leq 2m_\mu-2$. If $\mu_1$ is odd, then $\mu_1\leq 2m_\mu+1$. 
\item Suppose $|\mu|$ is even. If $\mu_1$ is even, then $\mu_1 \leq 2m_\mu$. If $\mu_1$ is odd, then $\mu_1\leq 2m_\mu-1$. 
\end{enumerate}
\end{lem}
\begin{proof} This is an easy combinatorial exercise using the definition of 2-cores.
\end{proof}
We define $M\in\N$ as follows.
\begin{enumerate}
\item If $G$ is of type $B_n$,  set $M$ to be such that $M^2 < 2n+1 \leq (M+1)^2$.
\item If $G$ is of type $C_n$,  set $M$ to be such that $(M-1)M<2n\leq M(M+1)$.
\end{enumerate}
\begin{lem} \label{lem:parbound}Let $\lambda$ be a Jordan type of a distinguished nilpotent element in $\g$. Also suppose that a partition $\mu \vdash |\lambda|$ with minimal 2-core satisfies $m_\mu <M$. Then we have $\mu_1 <\lambda_1$.
\end{lem}
\begin{proof} First assume that $G$ is of type $B$. Then since $\lambda$ is strict and each part is odd, we have $M^2<2n+1 = |\lambda| \leq \frac{(\lambda_1+1)^2}{4}$, thus $2M <\lambda_1+1$. As $\lambda_1$ is odd, we have $2M \leq \lambda_1-1$. Now by Lemma \ref{lem:estim} we have
\begin{enumerate}
\item if $\mu_1$ is even, then $\mu_1 \leq 2m_{\mu}-2<2M-2\leq \lambda_1-3<\lambda_1,$ and
\item if $\mu_1$ is odd, then $\mu_1 \leq 2m_{\mu}+1<2M+1\leq \lambda_1.$
\end{enumerate}
Thus the result follows. Now assume that $G$ is of type $C$. Similarly, we have $(M-1)M<2n=|\lambda|\leq \frac{\lambda_1(\lambda_1+2)}{4}$, thus $2M<\lambda_1+2$. As $\lambda_1$ is even, we have $2M\leq \lambda_1$. Now by Lemma \ref{lem:estim} we have
\begin{enumerate}
\item if $\mu_1$ is even, $\mu_1 \leq 2m_{\mu}<2M\leq \lambda_1,$ and
\item if $\mu_1$ is odd, $\mu_1 \leq 2m_{\mu}-1<2M-1\leq \lambda_1-1< \lambda_1.$
\end{enumerate}
Thus the result follows.
\end{proof}
%
%

We define
$$\mathscr{L}\colonequals \spn{ \Res_{``\!\sym_n\!\!"}^W \ch \chi^{\mu} \mid m_\mu < M} \subset \mathcal{R}(\sym_n).$$
Then by Lemma \ref{lem:parbound}, in order to prove (\ref{eq:claim1}) it suffices to show that $\mathscr{L} = \mathcal{R}(\sym_n)$. The following lemma provides a main tool to describe $\mathscr{L}$.
\begin{lem} If $G$ is of type $B$, then $\Psi(\Res_{``\!\sym_n\!\!"}^W \ch \chi^\mu) = (-1)^{|\cq{1}{\mu}|}s_{\cq{0}{\mu}}s_{(\cq{1}{\mu})'}$. If $G$ is of type $C$, then $\Psi(\Res_{``\!\sym_n\!\!"}^W \ch \chi^\mu) = (-1)^{|\cq{0}{\mu}|}s_{\cq{1}{\mu}}s_{(\cq{0}{\mu})'}$. (See Section \ref{sec:symm} for the definition of $\Psi$.)
\end{lem}
\begin{proof} We only consider the case when $G$ is of type $B$ and the other case is totally analogous. By the definition of $\Res_{``\!\sym_n\!\!"}^W$, it is enough to show that for any $\rho \vdash n$ we have
$$ \ch \chi^\mu (w_{(\emptyset,\rho)}) = \br{s_{\cq{0}{\mu}}(x)s_{\cq{1}{\mu}}(y), p_{(\emptyset, \rho)}} = \br{(-1)^{|\cq{1}{\mu}|}s_{\cq{0}{\mu}}s_{(\cq{1}{\mu})'}, p_{\rho}}.$$
Recall that $p_{(\emptyset, \rho)} =\prod_{i\geq 1}(p_{\rho_i}(x)-p_{\rho_i}(y))$. If we apply the involution $\omega_y : \Lambda(x)\otimes \Lambda(y) \rightarrow \Lambda(x)\otimes \Lambda(y): s_\alpha(x)s_\beta(y) \mapsto s_\alpha(x) s_{\beta'}(y)$, then we have
\begin{align*}
&\br{s_{\cq{0}{\mu}}(x)s_{\cq{1}{\mu}}(y), p_{(\emptyset, \rho)}}= \br{s_{\cq{0}{\mu}}(x)s_{(\cq{1}{\mu})'}(y), \prod_{i\geq 1}(p_{\rho_i}(x)+(-1)^{\rho_i}p_{\rho_i}(y)) }
\\&=(-1)^{|\cq{1}{\mu}|}\br{s_{\cq{0}{\mu}}(x)s_{(\cq{1}{\mu})'}(y), \prod_{i\geq 1}(p_{\rho_i}(x)+p_{\rho_i}(y)) }=(-1)^{|\cq{1}{\mu}|}\br{s_{\cq{0}{\mu}}s_{(\cq{1}{\mu})'}, p_{\rho} }
\end{align*}
by adjunction of $\Delta$ and $\nabla$. Thus the result follows.
\end{proof}
Therefore, we have
$$\Psi( \mathscr{L}) = \spn{s_\alpha s_{\beta'} \mid |\alpha|+|\beta|=n,\ \ l(\alpha), l(\beta)<M} \subset \Lambda^n$$
and it remains to show that $\Psi(\mathscr{L}) = \Lambda^n$.

%
\begin{lem} \label{lem:spanBC}Let $M>0$ be such that $M^2>n$. Then $\Psi(\mathscr{L}) =\Lambda^n$.
\end{lem}
\begin{proof} We introduce an increasing filtration
$$\Lambda_{k}^n \colonequals \spn{s_{\lambda} \mid \sum_{i\geq M} \lambda_i \leq k}\subset \Lambda^n.$$
Clearly it exhausts $\Lambda^n$. We show $\Lambda_{k}^n \subset \Psi( \mathscr{L})$ by induction on $k\geq 0$. The case $k=0$ is straightforward, thus suppose $k\geq 1$. For any $\lambda$ such that $s_\lambda \in \Lambda^n_{k}- \Lambda^n_{k-1}$, we define $\alpha=(\lambda_1, \cdots, \lambda_{M-1})$ and $\beta = (\lambda_{M}, \cdots, \lambda_{l(\lambda)})'$ so that $\lambda = \alpha \cup \beta'$. Then $l(\beta) = \lambda_{M}< M$,  since otherwise we have $|\lambda| \geq \sum_{i=1}^{M} \lambda_i \geq M\lambda_{M}\geq M^2>n$ which is absurd. We claim $s_\alpha s_{\beta'} - s_\lambda \in \Lambda^n_{k-1}.$ 
Indeed, if $\br{s_\alpha s_{\beta'},s_\mu}\neq 0$ for some $\mu \vdash n$, then clearly $\alpha \subset \mu$. Thus we have $\sum_{i\geq M} \mu_i \leq k$ and the equality holds if and only if $\mu_i = \alpha_i$ for $1\leq i <M$. Now the combinatorial description of the Littlewood-Richardson rule (e.g. \cite[Theorem A1.3.3]{sta86}) implies that
$$s_\mu \in \Lambda^n_{k} - \Lambda^n_{k-1}, \ \ \br{s_\alpha s_{\beta'},s_\mu}\neq 0 \quad \Rightarrow \quad \mu=\lambda$$
and also $\br{s_\alpha s_{\beta'},s_\lambda}=1$. Thus the result follows.
\end{proof}

\begin{lem} Let $n \geq 2$. Then $M^2>n$ except when $G$ is of type $B_4$.
\end{lem}
\begin{proof} Easy.
\end{proof}
This finishes the proof except when $n=1$ or $G=SO_{9}(\C)$. If $n=1$, then the statement can be easily checked by direct calculation. If $G=SO_{9}(\C)$, the argument above does not work only when $\lambda=(5,3,1)$. However, direct calculation shows that we have
$$\ch\tsp{(5,3,1)}=\ch(\chi^{(5,3,1)}\oplus\chi^{(5,4)}\oplus\chi^{(6,2,1)}\oplus\chi^{(7,1,1)}\oplus\chi^{(9)})=\gue{(5,3,1)},$$
and it completes the proof.

\section{Proof of Main Theorem \ref{mainthm1} for type $D$} \label{sec:mainD}
It remains to handle the case when $G$ is of type $D_n$ for $n\geq 2$. Here, we define 
$$\ague{\lambda}\colonequals \sum_{\substack{\mu \vdash |\lambda|\\\textup{ not very even}}} \sg{\lambda}\sg{\mu}K_{\mu, \lambda}(-1)\ch\chi^\mu + \sum_{\substack{\mu \vdash |\lambda|\\\textup{ very even}}} \sg{\lambda}K_{\mu, \lambda}(-1)\ch(\chi^{\mu+}\oplus \chi^{\mu-}).$$
(Note that $\sg{\mu}=1$ if $\mu$ is very even.)
\subsection{$\ch \atsp{\lambda} = \ague{\lambda}$}  The Weyl group $W$ of $G$ can be regarded naturally as a subgroup of $W(C_n)$. If we write $\gue{C, \lambda}\colonequals \sum_{\mu \vdash |\lambda|} \sg{\lambda}\sg{\mu}K_{\mu, \lambda}(-1)\ch\chi^\mu_C$ where $\chi^\mu_C$ is the irreducible character of $W(C_n)$ parametrized by $\mu \vdash 2n$, then it is clear that $\ague{\lambda}=\Res^{W(C_n)}_{W} \gue{C, \lambda}$. 

We first prove $\ch \atsp{\lambda} = \ague{\lambda}$. Our strategy is similar to the previous section; we first state and prove the restriction and induction properties of $\ague{\lambda}$.

\begin{prop} \label{prop:resD}Assume $1\leq k \leq n$. Let $\sym_k \times W' \subset W$ be a parabolic subgroup where $W'$ is of the same type as $W$. If $1\leq k \leq n-2$, then we choose $W'$ such that $\sym_k \times W'$ becomes a maximal parabolic subgroup of $W$. ($W'$ is trivial if $k =n-1$ or $n$.) If $c\in \sym_k$ is a $k$-cycle and $\lambda=(1^{m_1}2^{m_2}\cdots)$ then we have
\begin{align*} \Res^W_{c\cdot W'} \ague{\lambda} =& \sum_{i\geq k, m_i \geq 2} 2\floor{\frac{m_i}{2}} \ague{\lambda\pch{i,i}{i-k, i-k}}
\\&+ \sum_{i\geq 2k, m_i \textup{ odd}} (-1)^{\h{\lambda/\left(\lambda\pch{i}{i-2k}\right)}} \ague{\lambda\pch{i}{i-2k}}
\\&+ \sum_{\substack{0<i-j<2k\leq i+j, \\m_i, m_j \textup{ odd}}} (-1)^{\h{\lambda/ \left(\lambda\pch{i}{j}\right)}}2 \ague{\lambda\pch{i,j}{\frac{i+j}{2}-k,\frac{i+j}{2}-k}}.
\end{align*}
Here $\ague{\lambda\pch{i,i}{i-k, i-k}}, \ague{\lambda\pch{i}{i-2k}}, \ague{\lambda\pch{i,j}{\frac{i+j}{2}-k,\frac{i+j}{2}-k}}$ are defined with respect to $W'$. Also, for $k \geq n-1$ we use the convention that $\ague{(1,1)}= 2$ and $\ague{\emptyset}= 1.$ (See \ref{subsec:char} for the definition of $\Res^W_{c\cdot W'}$.)
\end{prop}
\begin{proof} Let $W'' \subset W(C_n)$ be the parabolic subgroup of type $C_{n-k}$ such that $W''\cap W = W'$. Then Proposition \ref{prop:restgue} shows that
\begin{align*} \Res^{W(C_n)}_{c\cdot W''} \gue{C,\lambda} =& \sum_{i\geq k, m_i \geq 2} 2\floor{\frac{m_i}{2}} \gue{C,\lambda\pch{i,i}{i-k, i-k}}
\\&+ \sum_{i\geq 2k, m_i \textup{ odd}} (-1)^{\h{\lambda/\left(\lambda\pch{i}{i-2k}\right)}} \gue{C,\lambda\pch{i}{i-2k}}
\\&+ \sum_{\substack{0<i-j<2k\leq i+j, \\m_i, m_j \textup{ odd}}} (-1)^{\h{\lambda/ \left(\lambda\pch{i}{j}\right)}}2 \gue{C,\lambda\pch{i,j}{\frac{i+j}{2}-k,\frac{i+j}{2}-k}}.
\end{align*}
Now the result follows by applying $\Res^{W''}_{W'}$ on both sides. (Note that $\Res^{W(C_1)}_{\{*\}} \gue{C,(1,1)}=2$ and $\gue{C,\emptyset}=1$.)
\end{proof}

\begin{prop} \label{prop:indD}Assume $1\leq k \leq n-2$. Let $\sym_k \times W' \subset W$ be a maximal parabolic subgroup where $W'$ is of the same type as $W$. Then we have
$$\ague{\lambda\cup (k,k)} = \Ind_{\sym_k\times W' }^W 1 \times \ague{\lambda}.$$
Here $1=\ch \Id_{\sym_k}$. Furthermore, we have
$$\ague{ (n-1,n-1,1,1)} = \Ind_{\sym_{n-1}}^W 2, \quad \ague{(n,n)} = \Ind_{\sym_{n+}}^W 1+\Ind_{\sym_{n-}}^W 1.$$
\end{prop}
\begin{proof} We define $W'' \subset W(C_n)$ as before. We know that $\gue{C,\lambda\cup (k,k)} = \Ind_{\sym_k\times W'' }^{W(C_n)} 1 \times \gue{C,\lambda} .$ If we apply $\Res^{W(C_n)}_W$ on both sides then we have
\begin{align*}
\ague{\lambda\cup (k,k)}=&\Res^{W(C_n)}_W\gue{C,\lambda\cup (k,k)}=\Res^{W(C_n)}_W\Ind_{\sym_k\times W'' }^{W(C_n)}1 \times \gue{C,\lambda}
\\=&\Ind^{W} _{\sym_k \times W'}\Res^{\sym_k\times W'' }_{\sym_k \times W'} 1 \times \gue{C,\lambda}=\Ind^{W} _{\sym_k \times W'} 1 \times \ague{\lambda}
\end{align*}
by Mackey's formula. Similarly, we have
\begin{align*}
\ague{(n-1,n-1,1,1)}=&\Res^{W(C_n)}_W\gue{C,(n-1,n-1,1,1)}=\Res^{W(C_n)}_W\Ind_{\sym_{n-1}\times W(C_1) }^{W(C_n)}1 \times \gue{C,(1,1)}
\\=&\Ind^{W} _{\sym_{n-1} }\Res^{\sym_{n-1}\times W(C_1) }_{\sym_{n-1}} 1 \times \gue{C,(1,1)}=\Ind^{W} _{\sym_{n-1}} 2 ,
\\\ague{(n,n)}=&\Res^{W(C_n)}_W\gue{C,(n,n)}=\Res^{W(C_n)}_W\Ind_{\sym_{n} }^{W(C_n)}1
= \Ind_{\sym_{n+}}^W 1+\Ind_{\sym_{n-}}^W 1.
\end{align*}
\end{proof}
Similarly to the previous section, by Proposition \ref{prop:indD}, in order to prove $\ch \atsp{\lambda} = \ague{\lambda}$ it suffices to consider the case when $\lambda$ corresponds to a distinguished nilpotent element. We fix such $\lambda$ from now on. Also by Proposition \ref{prop:resD}, we only need to show that $\ch\atsp{\lambda}(w)=\ague{\lambda}(w)$ for $w\in W$ parametrized by $(\emptyset, \rho)$ for some $\rho \vdash n$. To that end we follow the arguments in the previous section.
\begin{lem}
\begin{enumerate}
\item $\lambda_1 \leq \mu_1$ if the following holds: $\mu$ is not very even and $\br{\atsp{\lambda}, \chi^{\mu}} \neq 0$, or $\mu$ is very even and either $\br{\atsp{\lambda}, \chi^{\mu+}} \neq 0$ or $\br{\atsp{\lambda}, \chi^{\mu-}} \neq 0$.
\item $\lambda_1 \leq \mu_1$ if the following holds: $\mu$ is not very even and $\br{\ague{\lambda}, \ch\chi^{\mu}} \neq 0$, or $\mu$ is very even and either $\br{\ague{\lambda}, \ch\chi^{\mu+}} \neq 0$ or $\br{\ague{\lambda}, \ch\chi^{\mu-}} \neq 0$.
\end{enumerate}
\end{lem}
\begin{proof} It can be proved similarly to Lemma \ref{lem:triangular}.
\end{proof}
\begin{rmk} One needs to be careful here since $\chi^{\mu}=\chi^\nu$ does not necessarily mean $\mu=\nu$. In fact, if $\mu, \nu \vdash 2n$ are partitions such that $\cq{c}{\mu}=\cq{c}{\nu}=\emptyset$, $\cq{0}{\mu}=\cq{1}{\nu}$, and $\cq{1}{\mu}=\cq{0}{\nu}$, then $\chi^\mu = \chi^\nu$ but $\mu\neq \nu$ unless $\mu$ is very even. However, the lemma above still remains valid; one may find the following fact useful.
\end{rmk}
\begin{lem} Suppose $\mu, \nu \vdash 2n$ are partitions such that $\cq{c}{\mu}=\cq{c}{\nu}=\emptyset$, $\cq{0}{\mu}=\cq{1}{\nu}$, and $\cq{1}{\mu}=\cq{0}{\nu}$. Then there exists $r \in \N$ such that $\{\mu_1, \nu_1\}$ equals either $\{2r\}$ or $\{2r-1, 2r\}$. In particular, $\max \{\mu_1, \nu_1\}$ is always even.
\end{lem}

Recall that we defined $m_\mu \colonequals \max((\cq{0}{\mu})_1, (\cq{1}{\mu})_1).$ We set $M\in \N$ such that $(M-1)^2 < 2n \leq M^2$. We also let $M_C \in \N$ such that $(M_C-1)M_C<2n\leq M_C(M_C+1)$ as in the previous section for type $C$.
\begin{lem} \label{lem:bound1}Let $\lambda$ be a Jordan type of a distinguished nilpotent element in $\g$. Also suppose we are given a partition $\mu \vdash 2n$ such that $m_\mu <M$. Then we have $\mu_1 <\lambda_1$.
\end{lem}
\begin{proof} Since $\lambda$ is strict and each part is odd, we have $(M-1)^2<2n = |\lambda| \leq \frac{(\lambda_1+1)^2}{4}$, thus $2M-2 <\lambda_1+1$. As $\lambda_1$ is odd, we have $2M\leq \lambda_1+1$. Now by Lemma \ref{lem:estim} we have
\begin{enumerate}
\item if $\mu_1$ is even, then $\mu_1 \leq 2m_{\mu}<2M\leq \lambda_1+1,$ and
\item if $\mu_1$ is odd, then $\mu_1 \leq 2m_{\mu}-1<2M-1\leq\lambda_1.$
\end{enumerate}
In the first case, as $\mu_1$ is even we also have $\mu_1 < \lambda_1$. Thus the result follows. 
\end{proof}
\begin{lem} \label{lem:bound2}For $n \geq 2$, $M^2\geq M_C^2>n$.
\end{lem}
\begin{proof} Easy.
\end{proof}
Now we prove $\ch\atsp{\lambda} = \ague{\lambda}$. First set $\mathscr{P} \subset W(C_n)$ to be the union of conjugacy classes in $W(C_n)$ parametrized by $(\emptyset, \rho)$ for some $\rho \vdash n$. Also we let $\mathscr{P}_e \colonequals \mathscr{P} \cap W$. Then $\mathscr{P}_e$ is the union of conjugacy classes in $W$ parametrized by $(\emptyset, \rho)$ for some $\rho \vdash n$ such that $l(\rho)$ is even. We denote by $\mathcal{R}(\mathscr{P}) \subset \mathcal{R}(W(C_n))$ (resp. $\mathcal{R}(\mathscr{P}_e)\subset \mathcal{R}(W)$) the set of class functions of $W(C_n)$ (resp. $W$) supported on $\mathscr{P}$ (resp. $\mathscr{P}_e$).

If we follow the argument in the previous section, we need to show that
$$\{\ch \chi^\mu|_{\mathscr{P}_e} \mid \mu \textup{ not very even}, \mu_1 < \lambda_1\} \cup \{\ch \chi^{\mu\pm}|_{\mathscr{P}_e} \mid \mu \textup{ very even}, \mu_1 < \lambda_1\}$$
spans $\mathcal{R}(\mathscr{P}_e)$. But we already know that $\{\ch \chi^\mu_C |_{\mathscr{P}} \mid m_\mu < M_C\}$ spans $\mathcal{R}(\mathscr{P})$ by Lemma \ref{lem:spanBC}. By taking the restriction to $\mathscr{P}_e$, we see that
$$\{\ch \chi^\mu |_{\mathscr{P}_e} \mid \mu \textup{ not very even}, m_\mu < M_C\}\cup \{\ch (\chi^{\mu+}\oplus\chi^{\mu-}) |_{\mathscr{P}_e} \mid \mu \textup{ very even}, m_\mu < M_C\}$$
spans $\mathcal{R}(\mathscr{P}_e)$. Now the result easily follows from Lemma \ref{lem:bound1} and \ref{lem:bound2}.

\subsection{Very even case} If $\lambda$ is not very even, then the identity $\ch\atsp{\lambda} = \ague{\lambda}$ directly implies Main Theorem \ref{mainthm1} as $\atsp{\lambda}= \tsp{\lambda}\oplus\tsp{\lambda}$. Even when $\lambda$ is very even, it proves Main Theorem \ref{mainthm1} modulo the multiplicities of $\chi^{\mu\pm}$ for $\mu$ very even. Here, we fill this gap and finish the proof.

Suppose $n$ is even and $\lambda,\mu \vdash n/2$ are given. We define 
$$\diff{\lambda} \colonequals \ch \tsp{(2\lambda\cup2\lambda)+}-\ch\tsp{(2\lambda\cup2\lambda)-}.$$
Then we claim the following. Note that it suffices to complete the proof of Main Theorem \ref{mainthm1}.
\begin{prop} $\br{\diff{\lambda}, \ch\chi^{(2\mu\cup2\mu)+}}=-\br{\diff{\lambda},\ch\chi^{(2\mu\cup2\mu)-}}=K_{\mu, \lambda}$, where $K_{\mu, \lambda}=K_{\mu, \lambda}(1)$ is the Kostka number.
\end{prop}
\begin{proof}
Set $m_{\mu,\lambda}\colonequals \br{\diff{\lambda}, \ch\chi^{(2\mu\cup2\mu)+}}=- \br{\diff{\lambda}, \ch\chi^{(2\mu\cup2\mu)-}}.$ Since $\br{\diff{\lambda}, \ch\chi^\alpha}=0$ if $\alpha$ is not very even, we have
$$\br{\diff{\lambda},\diff{\mu}} = 2\sum_{\nu\vdash n/2} m_{\nu,\lambda}m_{\nu,\mu}.$$
Assuming the proposition, it should be the same as $2\sum_{\nu\vdash n/2} K_{\nu,\lambda}K_{\nu,\mu}=2\br{h_\lambda,h_\mu}$. As the matrix $\{m_{\mu,\lambda}\}$ is upper-unitriangular (which comes from the upper-triangularity property of Springer representations),  the proposition is in fact equivalent to the identity $\br{\diff{\lambda},\diff{\mu}}=2\br{h_\lambda,h_\mu}$ for any $\lambda, \mu\vdash n/2$.

By Theorem \ref{thm:mainD}, we have
\begin{align*}
\br{\diff{\lambda},\diff{\mu}} &= 2\sum_{\rho\vdash n/2} \frac{1}{z_\rho 2^{2l(\rho)}} \gr^{2\lambda\cup2\lambda}_{4\rho}(-1)\gr^{2\mu\cup2\mu}_{4\rho}(-1)=2\sum_{\rho\vdash n/2} \frac{1}{z_\rho } \br{p_{2\rho}, h_{2\lambda}}\br{p_{2\rho}, h_{2\mu}}
\end{align*}
where the last equality is deduced from \cite[Theorem 3.4]{llt94}. By adjunction of the Frobenius and Verschiebung operators it is equal to
\begin{align*}
2\sum_{\rho\vdash n/2} \frac{1}{z_\rho } \br{p_{\rho}[p_2], h_{2\lambda}}\br{p_{\rho}[p_2], h_{2\mu}}&=2\sum_{\rho\vdash n/2} \frac{1}{z_\rho } \br{p_{\rho}, \versch(h_{2\lambda})}\br{p_{\rho}, \versch(h_{2\mu})}\\&=2\sum_{\rho\vdash n/2} \frac{1}{z_\rho } \br{p_{\rho}, h_{\lambda}}\br{p_{\rho}, h_{\mu}},
\end{align*}
which is clearly the same as $2\br{h_{\lambda}, h_{\mu}}$. Thus the result follows.
\end{proof}

\bibliographystyle{amsalphacopy}
\bibliography{kostka}

\end{document}